\documentclass[reqno,11pt]{amsart}
\usepackage{amssymb, amsmath,latexsym,amsfonts,amsbsy, amsthm}
\usepackage{color}
\usepackage{mathrsfs}
\usepackage{esint}
\usepackage{graphics,color}
\usepackage{enumerate}
\usepackage{marginnote}

\newcommand{\beq}{\begin{equation}}
\newcommand{\eeq}{\end{equation}}
\newcommand{\ben}{\begin{eqnarray}}
\newcommand{\een}{\end{eqnarray}}
\newcommand{\beno}{\begin{eqnarray*}}
\newcommand{\eeno}{\end{eqnarray*}}

\renewcommand{\theequation}{\thesection.\arabic{equation}}

\newtheorem{theorem}{Theorem}[section]

\newtheorem{lemma}[theorem]{Lemma}
\newtheorem{proposition}[theorem]{Proposition}

\newtheorem{Theorem}{Theorem}[section]
\newtheorem{Definition}[Theorem]{Definition}

\newtheorem{Lemma}[Theorem]{Lemma}
\newtheorem{Corollary}[Theorem]{Corollary}
\newtheorem{Remark}[Theorem]{Remark}

\begin{document}

\title[Boussinesq Equation]
{H\"{o}lder continuous periodic Solution Of Boussinesq Equation with partial viscosity}

\author{Tao Tao}
\address{School of Mathematics Sciences, Shandong University, Jinan, China}
\email{taotao@amss.ac.cn}
\author{Liqun Zhang}
\address{Academy of Mathematic and System Science , CAS Beijing 100190, China; and School of Mathematical Sciences, University of Chinese Academy of Sciences, Beijing 100049, China}
\email{lqzhang@math.ac.cn}

\date{\today}
\maketitle

\renewcommand{\theequation}{\thesection.\arabic{equation}}
\setcounter{equation}{0}

\begin{abstract}
We show the existence of H\"{o}lder continuous periodic solution with compact support in time of the Boussinesq equations with partial viscosity. The H\"{o}lder regularity of the solution we constructed is anisotropic which is compatible with partial viscosity of the equations.
\end{abstract}


\vskip 0.2cm


\setcounter{equation}{0}

\section{Introduction}
The Boussinesq model was introduced for understanding the effect of potentially large conversions between internal energy and mechanical energy in fluids. The Boussinesq equation modeled many geophysical flows, such as atmospheric fronts and ocean circulations (see, for example, \cite{Ma},\cite{Pe}). It was used in recent theoretical discussion of the energetics of horizontal convection as well as in discussing the energetics of turbulent mixing in stratified fluids.  Such a model forms the basis for a majority of numerical simulation of stratified turbulence.

In this paper, we consider the following 3-dimensional Boussinesq system with only vertical viscosity
\begin{equation}\label{e:partial boussinesq equation}
\begin{cases}
\partial_tv+v\cdot\nabla v+\nabla p-\partial_{zz}v=\theta e_{3}, \quad {\rm in} \quad R_+\times T^3\\
\hbox{div}v=0,\quad {\rm in} \quad R_+\times T^3\\
\partial_t\theta+v\cdot\nabla\theta-\partial_{zz}\theta=0, \quad {\rm in} \quad R_+\times T^3
\end{cases}
\end{equation}
 on torus $T^3$, where $e_3=(0,0,1)^T$, $z$ is the vertical variable. Here, $v$ is the velocity vector, $p$ is the pressure, $\theta$ denotes the temperature which is a scalar function.

The global existence and well-posedness have been established by many authors for the Cauchy problem of (\ref{e:partial boussinesq equation}) in 2d (see, for example, \cite{CW}, \cite{LT}). For the 3-dimensional case, the global existence of  smooth solution of (\ref{e:partial boussinesq equation}) remains open. We are interested in constructing some continuous solutions of (\ref{e:partial boussinesq equation}).
To understand the turbulence phenomena in hydrodynamics, one needs to go beyond classical solutions. The triple $(v, p, \theta)$ on $R_{+}\times T^3$ is called a weak solution of (\ref{e:partial boussinesq equation}) if they belong to $L^2_{loc}(R_+\times T^3)$ and solve (\ref{e:partial boussinesq equation})  in the following sense:
\begin{align}
\int_{R_+}\int_{T^3}(v\cdot\partial_t\varphi+v\otimes v:\nabla\varphi+p\hbox{div}\varphi+v\cdot\partial_{zz}\varphi+\theta e_3\cdot\varphi )dxdt=0\nonumber
\end{align}
for all $\varphi\in C_c^\infty(R_{+}\times T^3;R^3).$
\begin{align}
\int_{R_+}\int_{T^3}(\partial_t\phi\theta+v\cdot\nabla\phi\theta+\partial_{zz}\phi\theta)dxdt=0\nonumber
\end{align}
for all $\phi\in C_c^\infty(R_{+}\times T^3;R)$ and
\begin{align}
\int_{R_+}\int_{T^3}v\cdot\nabla\psi dxdt=0\nonumber
\end{align}
for all $\psi\in C_c^\infty(R_{+}\times T^3;R).$ Here and in the following arguments $x=(x_1,x_2,z)$. Sometimes we abuse the notation by using $(x,y,z)$ to denote the space variables.

The study of continuous weak solutions in fluid dynamics becomes more and more interesting. One of the famous example is the anomalous dissipation on weak solution of Euler equation which was first considered by Lars Onsager in his famous 1949 note about statistical hydrodynamics, see \cite{ONS}.

The existence of dissipative solution has been studied by many authors. In the earlier work of V. Scheffer (\cite{VS}) and later by A. Shnirelman (\cite{ASH2}), they constructed weak solutions of Euler equations with compact support in time. For weak solutions with decreasing energy, there are some results by A. Shnirelman (\cite{ASH1}) and Camillo De Lellis, L\'{a}szl\'{o} Sz\'{e}kelyhidi (\cite{CDL,CDL0}).

Recently, a great progress was made by Camillo De Lellis, L\'{a}szl\'{o} Sz\'{e}kelyhidi, etc, in the construction of H\"{o}lder continuous solution. They developed an iterative scheme (some kind of convex integration) in \cite{CDL2} and constructed continuous periodic solution on $T^3$ which satisfies the prescribed kinetic energy by using Beltrami flow on $T^3$ and Geometric Lemma. The solution is a superposition of infinitely many perturbed and weakly interacting Beltrami flows. Combining the Nash-Moser mollify technique in the iterative scheme, they also constructed H\"{o}lder periodic solution with exponent $\theta$, for any $\theta<\frac{1}{10}$, which satisfies the prescribed kinetic energy in \cite{CDL3}.
In \cite{BCDL1}, Camillo De Lellis, L\'{a}szl\'{o} Sz\'{e}kelyhidi and T. Buckmaster constructed H\"{o}lder continuous weak solutions with any $\theta<\frac{1}{5}$, which satisfies the prescribed kinetic energy, also see \cite{BCDLI}. In whole space $R^3$, P. Isett and Sung-jin Oh in \cite{ISOH2} constructed H\"{o}lder continuous solutions  with  any $\theta<\frac{1}{5}$, which satisfies the prescribed kinetic energy or is a perturbation of smooth Euler flow. Moreover, S. Daneri considered the Cauchy problem for dissipative H\"{o}lder Euler flow in (\cite{DA,DAL}) and A. Choffrut studied h-principles for the incompressible Euler equations in \cite{CHO}.

For the Onsager critical spatial regularity (H\"{o}lder exponent $\theta=\frac{1}{3}$), there are also some progress. By keeping track of sharper, time localized estimates in \cite{TBU}, T. Buckmaster constructed H\"{o}lder continuous (with exponent $\theta<\frac{1}{5}-\varepsilon$ in time-space) periodic solutions which for almost every time belongs to $C_x^{\theta}$, for any $\theta<\frac{1}{3}$, and is compactly temporal supported. In \cite{BCDL2}, Camillo De Lellis, L\'{a}szl\'{o} Sz\'{e}kelyhidi and T. Buckmaster also constructed  H\"{o}lder continuous periodic solution which belongs to $L^1_tC_x^\theta$, for any $\theta<\frac{1}{3}$, and has compact support in time. Recently, P. Isett gives a proof of Onsager's conjecture in (\cite{IS2}), and C. De Lellis, L. Sz\'{e}kelyhidi, T. Buckmaster and V.Vicol give another short proof in \cite{BCDLV} for admissible weak solution. Furthermore, Buckmaster and Vicol establish the nonuniqueness of weak solution to the incompressible Navier-Stokes in \cite{BV} by introducing some new ideas.

Moreover, Vicol and Isett constructed H\"{o}lder continuous weak solution for some class of active scalar equations in \cite{ISV2},  Luo and Xin constructed H\"{o}lder continuous solutions with compact support in time for 3-dimensional Prandtl's system  in \cite{LX}.

Motivated by the above earlier works, we considered the Boussinesq equations and want to know if the similar phenomena can also happen when add the temperature effects. In the present paper, we consider the Boussinesq equations with vertical viscosity (or diffusion) on torus and show that vertical viscosity can't obstruct the anomalous dissipative phenomena in 3-dimensional Boussinesq equations. Following the general scheme in the construction of Euler equations and inspired by \cite{ISV2}, \cite{CCDE}, \cite{LX}  and \cite{NASH}, by establishing the corresponding geometric lemma and improving the iteration scheme, we obtain the following existence result.

We denote $v\in C^\alpha$ if
$$\sup_{(t,x,y,z)\neq(t',x',y',z')}\frac{|v(t,x,y,z)-v(t',x',y',z')|}{|(t,x,y,z)-(t',x',y',z')|^\alpha}<\infty,$$
 and denote $v(t,x,y,z)\in C_z^{\alpha}$ if
$$\sup_{t,x,y}\sup_{z\neq z'}\frac{|v(t,x,y,z)-v(t,x,y,z')|}{|z-z'|^\alpha}<\infty.$$

\begin{theorem}\label{t:main 1}
For any $\varepsilon> 0$,  there exists a triple
\begin{align}
    (v,p,\theta)\in C_c(R_+\times T^3)\nonumber
\end{align}
 such that they solve the system \eqref{e:partial boussinesq equation} in the sense of distribution with compact support in time
and
$$v\in C^{\frac{1}{55}-\varepsilon},\quad v\in C_z^{\frac{1}{19}-\varepsilon}\quad \theta\in  C^{\frac{1}{45}-\varepsilon},\quad \theta\in  C_z^{\frac{1}{15}-\varepsilon},\quad p\in C^\infty.$$
\end{theorem}

\begin{Remark}
For the anisotropic Boussinesq equations
\begin{equation}
\begin{cases}
\partial_tv+v\cdot\nabla v+\nabla p-\partial_{xx}v=\theta e_{3}, \quad {\rm in} \quad R_+\times T^3,\\
{\rm div}v=0,\quad {\rm in} \quad R_+\times T^3,\\
\partial_t\theta+v\cdot\nabla\theta-\partial_{xx}\theta=0, \quad {\rm in} \quad R_+\times T^3 \nonumber
\end{cases}
\end{equation}
and
\begin{equation}
\begin{cases}
\partial_tv+v\cdot\nabla v+\nabla p-\partial_{yy}v=\theta e_{3}, \quad {\rm in} \quad R_+\times T^3,\\
{\rm div}v=0,\quad {\rm in} \quad R_+\times T^3,\\
\partial_t\theta+v\cdot\nabla\theta-\partial_{yy}\theta=0, \quad {\rm in} \quad R_+\times T^3,\nonumber
\end{cases}
\end{equation}
by a similar argument, we also can construct H\"{o}lder continuous weak solution with compact support in time.
\end{Remark}

\begin{Remark}
One can construct more regularity solution which does not depend on $z$. But here we are interested in constructing H\"{o}lder continuous weak solution depending on $z$.
\end{Remark}

\begin{Remark}
In the iterative scheme, we don't change the pressure, see Proposition \ref{p: iterative 1}, hence the pressure in our theorem is smooth.
\end{Remark}

\begin{Remark}
The H\"{o}lder exponents in our theorem is very small, the main reason is because we use a multiple-step iterative scheme which is similar to Nash's isometric embedding of manifold manifold. Moveover, if we use the technique of \cite{BCDLI} to improve our construction, then we can improve the estimate for transport term(see Lemma \ref{e:transport estimate} and Lemma \ref{e:trans}) and obtain more better regularity. But, if we want to improve the regularity largely, then we should find a good building block for Boussinesq equation in the iterative scheme, like Beltrami flow or Mikado flow for the Euler equation.
\end{Remark}

\setcounter{equation}{0}

\section{Setup and Plan of the paper}
As in \cite{CDL2}, the proof of theorem \ref{t:main 1} will be achieved through an iteration procedure. In what follows $\mathcal{S}^{3\times3}$ denotes the vector space of symmetric $3\times3$ matrices.
\begin{Definition}
Assume that $v,p,\theta,R,f$ are smooth functions on $R_+\times T^3$ taking values,~respectively,~in $R^{3},R,R,\mathcal{S}^{3\times3},R^{3}$.
We say that they solve the anisotropic Boussinesq-Stress system  if
\begin{equation}\label{d:anistropic boussinesq reynold}
\begin{cases}
\partial_tv+v\cdot\nabla v+\nabla p-\partial_{zz}v=\theta e_{3}+{\rm div}R, \quad {\rm in} \quad R_+\times T^3,\\
{\rm div}v=0,\quad {\rm in} \quad R_+\times T^3,\\
\partial_t\theta+v\cdot\nabla\theta-\partial_{zz}\theta={\rm div}f, \quad {\rm in} \quad R_+\times T^3.
\end{cases}
\end{equation}
\end{Definition}

\subsection{Some notation on norm}
In the following, $ m=0,1,2,...$ and $\beta$ is a multi index.
We denote the norm $\|f\|_0$ by
\begin{align}
 \|f\|_{0}:=\hbox{sup}_{T^{3}}|f|.\nonumber
 \end{align}
 Then, define the semi-norm
 \begin{align}
 [f]_{m}:=\hbox{max}_{|\beta|=m}\|\nabla^{\beta}f\|_{0}\nonumber
\end{align}
and norm
\begin{align}
\|f\|_{m}:=\sum_{j=0}^{m}[f]_{j}.\nonumber
\end{align}
If $f=f_1+if_2$ is a complex-valued function, then we set $\|f\|_m:=\|f_1\|_m+\|f_2\|_m$.

Moreover, for function depending on space and time, we introduce the following space-time norm:
\begin{align}
\|f\|_m:=\sup_t\|f(t,\cdot)\|_m,\quad \|f\|_{C^1}:=\|f\|_1+\|\partial_tf\|_0,\quad \|f\|_{C_z^1}=\|f\|_0+\|\partial_zf\|_0.\nonumber
\end{align}
We now state the main proposition of the paper,~of which Theorem \ref{t:main 1} is a corollary.
\begin{proposition}\label{p: iterative 1}
Let $1<a<b<4$ be any two numbers and $\varepsilon>0$ be any small constant. Then there exists a absolute positive constant $M$ such that the following properties hold:

For any $0<\kappa\leq 1$,~if $(v,~ p, ~\theta, ~R, ~f)\in C_c^{\infty}((a,b)\times T^3)$ solves Boussinesq-Stress system (\ref{d:anistropic boussinesq reynold}) and
\begin{align}
     \|R\|_0\leq& \kappa,\label{e:reynold initial}\\ \|f\|_0\leq& \kappa,\label{e:reynold initial 2}
\end{align}
set
\begin{align}\label{d:difinition on c1 norm}
\Lambda:=\max\{1, \|R\|_{C^1}, \|f\|_{C^1}, \|v\|_{C^1}, \|\theta\|_{C^1}\},\quad
\bar{\Lambda}:=\max\{1, \|R\|_{C_z^1}, \|f\|_{C_z^1}, \|v\|_{C_z^1}, \|\theta\|_{C_z^1}\},
\end{align}
then for any $\bar{\kappa}\leq\kappa^{\frac{3}{2}}$, we can construct new functions $(\tilde{v},~\tilde{p},~\tilde{\theta},~\tilde{R},~\tilde{f})\in C_c^{\infty}((a-\kappa,b+\kappa)\times T^3),$  which also solves Boussinesq-Stress system (\ref{d:anistropic boussinesq reynold}) and satisfies
\begin{align}
    \|\tilde{R}\|_0\leq&\bar{\kappa} ,\label{e:next step strees estimate}\\
     \|\tilde{f}\|_0\leq& \bar{\kappa},\\
    \|\tilde{v}-v\|_0 \leq& M\sqrt{\kappa},\\
   \|\tilde{\theta}-\theta\|_0 \leq& M\sqrt{\kappa},\\
    \tilde{p}=p,\label{e:next step pressure estimate}
    \end{align}
    and
\begin{align}\label{b:bound on first derivative}
\Lambda_1:=&\max\{1,\|\tilde{R}\|_{C^1}, \|\tilde{f}\|_{C^1}, \|\tilde{v}\|_{C^1}, \|\tilde{\theta}\|_{C^1}\}\nonumber\\
\leq& A\Big(\kappa^{\frac{6+11\varepsilon}{2}}
 \Big(\frac{\kappa}{\bar{\kappa}^2}\Big)^{6+16\varepsilon}\Lambda^{(1+\varepsilon)^6}
 +\frac{\kappa^{5.5+11\varepsilon}\bar{\Lambda}^{2(1+\varepsilon)^5}}
{\bar{\kappa}^{6+16\varepsilon}}\Big(\frac{\sqrt{\kappa}}{\bar{\kappa}}\Big)^{7+21\varepsilon}\Big),\nonumber\\
\bar{\Lambda}_1:=&\max\{1,\|\tilde{R}\|_{C_z^1}, \|\tilde{f}\|_{C_z^1}, \|\tilde{v}\|_{C_z^1}, \|\tilde{\theta}\|_{C_z^1}\}
\leq A\Big(\frac{\kappa}{\bar{\kappa}}\Big)^6\bar{\Lambda}.
\end{align}
More precisely, we have
\begin{align}\label{e:growth estimate for temperture}
\|\partial_z\tilde{\theta}\|_0\leq &A\Big(\frac{\kappa}{\bar{\kappa}}\Big)^3\bar{\Lambda},\nonumber\\
\|(\partial_t,\partial_x,\partial_y)\tilde{\theta}\|_0
\leq&A\Big(\kappa^{\frac{3+\varepsilon}{2}}
 \Big(\frac{\kappa}{\bar{\kappa}^2}\Big)^{3+4\varepsilon}\Lambda^{(1+\varepsilon)^3}
 +\frac{\kappa^{\frac{5}{2}+\varepsilon}\bar{\Lambda}^{2(1+\varepsilon)^2}}
{\bar{\kappa}^{3+4\varepsilon}}\Big(\frac{\sqrt{\kappa}}{\bar{\kappa}}\Big)^{4+6\varepsilon}\Big),
\end{align}
where the constant $A$ depends on $\varepsilon, \|v\|_0.$
\end{proposition}

We will prove Proposition \ref{p: iterative 1} in the subsequent sections. Now, we show how to obtain Theorem \ref{t:main 1} from this proposition.
\begin{proof}[Proof of theorem 1.1]
We first set
\begin{align}
&v_{0}:=10M\chi(t)
\left(\begin{array}{cc}
\cos(\lambda y)\\
0\\
0
\end{array}
\right),\quad \theta_0:=10M\chi(t)\cos(\lambda y), \quad p_{0}:=10M\chi'(t)\frac{\sin(\lambda z)}{\lambda},\nonumber\\
&R_{0}:=
10M\left(
\begin{array}{ccc}
0  &  \chi'(t)\frac{\sin(\lambda y)}{\lambda}  &  0 \\
\chi'(t)\frac{\sin(\lambda y)}{\lambda}  & 0 & - \chi(t)\frac{\sin(\lambda y)}{\lambda}\\
0  & - \chi(t)\frac{\sin(\lambda y)}{\lambda} & \chi'(t)\frac{\sin(\lambda z)}{\lambda}
\end{array}
\right),\quad
f_{0}:=10M\chi'(t)\left(
\begin{array}{cc}
0\\
 \frac{\sin(\lambda y)}{\lambda}\\
  0
\end{array}
\right),\nonumber
\end{align}
 where $\chi(t)\in C_c^{\infty}(1,4)$, $\chi=1~~~{\rm in}~~~(2,3)$ and $0\leq\chi\leq1$.\\
Obviously, they solve Boussinesq-Stress system (\ref{d:anistropic boussinesq reynold}).
We take $a,b \geq \frac{3}{2}$ and set $\kappa_n:=a^{-b^n},\quad n=0,1,2,\cdots$. Then taking $\lambda$ sufficiently large such that
\begin{align}
\|R_{0}\|_0\leq \kappa_0,~~
\|f_{0}\|_0\leq \kappa_0.\nonumber
\end{align}
Then, by using Proposition \ref{p: iterative 1} iteratively, we can construct $$(v_{n},~p_{n},~\theta_{n},~R_{n},~f_{n})\in
C_c^\infty((1-\sum_{i=0}^n\kappa_i, 4+\sum_{i=0}^n\kappa_i)\times T^3), \quad n=1,2,\cdots$$
such that they solve system (\ref{d:anistropic boussinesq reynold}) and satisfy the following estimates
\begin{eqnarray}
    &\|R_{n}\|_0\leq \kappa_n ,\label{e:stress_final}\\
    &\|f_{n}\|_0\leq  \kappa_n,\\
    &\|v_{n+1}-v_{n}\|_0 \leq M\sqrt{\kappa_n},\label{e:velocity_final}\\
    &\|\theta_{n+1}-\theta_{n}\|_0 \leq M\sqrt{\kappa_n},\label{e:temperature_final}\\
    &p_{n+1}=p_{n},\label{e:pressure_final}\\
    &\Lambda_{n+1}:=\max\{1, \|R_{n+1}\|_{C^1}, \|f_{n+1}\|_{C^1}, \|v_{n+1}\|_{C^1}, \|\theta_{n+1}\|_{C^1}\}\nonumber\\
    &\quad \leq  A\Big(\kappa_n^{\frac{6+11\varepsilon}{2}}
 \big(\frac{\kappa_n}{\kappa^2_{n+1}}\big)^{6+16\varepsilon}\Lambda_n^{(1+\varepsilon)^6}
 +\frac{\kappa_n^{5.5+11\varepsilon}\bar{\Lambda}_n^{2(1+\varepsilon)^5}}
{\kappa_{n+1}^{6+16\varepsilon}}\big(\frac{\sqrt{\kappa_n}}{\kappa_{n+1}}\big)^{7+21\varepsilon}\Big),\nonumber\\
&\bar{\Lambda}_{n+1}:=\max\{1,\|R_{n+1}\|_{C_z^1}, \|f_{n+1}\|_{C_z^1}, \|v_{n+1}\|_{C_z^1}, \|\theta_{n+1}\|_{C_z^1}\}
\leq A\big(\frac{\kappa_n}{\kappa_{n+1}}\big)^6\bar{\Lambda}_n.
\end{eqnarray}
More precisely, we have
\begin{align}\label{e:n sequence estimate growth}
\|\partial_z\theta_{n+1}\|_0\leq& A\Big(\frac{\kappa_n}{\kappa_{n+1}}\Big)^3\bar{\Lambda}_n,\nonumber\\
\|(\partial_t,\partial_x,\partial_y)\theta_{n+1}\|_0\leq&
A\Big(\kappa_n^{\frac{3+\varepsilon}{2}}
 \big(\frac{\kappa_n}{\kappa_{n+1}^2}\big)^{3+4\varepsilon}\Lambda_n^{(1+\varepsilon)^3}
 +\frac{\kappa_n^{\frac{5}{2}+\varepsilon}\bar{\Lambda}_n^{2(1+\varepsilon)^2}}
{\kappa_{n+1}^{3+4\varepsilon}}\big(\frac{\sqrt{\kappa_{n}}}{\kappa_{n+1}}\big)^{4+6\varepsilon}\Big).
\end{align}
It's obvious that $\sum_{i=0}^\infty\kappa_i<\frac{1}{2}$. By (\ref{e:stress_final})-(\ref{e:pressure_final}), we know that $(v_{n},~p_{n},~\theta_{n},~R_{n},~f_{n})$ are Cauchy sequence in $C_c((0,5)\times T^3)$, therefore there exist
\[\begin{aligned}
    (v,p,\theta)\in C_c((0,5)\times T^3)
\end{aligned}\]
such that
\[\begin{aligned}
    v_{n}\rightarrow v,\quad p_{n}\rightarrow p,\quad \theta_{n}\rightarrow \theta,
    \quad R_{n}\rightarrow 0,\quad f_{n}\rightarrow 0
\end{aligned}\]
in $C_c((0,5)\times T^3)$.

By (\ref{e:velocity_final}),
$$\|v-v_0\|_0\leq M\sum\limits_{n=0}^{\infty}\sqrt{\kappa_n}<4M,$$
therefore $$\|v\|_0\geq 6M,~~v\neq0.$$
Similarly, by (\ref{e:temperature_final}),
$$\theta\neq 0.$$
Passing into the limit in (\ref{d:anistropic boussinesq reynold}), we conclude that $v,~p,~\theta$ solve (\ref{d:anistropic boussinesq reynold}) in the sense of distribution.

Next, we prove that the solution $v , \theta$ is H\"{o}lder continuous. We claim that for a suitable choice of $a, b$, there exist constants $c, \bar{c}> 1$ such that
$$\Lambda_n\leq a^{cb^n},\quad \bar{\Lambda}_n\leq a^{\bar{c}b^n}.$$
We prove this claim by induction.\\
Indeed, for $n=0$, it's obvious if we take $a\geq \Lambda_0:=\max\{1, \|R_0\|_1, \|f_0\|_1, \|v_0\|_1, \|\theta_0\|_1\}$. Assuming that we have proved $\bar{\Lambda}_n\leq a^{\bar{c}b^n}$, then
\begin{align}
\bar{\Lambda}_{n+1}\leq A\Big(\frac{\kappa_n}{\kappa_{n+1}}\Big)^6\bar{\Lambda}_n
\leq Aa^{-\frac{\varepsilon}{2}}a^{\bar{c}b^{n+1}}a^{\big(6(b-1)+\frac{\varepsilon}{2}+\bar{c}-\bar{c}b\big)b^n}.\nonumber
\end{align}
We impose $\varepsilon<\frac{1}{2000}$ and set
$$b=\frac{3}{2},\quad \bar{c}=6+\varepsilon,$$
then we have
$$\bar{\Lambda}_{n+1}\leq  Aa^{-\frac{\varepsilon}{2}}a^{\bar{c}b^{n+1}}.$$
Then by choosing $a\geq A^{\frac{2}{\varepsilon}}$, we have $\bar{\Lambda}_{n+1}\leq a^{\bar{c}b^{n+1}}.$ Finally, we take $a:=\max\{A^{\frac{2}{\varepsilon}}, \Lambda_0 \}$, then the constant $a$ satisfies all conditions.
Moreover, assuming that we have proved $\Lambda_n\leq a^{cb^n}$, then
\begin{align}
\Lambda_{n+1}\leq&  A\Big(\kappa_n^{\frac{6+11\varepsilon}{2}}
 \Big(\frac{\kappa_n}{\kappa^2_{n+1}}\Big)^{6+16\varepsilon}\Lambda_n^{(1+\varepsilon)^6}
 +\frac{\kappa_n^{5.5+11\varepsilon}\bar{\Lambda}_n^{2(1+\varepsilon)^5}}
{\kappa_{n+1}^{6+16\varepsilon}}\Big(\frac{\sqrt{\kappa_n}}{\kappa_{n+1}}\Big)^{7+21\varepsilon}\Big),\nonumber\\
\leq&Aa^{-\frac{\varepsilon}{2}}a^{cb^{n+1}}\Big(a^{\big(9+27\varepsilon+c(1+\varepsilon)^6-\frac{3}{2}c\big)b^n}
+a^{\big(22.5+100\varepsilon-\frac{3}{2}c\big)b^n}\Big).\nonumber
\end{align}
Set
$$c=\frac{18+54\varepsilon}{1-14\varepsilon},$$
then we have
$$\Lambda_{n+1}\leq  Aa^{-\frac{\varepsilon}{2}}a^{cb^{n+1}}\leq a^{cb^{n+1}}.$$
Now we consider the approximate sequence $v_n, \theta_n$. By (\ref{e:velocity_final}), we have
\begin{align}
\|v_{n+1}-v_n\|_0\leq Ma^{-\frac{1}{2}b^n}.\nonumber
\end{align}
Moreover, we have
\begin{align}
\|v_{n+1}-v_n\|_{C^1}\leq \Lambda_n+\Lambda_{n+1}\leq 2a^{cb^{n+1}},\quad
\|\partial_z(v_{n+1}-v_n)\|_0\leq \bar{\Lambda}_n+\bar{\Lambda}_{n+1}\leq2a^{\bar{c}b^{n+1}}.\nonumber
\end{align}
Therefore, for any $\alpha, \alpha'\in (0,1)$,
\begin{align}
\|v_{n+1}-v_n\|_{C^\alpha}\leq  2Ma^{\big(\alpha cb-\frac{(1-\alpha)}{2}\big)b^n},\quad
\|v_{n+1}-v_n\|_{C_z^{\alpha'}}\leq  2Ma^{\big(\alpha' \bar{c}b-\frac{(1-\alpha')}{2}\big)b^n}.
\end{align}
If $\alpha<\frac{1}{1+2bc}$, then $\alpha cb-\frac{(1-\alpha)}{2}< 0$, thus $v_n$ are Cauchy sequence in $C^\alpha.$ Take the value of $b,c$, we know that $v\in C^\alpha$ for any $\alpha<\frac{1-14\varepsilon}{55+148\varepsilon}$. When $\varepsilon\rightarrow 0$, we have $\frac{1-14\varepsilon}{55+148\varepsilon}\rightarrow \frac{1}{55}$.\\
If $\alpha'<\frac{1}{1+2b\bar{c}}$, then $v_n$ are Cauchy sequence in $C^{\alpha'}_{z}.$ Thus, we know that $v\in C^{\alpha'}_z$ for any $\alpha'<\frac{1}{19+3\varepsilon}$.\\
By (\ref{e:pressure_final}), we know that $p\in C^\infty$.
By (\ref{e:temperature_final}), we have
\begin{align}
\|\theta_{n+1}-\theta_n\|_0\leq Ma^{-\frac{1}{2}b^n}.\nonumber
\end{align}
Moreover, by (\ref{e:n sequence estimate growth}), we have
\begin{align}
\|\partial_z\theta_{n+1}\|_0\leq  a^{(7.5+\varepsilon)b^n},\quad
\|(\partial_t,\partial_x,\partial_y)\theta_{n+1}\|_0\leq
a^{\big(4.5+8\varepsilon+c(1+\varepsilon)^3\big)b^n}.\nonumber
\end{align}
By interpolation, for any $\gamma, \gamma'\in (0,1)$, we have
\begin{align}
\|\theta_{n+1}-\theta_n\|_{C_z^{\gamma'}}\leq 2Ma^{\big(\gamma'(7.5+\varepsilon)-\frac{(1-\gamma')}{2}\big)b^n},\quad
\|\theta_{n+1}-\theta_n\|_{C^{\gamma}}\leq 2Ma^{\big(\gamma\big(4.5+8\varepsilon+c(1+\varepsilon)^3\big)-\frac{(1-\gamma)}{2}\big)b^n}.\nonumber
\end{align}
Take $\gamma'<\frac{1}{15+2\varepsilon}$ and $\gamma<\frac{1}{9+16\varepsilon+2c(1+\varepsilon)^3}$, then $\theta_n$ converge in $C^{\gamma'}_z$ and $C^{\gamma}$, which implies that $\theta\in C^{\gamma}$ and $\theta\in C_z^{\gamma'}$. When $\varepsilon\rightarrow 0$, we have $\frac{1}{9+16\varepsilon+2c(1+\varepsilon)^3}\rightarrow \frac{1}{45}$.
Thus, we complete our proof for the theorem.
\end{proof}

\subsection{Outline of the construction}

\indent

The rest of the paper will be dedicated to prove Proposition \ref{p: iterative 1}. The construction of the functions $\tilde{v}, \tilde{\theta}$ consists of several steps. In the first step, we adding perturbations to $v_0, \theta_0$  and get new functions $v_{01}, \theta_{01}$ as following:
 \begin{align}
 v_{01}=v_0+w_{1o}+w_{1c}:=v_0+w_1,\quad
 \theta_{01}=\theta_0+\chi_{1o}+\chi_{1c}:=\theta_0+\chi_1,\nonumber
 \end{align}
where $w_{1o}, w_{1c}, \chi_{1o}, \chi_{1c}$ are highly oscillatory functions and have explicit formula. Having added the perturbation, we will focus on finding functions $R_{01}, p_{01}$ and $f_{01}$ with the desired estimates which solve system (\ref{d:anistropic boussinesq reynold}). The main perturbation $w_{1o}$
and $\chi_{1o}$ will depend on four parameters, $\ell_1,~~\ell_{1z},~~\mu$ and $\lambda_{1}$, which will satisfy some additional conditions.

After the first step, the stresses $(R_{01}, f_{01})$ become smaller in the following sense: if
\begin{align}
e(t)\sum\limits_{i=1}^{6}k_i\otimes k_i-R_0=\sum\limits_{i=1}^{6}a_i^2k_i\otimes k_i,\quad
f_0=\sum\limits_{i=1}^{3}b_ik_i,\nonumber
\end{align}
where $k_i$ is defined in (\ref{e:six representation vector}), then
\begin{align}
R_{01}=\sum\limits_{i=2}^{6}a_i^2k_i\otimes k_i+\delta R_{01},\quad
f_{01}=\sum\limits_{i=2}^{3}b_ik_i+\delta f_{01},\nonumber
\end{align}
where $\delta R_{01}, \delta f_0$ can be arbitrary small through the appropriate choice on $\ell_1,~~\ell_{1z},~~\mu,~~\lambda_1$.

We repeat the above process, after six steps, we can obtain the desired $(\tilde{v},~\tilde{p},~\tilde{\theta},~\tilde{R},~\tilde{f}).$\\

The rest of paper is organized as follows: in section 3, we give a kind of decomposition of symmetric matrix and introduce two operators which are extensions of \cite{CDL2}. After these preliminaries, we perform the first step in sections 4, 5, 6. In section 4, we not only define the perturbation $w_{1o}, w_{1c}, \chi_{1o}, \chi_{1c}$ and new stress $R_{01}, f_{01}$, but also prescribed the constant $M$ of the estimates in Proposition 2.1. In sections 5 and 6, we will
calculate the main forms of $R_{01}, f_{01}$ and prove the relevant estimates of the various terms involved in the construction, in  term of the parameters
$\lambda_1, \mu, \ell_1, \ell_{1z}$ separately. After completing the first step, in sections 7, 8 and 9 we will construct $(v_{0n}, p_{0n}, \theta_{0n}, R_{0n}, f_{0n})$ and prove relevant estimates
by induction. First, we give construction in section 7 which is similar to the first step in section 4. After completing the construction, we calculate the main form $R_{0n}, f_{0n}$ in section 8 and prove the various error estimate in section 9 separately. Those two sections are also similar to section 5 and section 6 separately. Finally, in section 10, we will give a proof of Proposition 2.1 by choosing appropriate parameters $\mu,\lambda_n, \ell_n, \ell_{nz}$ for $1\leq n\leq 6$.

\setcounter{equation}{0}

\section{Preliminaries}

\subsection{Decomposition of symmetric matrix and vector}

\indent

Let
\begin{align}\label{e:six representation vector}
&k_1=(1,0,0)^T,\quad k_2=(0,1,0)^T,\quad k_3=(1,0,1)^T,\nonumber\\
&k_4=(1,1,0)^T,\quad k_5=(0,1,1)^T,\quad k_6=(1,1,1)^T,
\end{align}
then we have
\begin{align}
k_1\otimes k_1=\left(\begin{array}{ccc}
1 & 0 & 0\\
0 & 0 & 0\\
0 & 0 & 0
\end{array}\right),\quad
k_2\otimes k_2=
\left(\begin{array}{ccc}
0 & 0 & 0\\
0 & 1 & 0\\
0 & 0 & 0
\end{array}\right),\quad
k_3\otimes k_3=
\left(\begin{array}{ccc}
1 & 0 & 1\\
0 & 0 & 0\\
1 & 0 & 1
\end{array}\right),\nonumber\\
k_4\otimes k_4=
\left(\begin{array}{ccc}
1 & 1 & 0\\
1 & 1 & 0\\
0 & 0 & 0
\end{array}\right),\quad
k_5\otimes k_5=
\left(\begin{array}{ccc}
0 & 0 & 0\\
0 & 1 & 1\\
0 & 1 & 1
\end{array}\right),\quad
k_6\otimes k_6=
\left(\begin{array}{ccc}
1 & 1 & 1\\
1 & 1 & 1\\
1 & 1 & 1
\end{array}\right).\nonumber
\end{align}
Obviously, the above matrices form a basic of the space $\mathcal{S}^{3\times3}$ , hence
we have the following unique decomposition: there exist $\gamma_i\in C^\infty(\mathcal{S}^{3\times3})$ such that for any $R\in \mathcal{S}^{3\times3}$,
\begin{align}\label{d:decomposition 1}
R=\sum_{i=1}^6 \gamma_i(R)k_i\otimes k_i.
\end{align}
In fact, if we denote a symmetric matrix
\begin{align}
R=
\left(\begin{array}{ccc}
R_{11} & R_{12} & R_{13}\\
R_{12} & R_{22} & R_{23}\\
R_{13} & R_{23} & R_{33}
\end{array}\right)\nonumber
\end{align}
then we take
\beno
&&\gamma_1(R)=R_{11}-R_{12}-R_{33}+R_{23},\quad \gamma_2(R)=R_{22}-R_{33}-R_{12}+R_{13},\quad \gamma_3(R)=R_{33}-R_{23},\\
&&\gamma_4(R)=R_{12}-R_{13}+R_{33}-R_{23},\quad \gamma_5(R)=R_{33}-R_{13},\quad \gamma_6(R)=R_{13}-R_{33}+R_{23}.
\eeno
A straightforward computation give that $\gamma_i(R)$ satisfies (\ref{d:decomposition 1}). Moreover, it's obvious that they are smooth function.

Furthermore, we also have the following unique decomposition of 3-d vectors: there exist linear functions $b_i\in C^\infty(R^3)$ such that for any $f\in R^3$,
\begin{align}\label{d:decomposition 2}
f=\sum_{i=1}^3 b_i(f)k_i.
\end{align}
In fact, we take
\beno
(b_1(f), b_2(f), b_3(f))=(2f\cdot k_1-f\cdot k_3, f\cdot k_2, f\cdot k_3-f\cdot k_1).
\eeno
Obviously, $b_i$ is smooth, linear function and satisfies (\ref{d:decomposition 2}).

\subsection{The operator $\mathcal{R}$ and $\mathcal{G}$}

\indent

We define two operators
in order to deal with the stress error. The operator $\mathcal{R}$ was introduced in  \cite{IS1} and the operator $\mathcal{G}$ is given by us.
\subsubsection{The operator $\mathcal{R}$ }
The following lemma is taken from \cite{IS1}, we copy it here for the completeness of the paper (the proof refers to \cite{IS1}).
\begin{lemma}[$\mathcal{R}=\textrm{div}^{-1}$]\label{l:reyn}
There exists a linear operator $\mathcal{R}$ from $C^\infty (T^3, R^3)$ to $C^\infty (T^3, \mathcal{S}^{3\times3})$ such that the following property holds: for any $v\in C^\infty (T^3, R^3)$ we have
\begin{itemize}
\item[(1)] $\mathcal{R}v(x)$ is a symmetric matrix for each $x\in T^3$;
\item[(2)] ${\rm div}\mathcal{R} v = v-\fint_{T^3}v$.
\item[(3)] Let $F(x)=a(x)e^{i\lambda k\cdot x}$ be a smooth vector field on $T^3$ with $\lambda \in {\rm N}, k\in R^3\setminus\{0\}$ and $\int_{T^3}F(x)dx=0$. Then for any integer $m\geq 1$, we have the estimate
    \begin{align}\label{e:estimate oscillatory inverse}
    \|\mathcal{R}(F)\|_0\leq C_m\Big(\sum_{i=0}^{m-1}\frac{\|a\|_i}{\lambda^{i+1}}+\frac{\|a\|_m}{\lambda^{m}}\Big).
    \end{align}
\end{itemize}
\end{lemma}
Here and subsequent, we use the natation $\fint_{T^3}v(x)dx=\frac{1}{|T^3|}\int_{T^3}v(x)dx$.

\subsubsection{The operator $\mathcal{G}$ }

\indent

Let $f(x)=\varphi(x)e^{i\lambda k\cdot x}$ with $\varphi(x)\in C^{\infty}(T^3;C)$ and $\int_{T^3}f(x)dx=0$. Set
$$\mathcal{G}f_o(x):=\frac{k}{i\lambda|k|^2}\varphi(x)e^{i\lambda k\cdot x},$$
then
$${\rm div}\mathcal{G}f_o(x)=f(x)+\frac{k\cdot\nabla \varphi(x)}{i\lambda|k|^2}e^{i\lambda k\cdot x}.$$
Set
$$\mathcal{G}f_{c1}(x):=-\frac{k}{(i\lambda)^2|k|^4}k\cdot\nabla \varphi(x) e^{i\lambda k\cdot x},$$
then
$${\rm div}(\mathcal{G}f_o+\mathcal{G}f_{c1})(x)=f(x)-\frac{(k\cdot\nabla)^2 \varphi(x)}{(i\lambda)^2|k|^4}e^{i\lambda k\cdot x}.$$
It's obvious that
$$\|\mathcal{G}f_o\|_0\leq C_0\frac{\|\varphi\|_0}{\lambda},\quad\|\mathcal{G}f_{c1}\|_0\leq C_0\frac{\|\varphi\|_1}{\lambda^2}.$$
Performing this process, for any integer $m\geq 2$, there exist $\mathcal{G}f_{ci}:i=1,2,\cdot\cdot\cdot,m-1$ such that
$${\rm div}\Big(\mathcal{G}f_o+\sum_{i=1}^{m-1}\mathcal{G}f_{ci}\Big)(x)=f(x)+(-1)^{m-1}\frac{(k\cdot\nabla)^m \varphi(x)}{(i\lambda)^m|k|^{2m}}e^{i\lambda k\cdot x},\quad\|\mathcal{G}f_{ci}\|_0\leq C_0\frac{\|\varphi\|_i}{\lambda^{i+1}}.$$
Since
\begin{align}
\int_{T^3}f(x)dx=0, \quad \int_{T^3}{\rm div}\Big(\mathcal{G}f_o+\sum_{i=0}^{m-1}\mathcal{G}f_{ci}\Big)(x)dx=0,\nonumber
\end{align}
therefore
\begin{align}
\int_{T^3}\frac{(k\cdot\nabla)^m\varphi(x)}{(i\lambda|k|^2)^m}e^{i\lambda k\cdot x}dx=0.\nonumber
\end{align}
Thus, there exists $\mathcal{G}f_{cm}\in C^{\infty}(T^3; C^3)$ such that
$${\rm div}\mathcal{G}f_{cm}(x)=(-1)^{m}\frac{(k\cdot\nabla)^m \varphi(x)}{(i\lambda)^m|k|^{2m}}e^{i\lambda k\cdot x},\quad
\|\mathcal{G}f_{cm}\|_0\leq C_m\frac{\|\nabla^m\varphi\|_0}{\lambda^m}.$$
In fact, there exists $\mathcal{G}f_{cm}\in C^{\infty}(T^3)$ such that
\begin{align}
\int_{T^3}\mathcal{G}f_{cm}(x)dx=0,\quad\|\nabla\mathcal{G}f_{cm}\|_4\leq C_0\Big\|\frac{(k\cdot\nabla)^m \varphi}{(i\lambda)^m|k|^{2m}}\Big\|_4\leq C_m\frac{\|\nabla^m\varphi\|_0}{\lambda^m}.\nonumber
\end{align}
Thus, we have $\|\mathcal{G}f_{cm}\|_{W^{1,4}}\leq C_0\frac{\|\nabla^m\varphi\|_0}{\lambda^m}$. In particular, we have
$\|\mathcal{G}f_{cm}\|_0\leq C_m\frac{\|\nabla^m\varphi\|_0}{\lambda^m}$ which is what we claimed.

Finally, we set
$$\mathcal{G}f:=\mathcal{G}f_{o}+\sum_{i=1}^m\mathcal{G}f_{ci},$$
then
\begin{align}
{\rm div}\mathcal{G}f
=f,\quad
\mathcal{G}f\in C^{\infty}(T^3;C^3),\quad \|\mathcal{G}f\|_0\leq C_m\Big(\sum_{i=0}^{m-1}\frac{\|\varphi\|_i}{\lambda^{i+1}}+\frac{\|\varphi\|_m}{\lambda^{m}}\Big).\nonumber
 \end{align}
In conclusion, we have
 \begin{lemma}\label{p:inverse 2}
 Let the vector space $\Psi$ given by
 \begin{align}
 \Psi:=\Big\{H(x): H(x)=\sum_{j=0}^nH_j(x):=\sum_{j=0}^nb_j(x)e^{i\lambda_j k\cdot x},~~~ b_j\in C^{\infty}(T^3;C)\quad {\rm and}\quad \int_{T^3}H_j(x)dx=0\Big\},\nonumber
 \end{align}
 then there exists a linear operator $\mathcal{G}: \Psi\rightarrow C^{\infty}(T^3;C^3)$ such that for any positive integer $m$ and any $H(x)=\sum\limits_{j=0}^nb_j(x)e^{i\lambda_j k\cdot x}\in\Psi$,
 \begin{align}\label{i:identity and inequality}
 {\rm div}\mathcal{G}(H)(x)=H(x),\quad \|\mathcal{G}(H)(x)\|_0\leq C_m\sum_{j=0}^n\Big(\sum_{i=0}^{m-1} \frac{\|b_j\|_i}{\lambda_j^{i+1}}+\frac{\|b_i\|_m}{\lambda_j^m}\Big).
  \end{align}
\end{lemma}

\begin{proof}
We have defined the operator $\mathcal{G}$ on the functions
$$H_j(x)=b_j(x)e^{i\lambda_j k\cdot x}$$
and
$$\|\mathcal{G}(H_j)\|_0\leq C_m\Big(\sum_{i=0}^{m-1}\frac{\|b_j\|_i}{\lambda_j^{i+1}}
 +\frac{\|b_j\|_m}{\lambda^m}\Big).$$
Then, set
\begin{align}
\mathcal{G}(H)(x):=\sum_{j=0}^n\mathcal{G}(H_j)(x).\nonumber
\end{align}
It is obvious that $\mathcal{G}$ is a linear operator on $\Psi$ and
$\mathcal{G}H(x)$ satisfies (\ref{i:identity and inequality}).
\end{proof}

\setcounter{equation}{0}

\section{the construction of $v_{01}, p_{01}, \theta_{01}, R_{01}, f_{01} $}
The construction of $\tilde{v},~\tilde{p},~\tilde{\theta},~\tilde{R},~\tilde{f}$ from  $v,~p,~\theta,~R,~f$ consists of many steps. The main idea is to decompose the stress error into many blocks by (\ref{d:decomposition 1}) and (\ref{d:decomposition 2}), then we remove one block in each step. In this section, we perform the first step.

For convenience, we set $v_0:=v,~p_0:=p,~\theta_0:=\theta,~R_0:=R,~f_0:=f$.

\subsection{Construction of 1-th perturbation $w_1$ on velocity}

\indent
\subsubsection{Conditions on the parameters}
 Our construction depend on four parameters $\mu, \lambda_1, \ell_1, \ell_{1z}$ and we always assume that they satisfy the following inequalities:
  \begin{align}\label{a:assumption on parameter}
  \mu\geq \frac{1}{\kappa},\quad \ell_1^{-1}\geq\mu\Lambda,\quad \ell_{1z}^{-1}\geq\mu\bar{\Lambda}, \quad \ell_{1z}\geq\ell_1, \quad \lambda_1\geq \ell_1^{-(1+\varepsilon)},\quad
  \mu, \lambda_1, \frac{\lambda_1}{\mu}\in {\rm N} .
\end{align}

\subsubsection{Partition of unity and decomposition of stress error}
We first introduce a partition of unity. From \cite{CHO}, we have the following partition of unity: for two constants $c_{1}$ and $c_{2}$ such that $\frac{\sqrt{3}}{2}<c_{1}<c_{2}<1$ , we have a family of functions $\alpha_l\in C_c^{\infty}(R^3):l\in Z^3$ such that
\begin{align}\label{p:unity}
\sum\limits_{l\in Z^3}\alpha_l^2=1,\qquad \hbox{supp}\alpha_l\subseteq B_{c_2}(l).
\end{align}
Next, let $\varphi\in C_c^\infty(R)$ be a standard nonnegative function and we denote the corresponding family of mollifiers by
\begin{align}
\varphi_{\ell}(x):=\frac{1}{\ell}\varphi\Big(\frac{x}{\ell}\Big)\nonumber
\end{align}
and set
\begin{align}\label{m:anistropic modification}
f_{\ell_1}(t,x,y,z):=f\ast(\varphi_{\ell_1}(t)\varphi_{\ell_1}(x)\varphi_{\ell_1}(y)\varphi_{\ell_{1z}}(z))
\end{align}
for any $f\in C_c^{\infty}((a,b)\times T^3)$.\\
By (\ref{d:decomposition 1}), for symmetric matrix $R_{0}\in C_c^{\infty}((a,b)\times T^3)$, there holds
\begin{align}
R_0=\sum_{i=1}^6 \gamma_i(R_0)k_i\otimes k_i.\nonumber
\end{align}
We set $a_i:=\gamma_i(R_0) \in C_c^{\infty}((a,b)\times T^3)$,  thus we have
\begin{align}\label{d:decomposition of R0}
R_0=\sum_{i=1}^6 a_ik_i\otimes k_i.
\end{align}
Similarly, by (\ref{d:decomposition 2}), for vector $f_{0}\in C_c^{\infty}((a,b)\times T^3)$, there exist functions $c_i:= b_i(f)\in C_c^{\infty}((a,b)\times T^3):i=1,2,3$ such that
\begin{align}\label{d:decomposition of f0}
f_{0}=\sum_{i=1}^3 c_ik_i.
\end{align}
Thus,
\begin{align}\label{d:decomposition of strees error}
R_{0\ell_1}=\sum_{i=1}^6 a_{i\ell_1}k_i\otimes k_i,\quad f_{0\ell_1}=\sum_{i=1}^3 c_{i\ell_1}k_i.
\end{align}
By (\ref{e:reynold initial}) and (\ref{e:reynold initial 2}), we have
\begin{align}\label{b:bound on decomposition coeffience}
\|a_{i\ell_1}\|_0\leq 5\kappa,\quad \|c_{i\ell_1}\|_0\leq 2\kappa.
\end{align}
Then, we denote $e(t)\in C_c^{\infty}(a-\kappa,b+\kappa)$ by
\begin{align}\label{d:definition e}
e(t)=\left\{
\begin{array}{ll}
10\kappa,\qquad  t\in[a-\frac{\kappa}{2},b+\frac{\kappa}{2}],\\
0,\qquad \quad t\in(a-\kappa,b+\kappa)^c
\end{array}
\right.
\end{align}
and $0\leq e(t)\leq 10\kappa$.
\subsubsection{Construction of $1$-th main perturbation}
For any $l\in Z^3$, we set
   \begin{align}\label{d:l amp}
    b_{1l}:=&\sqrt{\frac{e(t)-a_{1\ell_1}}{2}}\alpha_l(\mu v_{0\ell_1}).
    \end{align}
By (\ref{a:assumption on parameter}) and (\ref{b:bound on decomposition coeffience}), (\ref{d:definition e}), we know that $b_{1l}$ is well-defined.
Then, as in \cite{ISV2}, set $[l]:=\sum_{j=0}^22^j[l_j]$ if $l=(l_0, l_1, l_2)$, where
    \begin{align}
    [l_j]=\left\{
    \begin{array}{cc}
    1, ~~~~l_j~~~{\rm is ~~~even},\\
    0,~~~~ l_j~~~{\rm is~~~ odd}.
    \end{array}
    \right.\nonumber
    \end{align}
Thus, $[l]$ can only take values in $\{0,1,2,...,7\}$.

We denote main $l$-perturbation $w_{1ol}$ by
   \begin{align}\label{d:w 1ol}
    w_{1ol}:=b_{1l}k_1\Big(e^{i\lambda_1 2^{[l]} (k_{1h}^{\perp},0)\cdot \big((x,y,z)-\frac{l}{\mu}t\big)}+e^{-i\lambda_1 2^{[l]} (k_{1h}^{\perp},0)\cdot \big((x,y,z)-\frac{l}{\mu}t\big)}\Big)
   \end{align}
and 1-th main perturbation $w_{1o}$ by
   \begin{align}\label{d:1o}
    w_{1o}:=\sum_{l\in Z^3}w_{1ol},
   \end{align}
where $k_1$ is taken from (\ref{e:six representation vector}), $k_{1h}$ is the horizontal component of $k_1$ and $a^{\perp}=(-a_2,a_1)^T$ if $a=(a_1,a_2)^T$. \\
Obviously, $w_{1ol},w_{1o}$ are all real 3-dimensional vector-valued functions.
   By (\ref{p:unity}), $\hbox{supp} \alpha_l\cap \hbox{supp}\alpha_{l'}=\emptyset$  if $|l-l'|\geq2$, therefore there is only finite term in the above summation.

 Furthermore, by (\ref{b:bound on decomposition coeffience}) and (\ref{d:definition e}), we can take  an absolute constant $M$ such that
$\|b_{1l}\|_0\leq \frac{M\sqrt{\kappa}}{300}.$
Thus, by (\ref{p:unity}), we have
  \begin{align}\label{e:bound w1o}
  \|w_{1o}\|_0\leq \frac{M\sqrt{\kappa}}{12}.
  \end{align}

 \subsubsection{Construction of correction $w_{1c}$}

We assume that $k_1=(k_{11}, k_{12}, k_{13}),\quad (k_{11}, k_{12})\neq 0$ and taking $(s_1,t_1)\in R^2$ such that $s_1k_{11}+t_1k_{12}=-k_{13}$. Then, denote $l$-th correction $w_{1cl}$ by
\begin{align}\label{d:w 1cl}
w_{1cl}:=\left(
    \begin{array}{ccc}
    -t_1\partial_zb_{1l}+\partial_yb_{1l}\\
    s_1\partial_zb_{1l}-\partial_xb_{1l}\\
    t_1\partial_xb_{1l}-s_1\partial_yb_{1l}
    \end{array}
    \right)
    \Bigg(\frac{e^{i\lambda_1 2^{[l]} (k_{1h}^{\perp},0)\cdot \big((x,y,z)-\frac{l}{\mu}t\big)}}{i\lambda_1 2^{[l]}}+
    \frac{e^{-i\lambda_1 2^{[l]} (k_{1h}^{\perp},0)\cdot \big((x,y,z)-\frac{l}{\mu}t\big)}}{-i\lambda_1 2^{[l]}}\Bigg)
\end{align}
and 1-th correction $w_{1c}$ by
   \begin{align}\label{d:w 1c}
    w_{1c}:=\sum_{l\in Z^3}w_{1cl}.
    \end{align}
A straightforward computation gives
\begin{align}
w_{1ol}+w_{1cl}={\rm curl}\Big(b_{1l}\vec{a}_1\Big(\frac{e^{i\lambda_1 2^{[l]} (k_{1h}^{\perp},0)\cdot \big((x,y,z)-\frac{l}{\mu}t\big)}}{i\lambda_1 2^{[l]}}+\frac{e^{-i\lambda_1 2^{[l]} (k_{1h}^{\perp},0)\cdot \big((x,y,z)-\frac{l}{\mu}t\big)}}{-i\lambda_12^{[l]}}\Big)\Big),\nonumber
\end{align}
where $\vec{a}_1=(s_1,t_1,1).$\\

Finally, set 1-th perturbation
   \begin{align}\label{d:w 1}
    w_1:=w_{1o}+w_{1c}.
   \end{align}
  Thus, if we denote $ w_{1l}$ by
   \begin{align}\label{d:w 1l}
   w_{1l}:=w_{1ol}+w_{1cl},
   \end{align}
   then
  \begin{align}
   w_1=\sum\limits_{l\in Z^3}w_{1l},\qquad {\rm div}w_{1l}=0,\qquad \hbox{div}w_1=0.\nonumber
   \end{align}
Moreover, if we set
\begin{align}\label{d:difination b1l}
g_{1l}:=b_{1l}k_1+\frac{1}{i\lambda_12^{[l]}}\left(
    \begin{array}{ccc}
    -t_1\partial_zb_{1l}+\partial_yb_{1l}\\
    s_1\partial_zb_{1l}-\partial_xb_{1l}\\
    t_1\partial_xb_{1l}-s_1\partial_yb_{1l}
    \end{array}
    \right),
g_{-1l}:=b_{1l}k_1+\frac{1}{-i\lambda_12^{[l]}}\left(
    \begin{array}{ccc}
    -t_1\partial_zb_{1l}+\partial_yb_{1l}\\
    s_1\partial_zb_{1l}-\partial_xb_{1l}\\
    t_1\partial_xb_{1l}-s_1\partial_yb_{1l}
    \end{array}
    \right),
\end{align}
then
\begin{align}\label{d:another representation}
w_{1l}=g_{1l}e^{i\lambda_1 2^{[l]} k_1^{\perp}\cdot \big((y,z)-\frac{l}{\mu}t\big)}+g_{-1l}e^{-i\lambda_1 2^{[l]} k_1^{\perp}\cdot \big((y,z)-\frac{l}{\mu}t\big)}.
\end{align}
It's obvious that $b_{1l}\in C_c^{\infty}((a-\kappa,b+\kappa)\times T^3)$, therefore
   $$w_{1ol},w_{1cl},w_{1l},w_1\in C_c^{\infty}((a-\kappa,b+\kappa)\times T^3).$$
Thus we complete the construction of perturbation $w_1$ on velocity.
\subsection{Construction of 1-th perturbation $\chi_1$ on temperature}

\indent

To construct $\chi_1$, we first
   denote $\beta_{1l}$ by
   \begin{align}\label{d:b1l}
    \beta_{1l}:=-\frac{c_{1\ell_1}}{\sqrt{2(e(t)-a_{1\ell_1})}}\alpha_l(\mu v_{0\ell_1}).
    \end{align}
    Since ${\rm supp}c_{1\ell_1}\subseteq (a-\ell_1,b+\ell_1)\times T^3$ and $e(t)-a_{1\ell_1}(t,x,y,z)\geq\kappa \quad {\rm in}\quad (a-\ell_1,b+\ell_1)\times T^3$, so $\beta_{1l}$ is well-defined.
    Then denote main $l$-perturbation $\chi_{1ol}$ by
    \begin{align}
    \chi_{1ol}:=\beta_{1l}\Big(e^{i\lambda_1 2^{[l]} (k_{1h}^{\perp},0)\cdot \big((x,y,z)-\frac{l}{\mu}t\big)}+e^{-i\lambda_1 2^{[l]} (k_{1h}^{\perp},0)\cdot \big((x,y,z)-\frac{l}{\mu}t\big)}\Big),\nonumber
    \end{align}
  $l$-correction $\chi_{1cl}$ by
    \begin{align}
    \chi_{1cl}:=\frac{\nabla\beta_{1l}\cdot (k_{1h}^{\perp},0)}{i\lambda_12^{[l]}|k_{1h}|^2}\Big(e^{i\lambda_1 2^{[l]} (k_{1h}^{\perp},0)\cdot \big((x,y,z)-\frac{l}{\mu}t\big)}-e^{-i\lambda_1 2^{[l]} (k_{1h}^{\perp},0)\cdot \big((x,y,z)-\frac{l}{\mu}t\big)}\Big)\nonumber
    \end{align}
    and $l$-perturbation $\chi_{1l}$ by
     \begin{align}
    \chi_{1l}:=\chi_{1ol}+\chi_{1cl}
    ={\rm div}\Big(\frac{\beta_{1l}(k_{1h}^{\perp},0)}{i\lambda_12^{[l]}|k_{1h}|^2}\Big(e^{i\lambda_1 2^{[l]} (k_{1h}^{\perp},0)\cdot \big((x,y,z)-\frac{l}{\mu}t\big)}-e^{-i\lambda_1 2^{[l]} (k_{1h}^{\perp},0)\cdot \big((x,y,z)-\frac{l}{\mu}t\big)}\Big)\Big).\nonumber
   \end{align}
   Finally, set
    \begin{align}
    \chi_{1o}:=\sum_{l\in Z^3}\chi_{1ol},\quad \chi_{1c}:=\sum_{l\in Z^3}\chi_{1cl}\nonumber
   \end{align}
    and $1$-th perturbation $\chi_1$
   \begin{align}
    \chi_1:=\sum_{l\in Z^3}\chi_{1l}.\nonumber
   \end{align}
   Thus, $\chi_{1l}$ and $\chi_1$ are all real scalar functions, and similar to the perturbation $w_1$, there are only finite terms in the summation of $\chi_1$.
Furthermore, if we set
\begin{align}\label{d:difinition on h1l}
h_{1l}=\beta_{1l}+\frac{\nabla\beta_{1l}\cdot (k_{1h}^{\perp},0)}{i\lambda_12^{[l]}|k_{1h}|^2},\quad
h_{-1l}=\beta_{1l}+\frac{\nabla\beta_{1l}\cdot (k_{1h}^{\perp},0)}{-i\lambda_12^{[l]}|k_{1h}|^2},
\end{align}
then
\begin{align}
&\chi_{1l}=h_{1l}e^{i\lambda_1 2^{[l]} (k_{1h}^{\perp},0)\cdot \big((x,y,z)-\frac{l}{\mu}t\big)}+h_{-1l}e^{-i\lambda_1 2^{[l]} (k_{1h}^{\perp},0)\cdot \big((x,y,z)-\frac{l}{\mu}t\big)},\nonumber\\
&\chi_{1}=\sum_{l\in Z^3}\Big(h_{1l}e^{i\lambda_1 2^{[l]} (k_{1h}^{\perp},0)\cdot \big((x,y,z)-\frac{l}{\mu}t\big)}+h_{-1l}e^{-i\lambda_1 2^{[l]} (k_{1h}^{\perp},0)\cdot \big((x,y,z)-\frac{l}{\mu}t\big)}\Big).\nonumber
\end{align}
Finally, by (\ref{b:bound on decomposition coeffience}), (\ref{d:definition e}) and (\ref{d:b1l}), after possibly taking a bigger number $M$ , we know that
 $ \|\beta_{1l}\|_0\leq \frac{M\sqrt{\kappa}}{300}.$
Thus, we obtain
  \begin{align}\label{e:bound x}
 \|\chi_1\|_0\leq \frac{M\sqrt{\kappa}}{6}.
  \end{align}
\subsection{The construction of  $v_{01}$,~$p_{01}$,~$\theta_{01}$,~$R_{01}$,~$f_{01}$ }

\indent

First, we denote
 $M_1, N_1,K_1$ by
   \begin{align}\label{d:difinition on m1}
   M_1:=&\sum\limits_{l\in Z^3}b^2_{1l}k_1\otimes k_1\Big(e^{2i\lambda_1 2^{[l]} (k_{1h}^{\perp},0)\cdot \big((x,y,z)-\frac{l}{\mu}t\big)}+
   e^{-2i\lambda_1 2^{[l]} (k_{1h}^{\perp},0)\cdot \big((x,y,z)-\frac{l}{\mu}t\big)}\Big)\nonumber\\
   &+\sum\limits_{l,l'\in Z^3 ,l\neq l'}w_{1ol}\otimes w_{1ol'},\nonumber\\
   N_1:=&\sum_{l\in Z^3}\Big[w_{1l}\otimes \Big(v_{0\ell_1}-\frac{l}{\mu}\Big)
   +\Big(v_{0\ell_1}-\frac{l}{\mu}\Big)\otimes w_{1l}\Big]+\sum_{l\in Z^3}\Big[w_{1l}\otimes \big(v_0-v_{0\ell_1}\big)
   +\big(v_0-v_{0\ell_1}\big)\otimes w_{1l}\Big],\nonumber\\
   K_1:=&\sum\limits_{l\in Z^3}\beta_{1l}b_{1l}k_1\Big(e^{2i\lambda_1 2^{[l]} (k_{1h}^\perp,0)\cdot \big((x,y,z)-\frac{l}{\mu}t\big)}+e^{-2i\lambda_1 2^{[l]} (k_{1h}^\perp,0)\cdot \big((x,y,z)-\frac{l}{\mu}t\big)}\Big{)}+\sum\limits_{l,l'\in Z^3 ,l\neq l'}w_{1ol}\chi_{1ol'}.
    \end{align}
 Notice that $N_{1}$ is a symmetric matrix.  Then we set
  \begin{align}\label{d:the first solution sequence}
    &v_{01}:=v_0+w_1,\quad
    p_{01}:=p_0,\quad
    \theta_{01}:=\theta_0+\chi_1,\quad f_{01}:=f_{0\ell_1}+2\sum\limits_{l\in Z^3}\beta_{1l}b_{1l}k_1+\delta f_{01},\nonumber\\
    & R_{01}:=-\Big(e(t)\sum_{i=1}^6k_i\otimes k_i-R_{0\ell_1}\Big)+2\sum\limits_{l\in Z^3}b^2_{1l}k_1\otimes k_1+\delta R_{01},\quad
   \end{align}
   where
\begin{align}\label{d:R0small}
   \delta R_{01}=&\mathcal{R}(\hbox{div}M_1)+N_1+\mathcal{R}\Big\{\partial_tw_{1}
  +\hbox{div}\Big[\sum_{l\in Z^3}\Big(w_{1l}\otimes\frac{l}{\mu}+\frac{l}{\mu}\otimes w_{1l}\Big)\Big]\Big\}-\mathcal{R}(\partial_{zz}w_1)\nonumber\\
   &+(w_{1o}\otimes w_{1c}+w_{1c}\otimes w_{1o}+w_{1c}\otimes w_{1c})
   -\mathcal{R}(\chi_1e_3)+R_0-R_{0\ell_1}
   \end{align}
and
   \begin{align}\label{d:f0small}
   \delta f_{01}=&\mathcal{G}(\hbox{div}K_1)
   +\mathcal{G}(w_1\cdot\nabla\theta_{0\ell_1})
   +\mathcal{G}\Big(\partial_t\chi_{1}+\sum_{l\in Z^3}\frac{l}{\mu}\cdot\nabla\chi_{1l}\Big)-\mathcal{G}(\partial_{zz}\chi_1)
   +\sum_{l\in Z^3}\Big(v_{0\ell_1}-\frac{l}{\mu}\Big)\chi_{1l}\nonumber\\
   &+\sum_{l\in Z^3}\big(v_0-v_{0\ell_1}\big)\chi_{1l}+w_{1c}\chi_1+w_{1o}\chi_{1c}+f_0-f_{0\ell_1}
   +w_1(\theta_0-\theta_{0\ell_1}).
  \end{align}
 By Lemma \ref{l:reyn}, we know that $\delta R_{01}$ is a symmetric matrix. Obviously,
   \begin{align*}
    \hbox{div}v_{01}=\hbox{div}v_0+\hbox{div}w_1=0.
   \end{align*}
   Moreover, by the definition of $R_{01},\delta R_{01}$ as well as $v_{01},p_{01},\theta_{01}$ and notice that $v_{0}, p_{0},
    \theta_0, R_{0}, f_{0}$  are solutions of the system (\ref{d:anistropic boussinesq reynold}),
together with Lemma \ref{l:reyn} , we know that
\[\begin{aligned}
    \hbox{div}R_{01}=&\hbox{div}R_0+\partial_tw_1
    -\chi_1e_3-\partial_{zz}w_1\nonumber\\
    &+\hbox{div}(w_{1o}\otimes w_{1o}+w_{1}\otimes v_{0}+v_0\otimes w_{1}
    +w_{1o}\otimes w_{1c}+w_{1c}\otimes w_{1o}+w_{1c}\otimes w_{1c})\nonumber\\
    =&\partial_tv_{0}+\hbox{div}(v_{0}\otimes v_{0})+\nabla p_{0}-\partial_{zz}v_0-\theta_0e_3+\partial_tw_1
    -\chi_1e_3-\partial_{zz}w_1\nonumber\\
    &+\hbox{div}(w_{1o}\otimes w_{1o}+w_{1}\otimes v_{0}+v_0\otimes w_1
    +w_{1o}\otimes w_{1c}
    +w_{1c}\otimes w_{1o}+w_{1c}\otimes w_{1c})\nonumber\\
    =&\partial_tv_{01}+\hbox{div}(v_{01}\otimes v_{01})+\nabla p_{01}-\partial_{zz}v_{01}-\theta_{01}e_3,
    \end{aligned}\]
    where we used
    \begin{align}
    &\fint_{T^3}w_{1}(t,x,y,z)dxdydz=0, \quad\fint_{T^3}\chi_{1}(t,x,y,z)dxdydz=0,\nonumber\\
    &\hbox{div}(M_1)+\hbox{div}\Big(2\sum\limits_{l\in Z^3}b^2_{1l}k_1\otimes k_1\Big)=\hbox{div}(w_{1o}\otimes w_{1o}).\nonumber
    \end{align}
Furthermore, from the definition of $f_{01}, \delta f_{01}$ as well as $v_{01}, \theta_{01}$ and notice that $v_{0}, p_{0},
    \theta_0, R_{0}, f_{0}$  are solutions of the system (\ref{d:anistropic boussinesq reynold}),
together with Lemma \ref{p:inverse 2}, we know that
 \begin{align}
 \hbox{div}f_{01}=&\hbox{div}f_0+\partial_{t}\chi_1-\partial_{zz}\chi_1
 +\hbox{div}(w_{1o}\chi_1+w_{1c} \chi_1+v_0 \chi_1+w_1 \theta_0)\nonumber\\
 =&\hbox{div}(v_0\theta_0+w_{1o}\chi_1+w_{1c} \chi_1+v_0 \chi_1+w_1 \theta_0)
 +\partial_t(\theta_0+\chi_1)-\partial_{zz}(\theta_0+\chi_1)\nonumber\\
 =&\partial_t\theta_{01}+\hbox{div}(v_{01}\theta_{01})-\partial_{zz}\theta_{01},\nonumber
 \end{align}
 where we used
    $$ \fint_{T^3}\chi_{1}(t,x,y,z)dxdydz=0,\quad
    \hbox{div}(K_1)+\hbox{div}\Big(2\sum\limits_{l\in Z^3}\beta_{1l}b_{1l}k_1\Big)=\hbox{div}(w_{1o}\chi_{1o}).$$
 Thus, the functions $(v_{01},p_{01},\theta_{01},R_{01},f_{01})$ solve the system (\ref{d:anistropic boussinesq reynold}).

\setcounter{equation}{0}

\section{The representations}

\indent

In this section, we will compute the forms of
$$I:=-\Big(e(t)\sum_{i=1}^6k_i\otimes k_i-R_{0\ell_1}\Big)+2\sum\limits_{l\in Z^3}b^2_{1l}k_1\otimes k_1,\quad
II:=f_{0\ell_1}+2\sum\limits_{l\in Z^3}\beta_{1l}b_{1l}k_1.$$

\subsection{The form of I}

\indent

First, by the definition (\ref{d:l amp}) on $b_{1l}$ , we have
\begin{align}
2\sum\limits_{l\in Z^3}b^2_{1l}k_1\otimes k_1
=\sum\limits_{l\in Z^3}\alpha_l^2(\mu v_{0\ell_1}) (e(t)-a_{1\ell_1})k_1\otimes k_1
=(e(t)-a_{1\ell_1})k_1\otimes k_1,\nonumber
\end{align}
where we use $\sum\limits_{l\in Z^3}\alpha_l^2=1$.\\
Moreover, by the identity (\ref{d:decomposition 1}), we have
\begin{align}
-\Big(e(t)\sum_{i=1}^6k_i\otimes k_i-R_{0\ell_1}\Big)=-\sum\limits_{i=1}^6(e(t)- a_{i\ell_1})k_i\otimes k_i.\nonumber
\end{align}
Therefore
\begin{align}
-\Big(e(t)\sum_{i=1}^6k_i\otimes k_i-R_{0\ell_1}\Big)+2\sum\limits_{l\in Z^3}b^2_{1l}k_1\otimes k_1
=-\sum\limits_{i=2}^6(e(t)-a_{i\ell_1})k_i\otimes k_i.\nonumber
\end{align}
Meanwhile, by (\ref{d:R0small}), we arrive at
\begin{align}
R_{01}=-\sum_{i=2}^6(e(t)-a_{i\ell_1})k_i\otimes k_i+\delta R_{01}.
\end{align}
Next section, we will prove that $\delta R_{01}$ is small.
\subsection{The form of II}

\indent

By the definition (\ref{d:l amp}) on $b_{1l}$ and (\ref{d:b1l}) on $\beta_{1l}$, it's easy to obtain
\begin{align}
2\sum\limits_{l\in Z^3}\beta_{1l}b_{1l}k_1
=-\sum\limits_{l\in Z^3}\alpha_l^2(\mu v_{0\ell_1})c_{1\ell_1}k_1
=-c_{1\ell_1}k_1.\nonumber
\end{align}
Thus, using the identity (\ref{d:decomposition 2}), we have
\begin{align}
&f_{0\ell_1}+2\sum\limits_{l\in Z^3}\beta_{1l}b_{1l}k_1=\sum\limits_{i=2}^3c_{i\ell_1}k_i.
\end{align}
Meanwhile, by (\ref{d:f0small}), we have
\begin{align}
f_{01}=\sum\limits_{i=2}^3c_{i\ell_1}k_i+\delta f_{01}.
\end{align}
Next section, we will prove that $\delta f_{01}$ is small.

\setcounter{equation}{0}

 \section{Estimate on $\delta R_{01}$ and $\delta f_{01}$}

 \indent


In the subsequent estimates, unless otherwise stated, $C_0$ denotes a constant which depends on $\|v_0\|_0$,
  but does not depend on $\ell_1,~~\ell_{1z},~~\mu,~~\lambda_1$, and $C_r$ will in addition to depend on $r$
and both of them can vary from line to line.

 In the following, we frequently use the elementary inequalities
 \begin{align}
[fg]_{r}\leq& C_r\bigl([f]_{r}\|g\|_0+[g]_{r}\|f\|_0\bigr)\label{i:inequality 1}
\end{align}
for any $r\geq0$.
Moreover, by the standard estimates on convolution, (\ref{d:difinition on c1 norm}), (\ref{d:decomposition of R0}) and (\ref{d:decomposition of f0}), we have
\begin{align}
\|a_{i\ell_1}\|_r+\|c_{i\ell_1}\|_r+\|v_{0\ell_1}\|_r+\|\theta_{0\ell_1}\|_r\leq& C_r\Lambda \ell_1^{1-r}\quad {\rm for~~any~~r\geq1},\label{e:estimate higher convolution}\\
\|v_{0\ell_1}-v_0\|_0+\|R_{0\ell_1}-R_0\|_0+\|f_{0\ell_1}-f_0\|_0+\|\theta_{0\ell_1}-\theta_0\|_0\leq& C_0(\Lambda \ell_1+\bar{\Lambda} \ell_{1z}).\label{e:estimate different}
\end{align}

Now, we collect a classical estimate on the H\"{o}lder norms of compositions, their proof can be found in \cite{CDL3}.
Let $u: R^n\rightarrow R^N$ and $\Psi: R^N\rightarrow R$ be two smooth functions. Then, for every $m\in {\rm N}\setminus {0}$ there is a constant $C_0=C_0(m,N,n)$ such that
\begin{align}
[\Psi(u)]_m\leq& C_0\sum\limits_{i=1}^m[\Psi]_i[u]_1^{(i-1)\frac{m}{m-1}}[u]_m^{\frac{m-i}{m-1}}.\label{i:composition inequality 2}
\end{align}

We summarize the main estimates on $b_{1l}$ and $\beta_{1l}$.
\begin{Lemma}\label{e:estimate various}
For any integer $r\geq 1$, we have the following estimates, for any $t>0$
\begin{align}
\|b_{1l}(t,\cdot)\|_r+\|\beta_{1l}(t,\cdot)\|_r\leq& C_r\sqrt{\kappa}\mu\Lambda\ell_1^{-(r-1)},\label{e:estimate on main perturbation}\\
\|\partial_zb_{1l}(t,\cdot)\|_r+\|\partial_z\beta_{1l}(t,\cdot)\|_r\leq& C_r\sqrt{\kappa}\mu\bar{\Lambda}\ell_1^{-r},\label{e:diffusion derivative estimate on main perturbation}\\
\|\partial_{zz}b_{1l}(t,\cdot)\|_r+\|\partial_{zz}\beta_{1l}(t,\cdot)\|_r\leq& C_r\sqrt{\kappa}\mu\bar{\Lambda}\ell^{-1}_{1z}\ell_1^{-r},\label{e:diffusion two order derivative estimate on main perturbation}\\
\|\partial_{zzz}b_{1l}(t,\cdot)\|_r+\|\partial_{zzz}\beta_{1l}(t,\cdot)\|_r\leq & C_r\sqrt{\kappa}\mu\bar{\Lambda}\ell^{-2}_{1z}\ell_1^{-r},\label{e:diffusion three order derivative estimate on main perturbation}\\
\|\partial_tb_{1l}(t,\cdot)\|_r+\|\partial_t\beta_{1l}(t,\cdot)\|_r\leq& C_r\sqrt{\kappa}\mu\Lambda\ell_1^{-r},\label{e:estimate on time derivative}\\
\|\partial_{tt}b_{1l}(t,\cdot)\|_r+\partial_{tt}\beta_{1l}(t,\cdot)\|_r\leq& C_r\sqrt{\kappa}\mu\Lambda\ell_1^{-r-1}\label{e:estimate on two time order}
\end{align}
and
\begin{align}
\|h_{\pm 1l}(t,\cdot)\|_r+\|g_{\pm 1l}(t,\cdot)\|_r\leq &C_r\sqrt{\kappa}\mu\Lambda\ell_1^{-(r-1)},\label{e:estimate on transport amp}\\
\|\partial_zh_{\pm 1l}(t,\cdot)\|_r+\|\partial_zg_{\pm 1l}(t,\cdot)\|_r\leq & C_r\sqrt{\kappa}\mu\bar{\Lambda}\ell_1^{-r},\label{e:diffusion derivative estimate on transport amp}\\
\|\partial_{zz}h_{\pm 1l}(t,\cdot)\|_r+\|\partial_{zz}g_{\pm 1l}(t,\cdot)\|_r\leq &C_r\sqrt{\kappa}\mu\bar{\Lambda}\ell^{-1}_{1z}\ell_1^{-r},\label{e:diffusion two order derivative estimate on transport amp}\\
\|\partial_{zzz}h_{\pm 1l}(t,\cdot)\|_r+\|\partial_{zzz}g_{\pm 1l}(t,\cdot)\|_r\leq & C_r\sqrt{\kappa}\mu\bar{\Lambda}\ell^{-2}_{1z}\ell_1^{-r},\label{e:diffusion three order derivative estimate on transport amp}\\
\|\partial_th_{\pm 1l}(t,\cdot)\|_r+\|\partial_tg_{\pm 1l}(t,\cdot)\|_r\leq& C_r\sqrt{\kappa}\mu\Lambda\ell_1^{-r},\label{e:estimate on transport time derivative}\\
\|\partial_{tt}h_{\pm 1l}(t,\cdot)\|_r+\|\partial_{tt}g_{\pm 1l}(t,\cdot)\|_r\leq &C_r\sqrt{\kappa}\mu\Lambda\ell_1^{-r-1}.\label{e:estimate on b1l}
\end{align}
Moreover,  for any $t>0$
\begin{align}
\|b_{1l}(t,\cdot)\|_0+\|\beta_{1l}(t,\cdot)\|_0\leq& C_0\sqrt{\kappa},\label{e:zero estimate on main perturbation}\\
\|\partial_zb_{1l}(t,\cdot)\|_0+\|\partial_z\beta_{1l}(t,\cdot)\|_0\leq& C_0\sqrt{\kappa}\mu\bar{\Lambda},\label{e:zero diffusion derivative estimate on main perturbation}\\
\|\partial_{zz}b_{1l}(t,\cdot)\|_0+\|\partial_{zz}\beta_{1l}(t,\cdot)\|_0\leq& C_0\sqrt{\kappa}\mu\bar{\Lambda}\ell^{-1}_{1z},\label{e:zero diffusion two order derivative estimate on main perturbation}\\
\|\partial_{zzz}b_{1l}(t,\cdot)\|_0+\|\partial_{zzz}\beta_{1l}(t,\cdot)\|_0\leq & C_0\sqrt{\kappa}\mu\bar{\Lambda}\ell^{-2}_{1z},\label{e:zero diffusion three order derivative estimate on main perturbation}\\
\|\partial_tb_{1l}(t,\cdot)\|_0+\|\partial_t\beta_{1l}(t,\cdot)\|_0\leq& C_0\sqrt{\kappa}\mu\Lambda,\label{e:zero estimate on time derivative}\\
\|\partial_{tt}b_{1l}(t,\cdot)\|_0+\partial_{tt}\beta_{1l}(t,\cdot)\|_0\leq& C_0\sqrt{\kappa}\mu\Lambda\ell_1^{-1}\label{e:zero estimate on two time order}
\end{align}
and
\begin{align}
\|h_{\pm 1l}(t,\cdot)\|_0+\|h_{\pm 1l}(t,\cdot)\|_0\leq &C_0\sqrt{\kappa},\label{e:zero estimate on transport amp}\\
\|\partial_zh_{\pm 1l}(t,\cdot)\|_0+\|\partial_zg_{\pm 1l}(t,\cdot)\|_0\leq &C_0\sqrt{\kappa}\mu\bar{\Lambda},\label{e:zero diffusion derivative estimate on transport amp}\\
\|\partial_{zz}h_{\pm 1l}(t,\cdot)\|_0+\|\partial_{zz}g_{\pm 1l}(t,\cdot)\|_0\leq &C_0\sqrt{\kappa}\mu\bar{\Lambda}\ell^{-1}_{1z},\label{e:zero diffusion two order derivative estimate on transport amp}\\
\|\partial_{zzz}h_{\pm 1l}(t,\cdot)\|_0+\|\partial_{zzz}g_{\pm 1l}(t,\cdot)\|_0\leq & C_0\sqrt{\kappa}\mu\bar{\Lambda}\ell^{-2}_{1z},\label{e:zero diffusion three order derivative estimate on transport amp}\\
\|\partial_th_{\pm 1l}(t,\cdot)\|_0+\|\partial_tg_{\pm 1l}(t,\cdot)\|_0\leq& C_0\sqrt{\kappa}\mu\Lambda,\label{e:zero estimate on transport time derivative}\\
\|\partial_{tt}h_{\pm 1l}(t,\cdot)\|_0+\|\partial_{tt}h_{\pm 1l}(t,\cdot)\|_0\leq &C_0\sqrt{\kappa}\mu\Lambda\ell_1^{-1}.\label{e:zero estimate on b1l}
\end{align}
\begin{proof}
First, by (\ref{b:bound on decomposition coeffience}), (\ref{d:definition e}), (\ref{e:estimate higher convolution}), (\ref{i:composition inequality 2}) and assumption (\ref{a:assumption on parameter}), for $r\geq2$, we obtain
\begin{align}
\Big[\sqrt{e(t)-a_{1\ell_1}(t,\cdot)}\Big]_r\leq & C_r\sum\limits_{i=1}^r\big\|(e(t)-a_{1\ell_1}(t,\cdot))^{\frac{1}{2}-i}\big\|_0[a_{1\ell_1}(t,\cdot))]_1^{(i-1)\frac{r}{r-1}}[a_{1\ell_1}(t,\cdot))]_r^{\frac{r-i}{r-1}}\nonumber\\
\leq & C_r\sqrt{\kappa}(\mu^r\Lambda^r+\mu\Lambda\ell_1^{-(r-1)})\leq C_r\sqrt{\kappa}\mu\Lambda\ell_1^{-(r-1)},\nonumber\\
\big[\alpha_l(\mu v_{0\ell_1})(t,\cdot)\big]_r\leq & C_r\sum\limits_{i=1}^r\|(\nabla^{i}\alpha)_l\|_0[\mu v_{0\ell_1}(t,\cdot))]_1^{(i-1)\frac{r}{r-1}}[\mu v_{0\ell_1}(t,\cdot))]_r^{\frac{r-i}{r-1}}\nonumber\\
\leq & C_r(\mu^r\Lambda^r+\mu\Lambda\ell_1^{-(r-1)})\leq C_r\mu\Lambda\ell_1^{-(r-1)}.\label{e:derivative estimate on decomposition}
\end{align}
Similarly,
\begin{align}\label{e:estimate on inverse amp}
\left[\frac{1}{\sqrt{e(t)-a_{1\ell_1}(t,\cdot)}}\right]_r\leq \frac{C_r}{\sqrt{\kappa}}\mu\Lambda\ell_1^{-(r-1)}.
\end{align}
Moreover, for $r=0,1$, it's easy to get
\begin{align}
\left[\sqrt{e(t)-a_{1\ell_1}(t,\cdot)}\right]_r\leq  C_r\sqrt{\kappa}\mu^r\Lambda^r,\quad
\big[\alpha_l(\mu v_{0\ell_1})(t,\cdot)\big]_r\leq  C_r\mu^r\Lambda^r.\nonumber
\end{align}
Recalling that
 \begin{align}
    b_{1l}:=\sqrt{\frac{e(t)-a_{1\ell_1}}{2}}\alpha_l(\mu v_{0\ell_1}),\nonumber
   \end{align}
by (\ref{i:inequality 1}), we have
\begin{align}
\|b_{1l}\|_r\leq C_r\sqrt{\kappa}\mu\Lambda\ell_1^{-(r-1)}\quad {\rm for~~any~~r\geq1},\quad \|b_{1l}\|_0\leq C_0\sqrt{\kappa}.\nonumber
\end{align}
By (\ref{d:difinition on c1 norm}) and (\ref{e:estimate higher convolution}), for any $r\geq1$,
\begin{align}\label{e:estimate on high derivative on cil}
\|c_{1\ell_1}\|_r\leq C_r\Lambda\ell_1^{-(r-1)},\quad \|\partial_z^ic_{1\ell_1}\|_r\leq C_i\bar{\Lambda}\ell_{1z}^{-(i-1)}\ell_1^{-r},~~~i=1,2,3
\end{align}
and
\begin{align}\label{e:zero estimate on high derivative on cil}
\|c_{1\ell_1}\|_0\leq C_0\kappa,\quad \|\partial_z^ic_{1\ell_1}\|_0\leq C_i\bar{\Lambda}\ell_{1z}^{-(i-1)},~~~i=1,2,3.
\end{align}
Thus, by (\ref{b:bound on decomposition coeffience}) on $\beta_{1l}$, (\ref{e:estimate on inverse amp}), (\ref{e:estimate on high derivative on cil}) and the same argument as above, we also have
\begin{align}
\|\beta_{1l}\|_r\leq C_r\sqrt{\kappa}\mu\Lambda\ell_1^{-(r-1)},\quad \forall r\geq 1,\quad
\|\beta_{1l}\|_0\leq C_0\sqrt{\kappa}.
\end{align}
By (\ref{d:b1l}) on $g_{\pm 1l}$, (\ref{d:difinition on h1l}) on $h_{\pm 1l}$ and (\ref{a:assumption on parameter}), it's easy to obtain
\begin{align}
&\|h_{\pm1l}\|_r\leq C_r\sqrt{\kappa}\mu\Lambda\ell_1^{-(r-1)},\quad
\|g_{\pm 1l}\|_r\leq C_r\sqrt{\kappa}\mu\Lambda\ell_1^{-(r-1)},\quad \forall r\geq 1,\nonumber\\
&\|h_{\pm1l}\|_0\leq C_0\sqrt{\kappa},\quad
\|g_{\pm 1l}\|_0\leq C_0\sqrt{\kappa}.\nonumber
\end{align}
Thus we complete the proof of (\ref{e:estimate on main perturbation}), (\ref{e:estimate on transport amp}), (\ref{e:zero estimate on main perturbation}) and (\ref{e:zero estimate on transport amp}).

We compute
\begin{align}
\partial_z\Big(\sqrt{e(t)-a_{1\ell_1}}\Big)&=-\frac{1}{2}\frac{(\partial_za_1)_{\ell_1}}{\sqrt{e(t)-a_{1\ell_1}}},\nonumber\\
\partial_{zz}\Big(\sqrt{e(t)-a_{1\ell_1}}\Big)&=-\frac{1}{2}\frac{
\partial_za_1\ast\big(\varphi_{\ell_1}(t)\varphi_{\ell_1}(x)\varphi_{\ell_1}(y)(\varphi')_{\ell_{1z}}(z)\big)
\ell^{-1}_{1z}}{\sqrt{e(t)-a_{1\ell_1}}}
-\frac{1}{4}\frac{(\partial_za_1)^2_{\ell_1}}{(\sqrt{e(t)-a_{1\ell_1}})^3},\nonumber\\
\partial_{zzz}\Big(\sqrt{e(t)-a_{1\ell_1}}\Big)&=-\frac{1}{2}\frac{
\partial_za_1\ast\big(\varphi_{\ell_1}(t)(\varphi)_{\ell_1}(x)\varphi_{\ell_1}(y)(\varphi'')_{\ell_{1z}}(z)\big)
\ell^{-2}_{1z}}{\sqrt{e(t)-a_{1\ell_1}}}\nonumber\\
&-\frac{3}{4}\frac{(\partial_za_1)_{\ell_1}
\partial_za_1\ast\big(\varphi_{\ell_1}(t)(\varphi)_{\ell_{1}}(x)\varphi_{\ell_1}(y)(\varphi')_{\ell_{1z}}(z)\big)
\ell^{-1}_{1z}}{(\sqrt{e(t)-a_{1\ell_1}})^3}
-\frac{3}{8}\frac{(\partial_za_1)^3_{\ell_1}}{(\sqrt{e(t)-a_{1\ell_1}})^5},
\end{align}
by (\ref{d:difinition on c1 norm}), (\ref{i:inequality 1}) and (\ref{i:composition inequality 2}), for any integer $r\geq 0$, we have
\begin{align}\label{e:estimate on x derivative}
\Big\|\partial_z\Big(\sqrt{e(t)-a_{1\ell_1}}\Big)\Big\|_r\leq C_r\sqrt{\kappa}\mu\bar{\Lambda}\ell_1^{-r},\quad&
\Big\|\partial_{zz}\Big(\sqrt{e(t)-a_{1\ell_1}}\Big)\Big\|_r\leq  C_r\sqrt{\kappa}\mu\bar{\Lambda}\ell^{-1}_{1z}\ell_1^{-r},\nonumber\\
\Big\|\partial_{zzz}\Big(\sqrt{e(t)-a_{1\ell_1}}\Big)\Big\|_r\leq  C_r\sqrt{\kappa}\mu\bar{\Lambda}\ell^{-2}_{1z}\ell_1^{-r}.
\end{align}
A similar argument gives
\begin{align}\label{e:estimate on unity function}
\Big\|\partial_z\Big(\alpha_l(\mu v_{0\ell_1})\Big)\Big\|_r\leq& C_r\mu\bar{\Lambda}\ell_1^{-r},\quad
\Big\|\partial_{zz}\Big(\alpha_l(\mu v_{0\ell_1})\Big)\Big\|_r\leq  C_r\mu\bar{\Lambda}\ell^{-1}_{1z}\ell_1^{-r},\nonumber\\
\Big\|\partial_{zzz}\Big(\alpha_l(\mu v_{0\ell_1})\Big)\Big\|_r\leq&  C_r\mu\bar{\Lambda}\ell^{-2}_{1z}\ell_1^{-r},
\end{align}
hence, we obtain
\begin{align}
\big\|\partial_zb_{1l}\big\|_r\leq C_r\sqrt{\kappa}\mu\bar{\Lambda}\ell_1^{-r},\quad
\big\|\partial_{zz}b_{1l}\big\|_r\leq  C_r\sqrt{\kappa}\mu\bar{\Lambda}\ell^{-1}_{1z}\ell_1^{-r},\quad
\big\|\partial_{zzz}b_{1l}\big\|_r\leq  C_r\sqrt{\kappa}\mu\bar{\Lambda}\ell^{-2}_{1z}\ell_1^{-r}.\nonumber
\end{align}
Similarly,
\begin{align}\label{e:estimate on inverse x derivative}
&\Big\|\partial_z\Big(\frac{1}{\sqrt{e(t)-a_{1\ell_1}}}\Big)\Big\|_r\leq \frac{C_r}{\sqrt{\kappa}}\mu\bar{\Lambda}\ell_1^{-r},\quad
\Big\|\partial_{zz}\Big(\frac{1}{\sqrt{e(t)-a_{1\ell_1}}}\Big)\Big\|_r\leq \frac{C_r}{\sqrt{\kappa}}\mu\bar{\Lambda}\ell^{-1}_{1z}\ell_1^{-r},\nonumber\\
&\Big\|\partial_{zzz}\Big(\frac{1}{\sqrt{e(t)-a_{1\ell_1}}}\Big)\Big\|_r\leq  \frac{C_r}{\sqrt{\kappa}}\mu\bar{\Lambda}\ell^{-2}_{1z}\ell_1^{-r}.
\end{align}
Finally, combining (\ref{e:estimate on high derivative on cil}), (\ref{e:estimate on unity function}) and (\ref{e:estimate on inverse x derivative}), we can obtain
\begin{align}
&\big\|\partial_z\beta_{1l}\big\|_r\leq C_r\sqrt{\kappa}\mu\bar{\Lambda}\ell_1^{-r},\quad
\big\|\partial_{zz}\beta_{1l}\big\|_r\leq  C_r\sqrt{\kappa}\mu\bar{\Lambda}\ell^{-1}_{1z}\ell_1^{-r},\quad
\big\|\partial_{zzz}\beta_{1l}\big\|_r\leq  C_r\sqrt{\kappa}\mu\bar{\Lambda}\ell^{-2}_{1z}\ell_1^{-r}.\nonumber
\end{align}
Thus, we obtain (\ref{e:diffusion derivative estimate on main perturbation})-(\ref{e:diffusion three order derivative estimate on main perturbation}), (\ref{e:zero diffusion derivative estimate on main perturbation})-(\ref{e:zero diffusion three order derivative estimate on main perturbation}).  By (\ref{d:difination b1l}) on $g_{\pm 1l}$, (\ref{d:difinition on h1l}) on $h_{\pm 1l}$ and parameter assumption (\ref{a:assumption on parameter}), it's easy to obtain (\ref{e:diffusion derivative estimate on transport amp})-(\ref{e:diffusion three order derivative estimate on transport amp}), (\ref{e:zero diffusion derivative estimate on transport amp})-(\ref{e:zero diffusion three order derivative estimate on transport amp}).

We let
$$\Gamma:=\sqrt{\frac{e(t)-a_{1\ell_1}}{2}}$$
and observe that
$$b_{1l}=\Gamma\alpha_l(\mu v_{0\ell_1}),$$
thus,
\begin{align}
\partial_tb_{1l}=&\partial_t\Gamma \alpha_l(\mu v_{0\ell_1})+\Gamma (\nabla\alpha)_l(\mu v_{0\ell_1})\cdot\mu(\partial_tv_0)_{\ell_1},\nonumber\\
\partial_{tt}b_{1l}=&\partial_{tt}\Gamma \alpha_l(\mu v_{0\ell_1})+2\partial_t\Gamma (\nabla\alpha)_l(\mu v_{0\ell_1})\cdot\mu(\partial_tv_0)_{\ell_1}+
\mu^2\Gamma (\partial_tv_0)^T_{\ell_1}(\nabla^2\alpha)_l(\mu v_{0\ell_1})(\partial_tv_0)_{\ell_1}\nonumber\\
&+\Gamma (\nabla\alpha)_l(\mu v_{0\ell_1})\cdot\mu(\partial_tv_0)\ast(\partial_t\varphi)_{\ell_1}\ell_1^{-1}.\nonumber
\end{align}
A directly calculation gives
\begin{align}
\partial_t\Gamma=\frac{e'(t)-(\partial_ta_1)_{\ell_1}}{\sqrt{2(e(t)-a_{1\ell_1})}},\quad
\partial_{tt}\Gamma=\frac{e^{''}(t)-(\partial_ta_1)\ast (\varphi')_{\ell_1}\ell_1^{-1}}{\sqrt{2(e(t)-a_{1\ell_1})}}+
\frac{(e'(t)-(\partial_ta_1)_{\ell_1})^2}{(\sqrt{2(e(t)-a_{1\ell_1}})^3},\nonumber
\end{align}
therefore, by assumption (\ref{a:assumption on parameter}), we have
\begin{align}
\|\partial_t\Gamma\|_0\leq C_0\sqrt{\kappa}\mu\Lambda,\quad
\|\partial_{tt}\Gamma\|_0\leq C_0\sqrt{\kappa}\mu\Lambda\ell_1^{-1}.\nonumber
\end{align}
By (\ref{i:inequality 1}), (\ref{i:composition inequality 2}) and assumption (\ref{a:assumption on parameter}), for $r\geq 1$ we have
\begin{align}
\|\partial_t\Gamma\|_r\leq C_r\sqrt{\kappa}\mu\Lambda\ell_1^{-r},\quad
\|\partial_{tt}\Gamma\|_r\leq C_r\sqrt{\kappa}\mu\Lambda\ell_1^{-r-1}.\nonumber
\end{align}
Then by (\ref{i:inequality 1}), we obtain that for any $r\geq 0$
\begin{align}
\|\partial_tb_{1l}\|_r\leq C_r\sqrt{\kappa}\mu\Lambda\ell_1^{-r},\quad
\|\partial_{tt}b_{1l}\|_r\leq C_r\sqrt{\kappa}\mu\Lambda\ell_1^{-r-1}.\nonumber
\end{align}
By a similar argument, we obtain, for any $r\geq 0$
\begin{align}
\|\partial_t\beta_{1l}\|_r\leq C_r\sqrt{\kappa}\mu\Lambda\ell_1^{-r},\quad&
\|\partial_{tt}\beta_{1l}\|_r\leq C_r\sqrt{\kappa}\mu\Lambda\ell_1^{-r-1}.\nonumber
\end{align}
Thus, we obtain (\ref{e:estimate on time derivative}), (\ref{e:estimate on two time order}), (\ref{e:zero estimate on time derivative}) and (\ref{e:zero estimate on two time order}).
From (\ref{d:difination b1l}) and (\ref{d:difinition on h1l}), it's easy to get (\ref{e:estimate on transport time derivative}), (\ref{e:estimate on b1l}), (\ref{e:zero estimate on transport time derivative}) and (\ref{e:zero estimate on b1l}).
Then the proof of this lemma is complete.
\end{proof}
\end{Lemma}
From the definition on $w_{1o}, w_{1c}, \chi_{1o}, \chi_{1c}$ and the above lemma, we have the following estimates:
 \begin{Lemma}[Estimates on main perturbation and correction]\label{e: estimate correction}
 \begin{align}
 \|w_{1o}\|_0\leq C_0\sqrt{\kappa},\quad  \|(\partial_t, \partial_x, \partial_y)w_{1o}\|_0\leq& C_0\sqrt{\kappa}\lambda_1,\quad  \|\partial_z w_{1o}\|_0\leq C_0\sqrt{\kappa}\mu\bar{\Lambda},\label{e:estimate main pertuebations}\\
 \|w_{1c}\|_0\leq C_0\sqrt{\kappa}\mu\Lambda\lambda_1^{-1},\quad  \|(\partial_t,\partial_x, \partial_y)w_{1c}\|_0\leq& C_0\sqrt{\kappa}\mu\Lambda,\quad \|\partial_zw_{1c}\|_0\leq C_0\sqrt{\kappa}\mu\bar{\Lambda}\ell_1^{-1}\lambda_1^{-1},\label{e:estimate on correction}\\
  \|\chi_{1o}\|_0\leq C_0\sqrt{\kappa},\quad  \|(\partial_t, \partial_x, \partial_y)\chi_{1o}\|_0\leq& C_0\sqrt{\kappa}\lambda_1,\quad  \|\partial_z \chi_{1o}\|_0\leq C_0\sqrt{\kappa}\mu\bar{\Lambda},\label{e:estimate on terperture correction}\\
  \|\chi_{1c}\|_0\leq C_0\sqrt{\kappa}\mu\Lambda\lambda_1^{-1},\quad  \|(\partial_t,\partial_x, \partial_y)\chi_{1c}\|_0\leq& C_0\sqrt{\kappa}\mu\Lambda,\quad \|\partial_z\chi_{1c}\|_0\leq C_0\sqrt{\kappa}\mu\bar{\Lambda}\ell_1^{-1}\lambda_1^{-1}.\label{e:estimate on termperture correction}
\end{align}
 \begin{proof}
 First, by (\ref{e:bound w1o}), we know $\|w_{1o}\|_0\leq C_0\sqrt{\kappa}$. Since
 \begin{align}
 \partial_tw_{1o}=&\sum_{l\in Z^3}\partial_tb_{1l}k_1\Big(e^{i\lambda_1 2^{[l]} (k_{1h}^{\perp},0)\cdot \big((x,y,z)-\frac{l}{\mu}t\big)}+e^{-i\lambda_1 2^{[l]} (k_{1h}^{\perp},0)\cdot \big((x,y,z)-\frac{l}{\mu}t\big)}\Big)\nonumber\\
 &-\sum_{l\in Z^3}b_{1l}i\lambda_1 2^{[l]} (k_{1h}^{\perp},0)\cdot\frac{l}{\mu}k_1\Big(e^{i\lambda_1 2^{[l]} (k_{1h}^{\perp},0)\cdot \big((x,y,z)-\frac{l}{\mu}t\big)}-e^{-i\lambda_1 2^{[l]}(k_{1h}^{\perp},0)\cdot \big((x,y,z)-\frac{l}{\mu}t\big)}\Big).\nonumber
 \end{align}
 Thus, by (\ref{e:zero estimate on main perturbation}), (\ref{e:zero estimate on time derivative}), parameter assumption (\ref{a:assumption on parameter}) and notice that $b_{1l}\neq 0$ implies $|l|\leq C_0\mu$, we obtain
 \begin{align}
 \|\partial_tw_{1o}\|_0\leq C_0\sqrt{\kappa}\lambda_1.\nonumber
\end{align}
A similar argument gives
\begin{align}
 \|(\partial_x,\partial_y) w_{1o}\|_0\leq C_0\sqrt{\kappa}\lambda_1,\quad  \|\partial_z w_{1o}\|_0\leq C_0\sqrt{\kappa}\mu\bar{\Lambda}.\nonumber
\end{align}
Then we obtain the proof of (\ref{e:estimate main pertuebations}). And (\ref{e:estimate on terperture correction}) follows similarly.

Next, by (\ref{d:w 1cl}), (\ref{i:inequality 1}), (\ref{e:estimate on main perturbation}) and the assumption (\ref{a:assumption on parameter}), we get
 \begin{align}
 \|w_{1cl}\|_0
 \leq C_0\frac{\|\nabla b_{1l}\|_0}{\lambda_1}
 \leq C_0\frac{\sqrt{\kappa}\mu\Lambda}{\lambda_1}.\nonumber
 \end{align}
By (\ref{p:unity}), we arrive at
 \begin{align}
 \|w_{1c}\|_0\leq C_0\sqrt{\kappa}\mu\Lambda\lambda_1^{-1}.\nonumber
 \end{align}
A straightforward computation gives
\begin{align}
   &\partial_tw_{1cl}\nonumber\\
   =&\left(
    \begin{array}{ccc}
    -t_1\partial_{tz}b_{1l}+\partial_{ty}b_{1l}\\
    s_1\partial_{tz}b_{1l}-\partial_{tx}b_{1l}\\
    t_1\partial_{tx}b_{1l}-s_1\partial_{ty}b_{1l}
    \end{array}
    \right)\Big(\frac{e^{i\lambda_1 2^{[l]} (k_{1h}^{\perp},0)\cdot \big((x,y,z)-\frac{l}{\mu}t\big)}}{i\lambda_12^{[l]}}+\frac{e^{-i\lambda_1 2^{[l]} (k_{1h}^{\perp},0)\cdot \big((x,y,z)-\frac{l}{\mu}t\big)}}{-i\lambda_12^{[l]}}\Big)\nonumber\\
   &-\left(
    \begin{array}{ccc}
    -t_1\partial_zb_{1l}+\partial_yb_{1l}\\
    s_1\partial_zb_{1l}-\partial_xb_{1l}\\
    t_1\partial_xb_{1l}-s_1\partial_yb_{1l}
    \end{array}
    \right) (k_{1h}^{\perp},0)\cdot\frac{l}{\mu}\Big(e^{i\lambda_1 2^{[l]} (k_{1h}^{\perp},0)\cdot \big((x,y,z)
    -\frac{l}{\mu}t\big)}
    +e^{-i\lambda_1 2^{[l]} (k_{1h}^{\perp},0)\cdot \big((x,y,z)-\frac{l}{\mu}t\big)}\Big).\nonumber
   \end{align}
By (\ref{i:inequality 1}), (\ref{e:estimate on time derivative}), (\ref{a:assumption on parameter}) and notice that $b_{1l}\neq 0$ implies $|l|\leq C_0\mu$, we get
   \begin{align}
 \|\partial_tw_{1c}\|_0\leq C_0\sqrt{\kappa}\mu\Lambda.\nonumber
 \end{align}
Similarly
 \begin{align}
 \|(\partial_x,\partial_y) w_{1c}\|_0\leq C_0\sqrt{\kappa}\mu\Lambda.\nonumber
 \end{align}
 Differentiating $w_{1cl}$ in $z$, we have
 \begin{align}
   \partial_zw_{1cl}=&\left(
    \begin{array}{ccc}
    -t_1\partial_{zz}b_{1l}+\partial_{zy}b_{1l}\\
    s_1\partial_{zz}b_{1l}-\partial_{zx}b_{1l}\\
    t_1\partial_{zx}b_{1l}-s_1\partial_{zy}b_{1l}
    \end{array}
    \right)\Big(\frac{e^{i\lambda_1 2^{[l]} (k_{1h}^{\perp},0)\cdot \big((x,y,z)-\frac{l}{\mu}t\big)}}{i\lambda_12^{[l]}}+\frac{e^{-i\lambda_1 2^{[l]} (k_{1h}^{\perp},0)\cdot \big((x,y,z)-\frac{l}{\mu}t\big)}}{-i\lambda_12^{[l]}}\Big).\nonumber
   \end{align}
 Thus, by (\ref{e:diffusion derivative estimate on main perturbation})
   \begin{align}
 \|\partial_zw_{1c}\|_0\leq C_0\sqrt{\kappa}\mu\bar{\Lambda}\ell_1^{-1}\lambda_1^{-1}.\nonumber
 \end{align}
Collecting all these estimates, we obtain the proof of (\ref{e:estimate on correction}). A similar argument gives (\ref{e:estimate on termperture correction}).
 \end{proof}
 \end{Lemma}
By (\ref{e:bound w1o}), (\ref{e:bound x}), lemma \ref{e: estimate correction} and constructions (\ref{d:the first solution sequence}) of $(v_{01}, p_{01}, \theta_{01})$, we conclude that
 \begin{Corollary}\label{e:first difference sequence estimate}
\begin{align}\label{e:difference estimate on various quantity}
&\|v_{01}-v_0\|_0\leq \frac{M\sqrt{\kappa}}{12}+C_0\sqrt{\kappa}\mu\Lambda\lambda_1^{-1},\quad\|(\partial_t, \partial_x,\partial_y)(v_{01}-v_0)\|_0\leq C_0\sqrt{\kappa}\lambda_1,\nonumber\\
&\|\theta_{01}-\theta_{0}\|_0\leq \frac{M\sqrt{\kappa}}{6}+C_0\sqrt{\kappa}\mu\Lambda\lambda_1^{-1},\quad
 \|(\partial_t, \partial_x,\partial_y)(\theta_{01}-\theta_{0})\|_0\leq C_0\sqrt{\kappa}\lambda_1,\nonumber\\
& \|p_{01}-p_0\|_0=0,\quad \|\partial_z(v_{01}-v_0)\|_0\leq C_0\sqrt{\kappa}\mu\bar{\Lambda},\quad\|\partial_z(\theta_{01}-\theta_{0})\|_0\leq C_0\sqrt{\kappa}\mu\bar{\Lambda}.
\end{align}
\end{Corollary}

 \subsection{Estimate on $\delta R_{01}$}

 \indent

 As in \cite{CDL2}, we split $\delta R_{01}$ into three parts, they are \\
  (1) The oscillation part
  $$\mathcal{R}(\hbox{div}M_1)-\mathcal{R}(\chi_1e_3)-\mathcal{R}(\partial_{zz}w_1).$$
   (2) The transport part
  \begin{align}
  \mathcal{R}\Big\{\partial_tw_{1}
   +\hbox{div}\Big[\sum_{l\in Z^3}\Big(w_{1l}\otimes \frac{l}{\mu}+\frac{l}{\mu}\otimes w_{1l}\Big)\Big]\Big\}
   =\mathcal{R}\Big(\partial_tw_{1}+\sum_{l\in Z^3}\frac{l}{\mu}\cdot\nabla w_{1l}\Big).\nonumber
   \end{align}
  (3) The error part
  \begin{align}
  N_1+(w_{1o}\otimes w_{1c}+w_{1c}\otimes w_{1o}+w_{1c}\otimes w_{1c})+R_0-R_{0\ell_1}.\nonumber
   \end{align}
Where we followed the notations $M_1$ and $N_1$ given in section 4.3.

In the following we will estimate each of them separately.

\begin{Lemma}[The oscillation part]\label{e:oscillation estimate}
\begin{align}
\|\mathcal{R}({\rm div}M_1)\|_0\leq& C_0(\varepsilon)\kappa \mu\Lambda\lambda_1^{-1},\quad \quad\|(\partial_t, \partial_x,\partial_y)\mathcal{R}({\rm div}M_1)\|_0\leq C_0(\varepsilon)\kappa \mu\Lambda,\nonumber\\
 \|\partial_z\mathcal{R}({\rm div}M_1)\|_0\leq&  C_0(\varepsilon)\kappa\mu\bar{\Lambda}\ell_1^{-1}\lambda_1^{-1},\label{e:oscillation estimate main perturbation}\\
\big\|\mathcal{R}(\chi_1e_3)\big\|_0\leq& C_0(\varepsilon)\sqrt{\kappa}\lambda_1^{-1},\quad \quad
\big\|(\partial_t,\partial_x,\partial_y)\mathcal{R}(\chi_1e_3)\big\|_0\leq C_0(\varepsilon)\sqrt{\kappa},\nonumber\\
 \big\|\partial_z\mathcal{R}(\chi_1e_3)\big\|_0\leq& C_0(\varepsilon)\sqrt{\kappa}\mu\bar{\Lambda}\lambda_1^{-1},\label{e:oscillation estimate terperture}\\
\big\|\mathcal{R}(\partial_{zz}w_1)\big\|_0\leq& C_0(\varepsilon)\sqrt{\kappa}\mu\bar{\Lambda}\ell_{1z}^{-1}\lambda_1^{-1},\quad \quad
\big\|(\partial_t,\partial_x,\partial_y)\mathcal{R}(\partial_{zz}w_1)\big\|_0\leq C_0(\varepsilon)\sqrt{\kappa}\mu\bar{\Lambda}\ell_{1z}^{-1}, \nonumber\\ \big\|\partial_z\mathcal{R}(\partial_{zz}w_1)\big\|_0\leq& C_0(\varepsilon)\sqrt{\kappa}\mu\bar{\Lambda}\ell_{1z}^{-2}\lambda_1^{-1}.\label{e:diffusion estimate}
\end{align}
\begin{proof}
We start with the fact that $k_1\cdot (k_{1h}^{\perp},0)=0.$ Recalling the notation of $M_1$ in (\ref{d:difinition on m1}), we have
\begin{align}
\hbox{div}M_1=M_{11}+M_{12},\nonumber
\end{align}
where
\begin{align}
M_{11}
   =&\sum\limits_{j=0}^7\sum\limits_{[l]=j}\Big(e^{2i\lambda_1 2^{j}  (k_{1h}^{\perp},0)\cdot \big((x,y,z)-\frac{l}{\mu}t\big)}+
   e^{-2i\lambda_1 2^{j}  (k_{1h}^{\perp},0)\cdot \big((x,y,z)-\frac{l}{\mu}t\big)}\Big)k_1\otimes k_1\nabla b^2_{1l},\nonumber\\
   M_{12}=&\sum\limits_{l,l'\in Z^3 ,l\neq l'}\hbox{div}(w_{1ol}\otimes w_{1ol'}).\nonumber
\end{align}
By (\ref{e:estimate on main perturbation}), (\ref{e:zero estimate on main perturbation}), (\ref{e:estimate oscillatory inverse}) on $\mathcal{R}$ with $m=1+\Big[\frac{1+\varepsilon}{\varepsilon}\Big]$
 and (\ref{a:assumption on parameter}), we arrive at
\begin{align}
\|\mathcal{R}(M_{11})\|_0
\leq C_m\sum\limits_{j=0}^7\Big(\sum_{i=0}^{m-1}\lambda_1^{-(i+1)}\Big\|\sum\limits_{[l]=j}\nabla b^2_{1l}(t,\cdot)\Big\|_i+\lambda_1^{-m}
\Big[\sum\limits_{[l]=j}\nabla b^2_{1l}(t,\cdot)\Big]_m
\Big)\leq C_m\kappa \mu\Lambda\lambda_1^{-1}.\nonumber
\end{align}
On the other hand, since $k_1\cdot (k_{1h}^{\perp},0)=0$ and by the notation (4.10), we obtain
\begin{align}
M_{12}=&\sum\limits_{j=0}^7\sum\limits_{[l]=j}\sum\limits_{1\leq|l-l'|<2}k_1\otimes k_1\nabla(b_{1l}b_{1l'})\Big(e^{i\lambda_1(2^{j}+2^{[l']})(k_{1h}^\perp,0)\cdot (x,y,z) -ig_{1,l,l'}(t)}\nonumber\\
&+e^{i\lambda_1(2^{j}-2^{[l']})(k_{1h}^\perp,0)\cdot(x,y,z)-i\overline{g}_{1,l,l'}(t)}
+e^{i\lambda_1(2^{[l']}-2^{j})(k_{1h}^\perp,0)\cdot (x,y,z)+i\overline{g}_{1,l,l'}(t)}\nonumber\\
&+e^{-i\lambda_1(2^{j}+2^{[l']})(k_{1h}^\perp,0)\cdot (x,y,z)+ig_{1,l,l'}(t)}\Big),\nonumber
\end{align}
where $$g_{1,l,l'}(t)=\lambda_1\Big(2^{[l]} (k_{1h}^\perp,0)\cdot\frac{l}{\mu}t+2^{[l']} (k_{1h}^\perp,0)\cdot\frac{l'}{\mu}t\Big),\quad
 \overline{g}_{1,l,l'}(t)=\lambda_1\Big(2^{[l]} (k_{1h}^\perp,0)\cdot\frac{l}{\mu}t-
 2^{[l']} (k_{1h}^\perp,0)\cdot\frac{l'}{\mu}t\Big).$$
As in the estimate of $M_{11}$, by (\ref{e:estimate on main perturbation}), (\ref{e:zero estimate on main perturbation}), (\ref{e:estimate oscillatory inverse}) on $\mathcal{R}$ with $m=1+\Big[\frac{1+\varepsilon}{\varepsilon}\Big]$
and (\ref{a:assumption on parameter}), and by noticing $b_{1p}b_{1q}=0$ if $|p-q|\geq2$, we have
\begin{align}
&\|\mathcal{R}(M_{12})\|_0\nonumber\\
\leq& C_m\sum\limits_{j=0}^7\Big(\sum_{i=0}^{m-1}\lambda_1^{-(i+1)}\Big\|\sum\limits_{[l]=j}\sum\limits_{1\leq|l-l'|<2} \nabla(b_{1l}b_{1l'})\Big\|_i+\lambda_1^{-m}
\sum_{i=0}^{m-1}\lambda_1^{-(i+1)}\Big[\sum\limits_{[l]=j}\sum\limits_{1\leq|l-l'|<2} \nabla(b_{1l}b_{1l'})\Big]_m
\Big)\nonumber\\
\leq& C_m\kappa \mu\Lambda\lambda_1^{-1}.\nonumber
\end{align}
Then we obtain the first estimate in (\ref{e:oscillation estimate main perturbation}) by summing up the two parts.

Next, differentiating $M_{11}$ in time,
\begin{align}
\partial_tM_{11}=&\sum\limits_{j=0}^7\sum\limits_{[l]=j}\Big(e^{2i\lambda_1 2^{j} (k_{1h}^{\perp},0)\cdot \big((x,y,z)-\frac{l}{\mu}t\big)}+
   e^{-2i\lambda_1 2^{j}  (k_{1h}^{\perp},0)\cdot \big((x,y,z)-\frac{l}{\mu}t\big)}\Big)k_1\otimes k_1\nabla\partial_t (b^2_{1l})\nonumber\\
  -&\sum\limits_{j=0}^7\sum\limits_{[l]=j}2i\lambda_1 2^{j}  (k_{1h}^{\perp},0)\cdot\frac{l}{\mu}\Big(e^{2i\lambda_1 2^{j}  (k_{1h}^{\perp},0) \cdot \big((x,y,z)-\frac{l}{\mu}t\big)}\nonumber\\
  -& e^{-2i\lambda_1 2^{j}  (k_{1h}^{\perp},0)\cdot \big((x,y,z)-\frac{l}{\mu}t\big)}\Big)k_1\otimes k_1\nabla (b^2_{1l}).\nonumber
\end{align}
Noticing $|l|\leq C_0\mu$ and applying the same argument as above, we can obtain
\begin{align}
\|\partial_t\mathcal{R}(M_{11})\|_0\leq C_m\kappa \mu\Lambda\lambda_1^{-1}(\ell_1^{-1}+\mu\Lambda+\lambda_1) \leq C_m\kappa \mu\Lambda.\nonumber
\end{align}
Similarly, we also have
\begin{align}
\|\partial_t\mathcal{R}(M_{12})\|_0 \leq C_m\kappa \mu\Lambda.\nonumber
\end{align}
Differentiating in $(x,y)$ on $M_{11}, M_{12}$, similarly we have
\begin{align}
\|(\partial_x,\partial_y)\mathcal{R}(M_{11})\|_0\leq C_m\kappa \mu\Lambda,\quad\|(\partial_x,\partial_y)\mathcal{R}(M_{12})\|_0\leq C_m\kappa \mu\Lambda.\nonumber
\end{align}
Finally we obtain
\begin{align}
\|(\partial_t,\partial_x,\partial_y)\mathcal{R}({\rm div}M_1)\|_0\leq C_m\kappa \mu_1\Lambda.\nonumber
\end{align}
This is the second estimate in (\ref{e:oscillation estimate main perturbation}).

A straightforward computation gives
\begin{align}
\partial_zM_{11}
   =&\sum\limits_{j=0}^7\sum\limits_{[l]=j}\Big(e^{2i\lambda_1 2^{j}  (k_{1h}^\perp,0)\cdot \big((x,y,z)-\frac{l}{\mu}t\big)}+
   e^{-2i\lambda_1 2^{j} (k_{1h}^\perp,0)\cdot \big((x,y,z)-\frac{l}{\mu}t\big)}\Big)k_1\otimes k_1(\nabla\partial_z) b^2_{1l}.\nonumber
  \end{align}
By (\ref{e:diffusion derivative estimate on main perturbation}), (\ref{e:estimate oscillatory inverse}) on $\mathcal{R}$ with $m=1+\Big[\frac{1+\varepsilon}{\varepsilon}\Big]$
 and (\ref{a:assumption on parameter}), we arrive at
 \begin{align}
\|\mathcal{R}(\partial_zM_{11})\|_0
  \leq C_m\kappa\mu\bar{\Lambda}\ell_1^{-1}\lambda_1^{-1}.\nonumber
  \end{align}
Similarly, we have
   \begin{align}
\|\mathcal{R}(\partial_zM_{12})\|_0
  \leq C_m\kappa\mu\bar{\Lambda}\ell_1^{-1}\lambda_1^{-1}.\nonumber
  \end{align}
  Combining the two parts, we arrive at
  $$\|\partial_z\mathcal{R}({\rm div}M_1)\|_0
  \leq C_m\kappa\mu\bar{\Lambda}\ell_1^{-1}\lambda_1^{-1}.$$
Then the proof of (\ref{e:oscillation estimate main perturbation}) is complete.

Since
\begin{align}
\chi_{1}=\sum\limits_{j=0}^7\sum\limits_{[l]=j}\Big(h_{1l}e^{i\lambda_1 2^{j}  (k_{1h}^\perp,0)\cdot \big((x,y,z)-\frac{l}{\mu}t\big)}+h_{-1l}e^{-i\lambda_1 2^{j} (k_{1h}^\perp,0)\cdot \big((x,y,z)-\frac{l}{\mu}t\big)}\Big),\nonumber
\end{align}
then by (\ref{e:estimate on transport amp}), (\ref{e:zero estimate on transport amp}), (\ref{e:estimate oscillatory inverse}) on $\mathcal{R}$ with $m=1+\Big[\frac{1+\varepsilon}{\varepsilon}\Big]$
 and  (\ref{a:assumption on parameter}), we obtain
\begin{align}
\big\|\mathcal{R}(\chi_{1}e_3)\big\|_0\leq& C_m\sum\limits_{j=0}^7\Big[\sum_{i=0}^{m-1}\lambda_1^{-(i+1)}\Big(\Big\|\sum\limits_{[l]=j}h_{1l}(t,\cdot)\Big\|_i
+\Big\|\sum\limits_{[l]=j}h_{-1l}(t,\cdot)\Big\|_i\Big)\nonumber\\
&+\lambda_1^{-m}
\Big[\sum\limits_{[l]=j}h_{1l}(t,\cdot)+\sum\limits_{[l]=j}h_{-1l}(t,\cdot)\Big]_m\Big]
\leq C_m\sqrt{\kappa}\lambda_1^{-1}.\nonumber
\end{align}
Differentiating in space and time on $\chi_1$, we have
\begin{align}
\big\|(\partial_t, \partial_x,\partial_y)\mathcal{R}(\chi_{1}e_3)\big\|_0
\leq C_0(\varepsilon)\sqrt{\kappa},\quad \big\|\partial_z\mathcal{R}(\chi_{1}e_3)\big\|_0
\leq C_0(\varepsilon)\sqrt{\kappa}\mu\bar{\Lambda}\lambda_1^{-1}.\nonumber
\end{align}
Then we complete the proof of (\ref{e:oscillation estimate terperture}).

Finally, by (\ref{d:another representation}), we have
\begin{align}
\partial_{zz}w_1=\sum\limits_{j=0}^7\sum\limits_{[l]=j}\Big(\partial_{zz}g_{1l}e^{i\lambda_1 2^{j} (k_{1h}^{\perp},0)\cdot \big((x,y,z)-\frac{l}{\mu}t\big)}+\partial_{zz}g_{-1l}e^{-i\lambda_1 2^{j} (k_{1h}^{\perp},0)\cdot \big((x,y,z)-\frac{l}{\mu}t\big)}\Big),\nonumber
\end{align}
then by (\ref{e:diffusion two order derivative estimate on transport amp}), (\ref{e:zero diffusion two order derivative estimate on transport amp}),  (\ref{e:estimate oscillatory inverse}) on $\mathcal{R}$ with $m=1+\Big[\frac{1+\varepsilon}{\varepsilon}\Big]$
 and (\ref{a:assumption on parameter}), we have
\begin{align}
\|\mathcal{R}(\partial_{zz}w_1)\|_0
\leq &C_m\sum\limits_{j=0}^7\Big[\sum_{i=0}^{m-1}\lambda_1^{-(i+1)}\Big(\Big\|\sum\limits_{[l]=j}\partial_{zz}g_{1l}(t,\cdot)\Big\|_i
+\Big\|\sum\limits_{[l]=j}\partial_{zz}g_{-1l}(t,\cdot)\Big\|_i\Big)\nonumber\\
&+\lambda_1^{-m}
\Big[\sum\limits_{[l]=j}\partial_{zz}g_{1l}(t,\cdot)+\sum\limits_{[l]=j}\partial_{zz}g_{-1l}(t,\cdot)\Big]_m\Big]
\leq C_m\sqrt{\kappa}\mu\bar{\Lambda}\ell_{1z}^{-1}\lambda_1^{-1}.\nonumber
\end{align}
Similarly, we obtain
\begin{align}
\|(\partial_t,\partial_x,\partial_y)\mathcal{R}(\partial_{zz}w_1)\|_0
\leq C_m\sqrt{\kappa}\mu\bar{\Lambda}\ell_{1z}^{-1}.\nonumber
\end{align}
And
\begin{align}
\partial_{zzz}w_1=\sum\limits_{j=0}^7\sum\limits_{[l]=j}\Big(\partial_{zzz}g_{1l}e^{i\lambda_1 2^{j} k_1^{\perp}\cdot \big((y,z)-\frac{l}{\mu}t\big)}+\partial_{zzz}g_{-1l}e^{-i\lambda_1 2^{j} k_1^{\perp}\cdot \big((y,z)-\frac{l}{\mu}t\big)}\Big),\nonumber
\end{align}
By (\ref{e:diffusion three order derivative estimate on transport amp}), (\ref{e:zero diffusion three order derivative estimate on transport amp}) and a similar argument as above, we obtain
\begin{align}
\|\partial_z\mathcal{R}(\partial_{zz}w_1)\|_0
\leq C_m\sqrt{\kappa}\mu\bar{\Lambda}\ell_{1z}^{-2}\lambda_{1}^{-1}.\nonumber
\end{align}
Then we complete the proof of (\ref{e:diffusion estimate}).
\end{proof}
\end{Lemma}

\begin{Lemma}[The transportation part]\label{e:transport estimate}
\begin{align}
&\Big\|\mathcal{R}\Big(\partial_tw_{1}+\sum_{l\in Z^3}\frac{l}{\mu}\cdot\nabla w_{1l}\Big)\Big\|_0\leq C_0(\varepsilon)\sqrt{\kappa}\mu\Lambda\lambda_1^{-1},\nonumber\\
&\Big\|(\partial_t,\partial_x,\partial_y)\mathcal{R}\Big(\partial_tw_{1}+\sum_{l\in Z^3}\frac{l}{\mu}\cdot\nabla w_{1l}\Big)\Big\|_0\leq C_0(\varepsilon)\sqrt{\kappa}\mu\Lambda,\nonumber\\
&\Big\|\partial_z\mathcal{R}\Big(\partial_tw_{1}+\sum_{l\in Z^3}\frac{l}{\mu}\cdot\nabla w_{1l}\Big)\Big\|_0\leq C_0(\varepsilon)\sqrt{\kappa}\mu\bar{\Lambda}\ell_1^{-1}\lambda_1^{-1}.
\end{align}
\begin{proof}
Recall that
\begin{align}
w_{1l}=g_{1l}e^{i\lambda_1 2^{[l]} (k_{1h}^{\perp},0)\cdot \big((x,y,z)-\frac{l}{\mu}t\big)}+g_{-1l}e^{-i\lambda_1 2^{[l]} (k_{1h}^{\perp},0)\cdot \big((x,y,z)-\frac{l}{\mu}t\big)}\nonumber
\end{align}
and $$w_1=\sum_{l\in Z^3}w_{1l}.$$
From the identity
$$\Big(\partial_t+\frac{l}{\mu}\cdot\nabla\Big)e^{\pm i\lambda_1 2^{[l]}(k_{1h}^{\perp},0)\cdot \big((x,y,z)-\frac{l}{\mu}t\big)}=0,$$
we arrive at
\begin{align}\label{e:transport representation}
&\partial_tw_{1}+\sum_{l\in Z^3}\frac{l}{\mu}\cdot\nabla w_{1l}\nonumber\\
=&\sum\limits_{j=0}^7\sum\limits_{[l]=j}\Big(\Big(\partial_t+
\frac{l}{\mu}\cdot\nabla\Big)g_{1l}e^{i\lambda_1 2^{j} (k_{1h}^{\perp},0)\cdot \big((x,y,z)-\frac{l}{\mu}t\big)}+\Big(\partial_t+
\frac{l}{\mu}\cdot\nabla\Big)g_{-1l}e^{-i\lambda_1 2^{j} (k_{1h}^{\perp},0)\cdot \big((x,y,z)-\frac{l}{\mu}t\big)}\Big{)}.
   \end{align}
By (\ref{e:estimate on transport amp}) and  (\ref{e:estimate oscillatory inverse}) on $\mathcal{R}$ with $m=1+\Big[\frac{1+\varepsilon}{\varepsilon}\Big]$, we have
\begin{align}
&\Big\|\mathcal{R}\Big(\partial_tw_{1}+\sum_{l\in Z^3}\frac{l}{\mu}\cdot\nabla w_{1l}\Big)\Big\|_0\nonumber\\
\leq& C_m\sum\limits_{j=0}^7\Big\{\sum_{i=0}^{m-1}\lambda_1^{-(i+1)}\Big(
\Big\|\sum\limits_{[l]=j}\Big(\partial_t+\frac{l}{\mu_{1}}\cdot\nabla\Big)g_{1l}(t,\cdot)\Big\|_i
+\Big\|\sum\limits_{[l]=j}\Big(\partial_t+\frac{l}{\mu_{1}}\cdot\nabla\Big)g_{-1l}(t,\cdot)\Big\|_i\Big)\nonumber\\
&+\lambda_1^{-m}\Big(\Big[\sum\limits_{[l]=j}\Big{(}\partial_t+\frac{l}{\mu_{1}}\cdot\nabla\Big)g_{1l}(t,\cdot)\Big]_m
+\Big[\sum\limits_{[l]=j}\Big{(}\partial_t+\frac{l}{\mu_{1}}\cdot\nabla\Big)g_{-1l}(t,\cdot)\Big]_m\Big)
\Big\}
\leq C_m\sqrt{\kappa}\mu\Lambda\lambda_1^{-1}.\nonumber
\end{align}
Differentiating in $(x,y)$ and $t$ on (\ref{e:transport representation}) and similarly we have
\begin{align}
\Big\|(\partial_t,\partial_x,\partial_y)\mathcal{R}\Big(\partial_tw_{1}+\sum_{l\in Z^3}\frac{l}{\mu}\cdot\nabla w_{1l}\Big)\Big\|_0
\leq C_m\sqrt{\kappa}\mu\Lambda.\nonumber
\end{align}
Moreover, we have
\begin{align}\label{e:transport representation}
&\partial_z\Big(\partial_tw_{1}+\sum_{l\in Z^3}\frac{l}{\mu}\cdot\nabla w_{1l}\Big)\nonumber\\
=&\sum\limits_{j=0}^7\sum\limits_{[l]=j}\Bigg(\Big(\partial_t+\frac{l}{\mu}\cdot\nabla\Big)\partial_zg_{1l}e^{i\lambda_1 2^{j} (k_{1h}^{\perp},0)\cdot \big((x,y,z)-\frac{l}{\mu}t\big)}+\Big(\partial_t+\frac{l}{\mu}\cdot\nabla\Big)\partial_zg_{-1l}e^{-i\lambda_1 2^{j} (k_{1h}^{\perp},0)\cdot \big((x,y,z)-\frac{l}{\mu}t\big)}\Bigg),\nonumber
   \end{align}
then by (\ref{e:diffusion derivative estimate on transport amp}), the above argument gives
 \begin{align}
\Big\|\partial_z\mathcal{R}\Big(\partial_tw_{1}+\sum_{l\in Z^3}\frac{l}{\mu}\cdot\nabla w_{1l}\Big)\Big\|_0
\leq C_m\sqrt{\kappa}\mu\bar{\Lambda}\ell_1^{-1}\lambda_1^{-1}.\nonumber
\end{align}
Then the proof of the lemma is complete.
\end{proof}
\end{Lemma}

\begin{Lemma}[Estimates on error part I]\label{e:error 1}
\begin{align}
&\|N_1\|_0\leq C_0\sqrt{\kappa}\Big(\mu^{-1}+\Lambda\ell_1+\bar{\Lambda}\ell_{1z}\Big),\quad \|(\partial_t, \partial_x,\partial_y)N_1\|_0\leq C_0\lambda_1\sqrt{\kappa}\Big(\mu^{-1}+\Lambda\ell_1+\bar{\Lambda}\ell_{1z}\Big),\nonumber\\
& \|\partial_zN_1\|_0\leq C_0\sqrt{\kappa}\bar{\Lambda}.
\end{align}
\end{Lemma}
\begin{proof}
From the definition (\ref{d:difinition on m1}) on $N_{1}$, we decompose $N_1$ into two parts
\begin{align}
 N_1=N_{11}+N_{12},\nonumber
\end{align}
where
\begin{align}
&N_{11}=\sum_{l\in Z^3}\Big[w_{1l}\otimes \Big(v_{0\ell_1}-\frac{l}{\mu}\Big)
   +\Big(v_{0\ell_1}-\frac{l}{\mu}\Big)\otimes w_{1l}\Big],\nonumber\\
&N_{12}=\sum_{l\in Z^3}\Big[w_{1l}\otimes \big(v_0-v_{0\ell_1}\big)
   +\big(v_0-v_{0\ell_1}\big)\otimes w_{1l}\Big].\nonumber
\end{align}
For the term $N_{11}$, by (\ref{d:another representation}), we have
\begin{align}\label{e:representation difference}
\sum_{l\in Z^3}w_{1l}\otimes \Big(v_{0\ell_1}-\frac{l}{\mu}\Big)
=&\sum_{l\in Z^3}\Big(g_{1l}e^{i\lambda_1 2^{j} (k_{1h}^{\perp},0)\cdot \big((x,y,z)-\frac{l}{\mu}t\big)}\nonumber\\
&+g_{-1l}e^{-i\lambda_1 2^{j} (k_{1h}^{\perp},0)\cdot \big((x,y,z)-\frac{l}{\mu}t\big)}\Big)\otimes \Big(v_{0\ell_1}-\frac{l}{\mu}\Big).
\end{align}
Obviously, by (\ref{d:difination b1l}), we know that $g_{\pm 1l}\neq0$ implies $|\mu v_{0\ell_1}-l|\leq1.$
By (\ref{p:unity}) and (\ref{e:estimate on transport amp}), it's easy to get
\begin{align}
\Big\|\sum_{l\in Z^3}w_{1l}\otimes \Big(v_{0\ell_1}-\frac{l}{\mu}\Big)\Big\|_0\leq C_0\sqrt{\kappa}\mu^{-1}.\nonumber
\end{align}
Similarly,
\begin{align}
\Big\|\sum_{l\in Z^3}\Big(v_{0\ell_1}-\frac{l}{\mu}\Big)\otimes w_{1l} \Big\|_0\leq C_0\sqrt{\kappa}\mu^{-1}.\nonumber
\end{align}
Hence we have
\begin{align}
\|N_{11}\|_0\leq C_0\sqrt{\kappa}\mu^{-1}.\nonumber
\end{align}
Similarly, by (\ref{e:estimate different}), we have
\begin{align}
\|N_{12}\|_0\leq C_0\sqrt{\kappa}(\Lambda\ell_1+\bar{\Lambda}\ell_{1x}).\nonumber
\end{align}
We obtain the first estimate of this lemma by summing up the two parts.

Differentiating (\ref{e:representation difference}) in time,
\begin{align}
&\partial_t\Big(\sum_{l\in Z^3}w_{1l}\otimes \Big(v_{0\ell_1}-\frac{l}{\mu}\Big)\Big)\nonumber\\
=&\sum_{l\in Z^3}\Big(\partial_tg_{1l}e^{i\lambda_1 2^{[l]} (k_{1h}^{\perp},0)\cdot \big((x,y,z)-\frac{l}{\mu}t\big)}+\partial_tg_{-1l}e^{-i\lambda_1 2^{[l]} (k_{1h}^{\perp},0)\cdot \big((x,y,z)-\frac{l}{\mu}t\big)}\Big)\otimes \Big(v_{0\ell_1}-\frac{l}{\mu}\Big)\nonumber\\
&+\sum_{l\in Z^3}\Big(g_{1l}e^{i\lambda_1 2^{[l]}(k_{1h}^{\perp},0)\cdot \big((x,y,z)-\frac{l}{\mu}t\big)}+g_{-1l}e^{-i\lambda_1 2^{[l]}(k_{1h}^{\perp},0)\cdot \big((x,y,z)-\frac{l}{\mu}t\big)}\Big)\otimes (\partial_tv_0)_{\ell_1}\nonumber\\
&-\sum\limits_{j=0}^7\sum\limits_{[l]=j}i\lambda_1 2^{j} (k_{1h}^{\perp},0)\cdot\frac{l}{\mu}\Big(g_{1l}e^{i\lambda_1 2^{j} (k_{1h}^{\perp},0)\cdot \big((x,y,z)-\frac{l}{\mu}t\big)}\nonumber-g_{-1l}e^{-i\lambda_1 2^{j} (k_{1h}^{\perp},0)\cdot \big((x,y,z)-\frac{l}{\mu}t\big)}\Big)\otimes \Big(v_{0\ell_1}-\frac{l}{\mu}\Big).
\end{align}
Notice that $|l|\leq C_0\mu$. By (\ref{a:assumption on parameter}), a similar argument as before, we get
\begin{align}
\Big\|\partial_t\Big(\sum_{l\in Z^3}w_{1l}\otimes \Big(v_{0\ell_1}-\frac{l}{\mu}\Big)\Big)\Big\|_0\leq C_0\sqrt{\kappa}\lambda_1\mu^{-1}.\nonumber
\end{align}
Similarly,
\begin{align}
\Big\|\partial_t\Big(\sum_{l\in Z^3}\Big(v_{0\ell_1}-\frac{l}{\mu}\Big)\otimes w_{1l} \Big)\Big\|_0\leq C_0\sqrt{\kappa}\lambda_1\mu^{-1}.\nonumber
\end{align}
Therefore,
\begin{align}
\big\|\partial_tN_{11}\big\|_0\leq C_0\sqrt{\kappa}\lambda_1\mu^{-1}.\nonumber
\end{align}
A similar argument as before, we also obtain
\begin{align}
\big\|(\partial_x,\partial_y) N_{11}\big\|_0\leq C_0\sqrt{\kappa}\lambda_1\mu^{-1},\nonumber
\end{align}
and
\begin{align}
\big\|(\partial_t, \partial_x,\partial_y)N_{12}\big\|_0\leq C_0\sqrt{\kappa}(\Lambda\ell_1+\bar{\Lambda}\ell_{1x})\lambda_1.\nonumber
\end{align}
Collecting all these estimates, we arrive at the second estimate of this lemma.

Differentiating (\ref{e:representation difference}) in $z$:
\begin{align}
&\partial_z\Big(\sum_{l\in Z^3}w_{1l}\otimes \Big(v_{0\ell_1}-\frac{l}{\mu}\Big)\Big)\nonumber\\
=&\sum_{l\in Z^3}\Big(\partial_zg_{1l}e^{i\lambda_1 2^{[l]} (k_{1h}^{\perp},0)\cdot \big((x,y,z)-\frac{l}{\mu}t\big)}+\partial_zg_{-1l}e^{-i\lambda_1 2^{[l]} (k_{1h}^{\perp},0)\cdot \big((x,y,z)-\frac{l}{\mu}t\big)}\Big)\otimes \Big(v_{0\ell_1}-\frac{l}{\mu}\Big)\nonumber\\
&+\sum_{l\in Z^3}\Big(g_{1l}e^{i\lambda_1 2^{[l]} (k_{1h}^{\perp},0)\cdot \big((x,y,z)-\frac{l}{\mu}t\big)}+g_{-1l}e^{-i\lambda_1 2^{[l]} (k_{1h}^{\perp},0)\cdot \big((x,y,z)-\frac{l}{\mu}t\big)}\Big)\otimes (\partial_zv_0)_{\ell_1}.\nonumber
\end{align}
By (\ref{p:unity}) and (\ref{e:zero diffusion derivative estimate on transport amp}), we have
\begin{align}
\Big\|\partial_z\Big(\sum_{l\in Z^3}w_{1l}\otimes \Big(v_{0\ell_1}-\frac{l}{\mu}\Big)\Big)\Big\|_0\leq C_0\sqrt{\kappa}\bar{\Lambda}.\nonumber
\end{align}
Similarly
\begin{align}
\Big\|\partial_z\Big(\sum_{l\in Z^3}\Big(v_{0\ell_1}-\frac{l}{\mu}\Big)\Big)\otimes w_{1l}\Big\|_0\leq C_0\sqrt{\kappa}\bar{\Lambda}.\nonumber
\end{align}
Hence
\begin{align}
\big\|\partial_zN_{11}\big\|_0\leq C_0\sqrt{\kappa}\bar{\Lambda}.\nonumber
\end{align}
By (\ref{d:difinition on c1 norm}), (\ref{p:unity}), (\ref{e:estimate different}), lemma \ref{e: estimate correction} and  (\ref{a:assumption on parameter}), a straightforward computation gives
\begin{align}
\|\partial_zN_{12}\|_0\leq C_0\sqrt{\kappa}\bar{\Lambda}.\nonumber
\end{align}
Then we arrive at
\begin{align}
\|\partial_zN_{1}\|_0\leq C_0\sqrt{\kappa}\bar{\Lambda}.\nonumber
\end{align}
Therefor we complete our proof of this lemma.
\end{proof}

\begin{Lemma}[Estimates on error part II]\label{e: error 2}
\begin{align}
\|w_{1o}\otimes w_{1c}+w_{1c}\otimes w_{1o}+w_{1c}\otimes w_{1c}\|_0\leq& C_0\kappa\mu\Lambda\lambda_1^{-1},\nonumber\\
\|(\partial_t,\partial_x,\partial_y)(w_{1o}\otimes w_{1c}+w_{1c}\otimes w_{1o}+w_{1c}\otimes w_{1c})\|_0\leq& C_0\kappa\mu\Lambda,\nonumber\\
\|\partial_z(w_{1o}\otimes w_{1c}+w_{1c}\otimes w_{1o}+w_{1c}\otimes w_{1c})\|_0\leq& C_0\kappa\mu\bar{\Lambda}\ell^{-1}_{1}\lambda_1^{-1}.\nonumber
\end{align}
\begin{proof}
By Lemma \ref{e: estimate correction}, it's easy to obtain
\begin{align}
\|w_{1o}\otimes w_{1c}+w_{1c}\otimes w_{1o}+w_{1c}\otimes w_{1c}\|_0
\leq C_0(\|w_{1o}\|_0\|w_{1c}\|_0+\|w_{1c}\|_0^2)
\leq C_0\kappa\mu\Lambda\lambda_1^{-1}\nonumber
\end{align}
and
\begin{align}
\|(\partial_t,\partial_x,\partial_y)(w_{1o}\otimes w_{1c}+w_{1c}\otimes w_{1o}+w_{1c}\otimes w_{1c})\|_0\leq C_0\kappa\mu\Lambda.\nonumber
\end{align}
Moreover, by lemma \ref{e: estimate correction} and (\ref{a:assumption on parameter})
\begin{align}
&\|\partial_z(w_{1o}\otimes w_{1c}+w_{1c}\otimes w_{1o}+w_{1c}\otimes w_{1c})\|_0\nonumber\\
\leq &C_0\Big(\|\partial_zw_{1o}\|_0\|w_{1c}\|_0+\|\partial_zw_{1c}\|w_{1o}\|_0\Big)\leq C_0\kappa\mu\bar{\Lambda}\ell^{-1}_{1}\lambda_1^{-1}.\nonumber
\end{align}
Then the proof of this lemma is complete.
\end{proof}
\end{Lemma}

\begin{Lemma}[Estimates on error part III]\label{e: error 3}
\begin{align}
\|R_{0\ell_1}-R_0\|_0\leq C_0(\Lambda\ell_1+\bar{\Lambda}\ell_{1z}),\quad \|(\partial_t,\partial_x,\partial_y)(R_{0\ell_1}-R_0)\|_0\leq C_0\Lambda,\quad \|\partial_z(R_{0\ell_1}-R_0)\|_0\leq C_0\bar{\Lambda}.\nonumber
\end{align}
\begin{proof}
By (\ref{e:estimate different}), it's easy to get the first estimate. The remaining two estimates are deduced from (\ref{d:difinition on c1 norm}).
\end{proof}
\end{Lemma}

Finally, by Lemma \ref{e:oscillation estimate}, Lemma \ref{e:transport estimate}, Lemma \ref{e:error 1}, Lemma \ref{e: error 2} and Lemma \ref{e: error 3}, we conclude that
\begin{align}\label{e:first stress error estimate}
\|\delta R_{01}\|_0\leq & C_0(\varepsilon)\Big(\Lambda\ell_1+\bar{\Lambda}\ell_{1z}+\sqrt{\kappa}\mu^{-1}+\lambda_1^{-1}\sqrt{\kappa}(\mu\Lambda
+\mu\bar{\Lambda}\ell_{1z}^{-1})\Big),\nonumber\\
\|(\partial_t,\partial_x,\partial_y)\delta R_{01}\|_0\leq & C_0(\varepsilon)\lambda_1\Big(\Lambda\ell_1+\bar{\Lambda}\ell_{1z}+\sqrt{\kappa}\mu^{-1}+\lambda_1^{-1}\sqrt{\kappa}(\mu\Lambda
+\mu\bar{\Lambda}\ell_{1z}^{-1})\Big),\nonumber\\
\|\partial_z\delta R_{01}\|_0\leq & C_0(\varepsilon)\Big[\bar{\Lambda}+\lambda_1^{-1}\sqrt{\kappa}
\Big(\mu\bar{\Lambda}\ell_1^{-1}+\mu\bar{\Lambda}\ell_{1z}^{-2}\Big)\Big].
\end{align}

\subsection{Estimates on $\delta f_{01}$}

 \indent

Recalling (4.55), as before, we split $\delta f_{01}$ into three parts:\\
(1) the oscillation part
\begin{align}
\mathcal{G}(\hbox{div}K_1)+\mathcal{G}(w_1\cdot\nabla\theta_{0\ell_1})-\mathcal{G}(\partial_{zz}\chi_1),\nonumber
\end{align}
(2) the transportation part
\begin{align}
\mathcal{G}\Big(\partial_t\chi_{1}+\sum_{l\in Z^3}\frac{l}{\mu}\cdot\nabla\chi_{1l}\Big),\nonumber
\end{align}
(3) the error part
\begin{align}
w_{1c}\chi_1+w_{1o}\chi_{1c}+\sum\limits_{l\in Z^3}\Big(v_{0\ell_1}-\frac{l}{\mu}\Big)\chi_{1l}+\sum_{l\in Z^3}\big(v_0-v_{0\ell_1}\big)\chi_{1l}+f_0-f_{0\ell_1}+w_1(\theta_0-\theta_{0\ell_1}).\nonumber
\end{align}

\begin{Lemma}[The Oscillation Part]\label{e:osci}
\begin{align}
\|\mathcal{G}({\rm div}K_1)\|_0\leq& C_0(\varepsilon)\kappa \mu\Lambda\lambda_1^{-1},\quad  \|(\partial_t,\partial_x,\partial_y)\mathcal{G}({\rm div}K_1)\|_0
\leq C_0(\varepsilon)\kappa \mu\Lambda,\nonumber\\
\|\partial_z\mathcal{G}({\rm div}K_1)\|_0
\leq& C_0(\varepsilon)\kappa\mu\bar{\Lambda} \ell_1^{-1}\lambda_1^{-1} ,\label{e:esyimate terpeture 1}\\
\|\mathcal{G}(w_1\cdot\nabla\theta_{0\ell_1})\|_0\leq& C_0(\varepsilon)\sqrt{\kappa}\Lambda\lambda_1^{-1},\quad\|(\partial_t,\partial_x,\partial_y)\mathcal{G}(w_1\cdot\nabla\theta_{0\ell_1})\|_0\leq C_0(\varepsilon)\sqrt{\kappa}\Lambda,\nonumber\\
\|\partial_z\mathcal{G}(w_1\cdot\nabla\theta_{0\ell_1})\|_0\leq& C_0(\varepsilon)\sqrt{\kappa}\Lambda\ell_{1z}^{-1}\lambda_1^{-1},\label{e:esyimate terpeture 2}\\
\|\mathcal{G}(\partial_{zz}\chi_1)\|_0\leq& C_0(\varepsilon)\sqrt{\kappa}\mu\bar{\Lambda}\ell^{-1}_{1z}\lambda_1^{-1},\quad
\|(\partial_t,\partial_x,\partial_y)\mathcal{G}(\partial_{zz}\chi_1)\|_0\leq C_0(\varepsilon)\sqrt{\kappa}\mu\bar{\Lambda}\ell^{-1}_{1z},\nonumber\\
\|\partial_z\mathcal{G}(\partial_{zz}\chi_1)\|_0\leq& C_0(\varepsilon)\sqrt{\kappa}\mu\bar{\Lambda}\ell^{-2}_{1z}\lambda_1^{-1}.\label{e:esyimate terpeture 3}
\end{align}
\begin{proof}
From the definition (\ref{d:difinition on m1}) on $K_1$, we may write
\begin{align}
\hbox{div}K_1=K_{11}+K_{12},\nonumber
\end{align}
where
\begin{align}
&K_1:
=\sum\limits_{j=0}^7\sum\limits_{[l]=j}\nabla(\beta_{1l}b_{1l})\cdot k_1\Big(e^{2i\lambda_1 2^{j} (k_{1h}^{\perp},0)\cdot \big((x,y,z)-\frac{l}{\mu}t\big)}+e^{-2i\lambda_1 2^{j} (k_{1h}^{\perp},0)\cdot \big((x,y,z)-\frac{l}{\mu}t\big)}\Big),\nonumber\\
&K_2:=\sum\limits_{l,l'\in Z^3 ,l\neq l'}\hbox{div}(w_{1ol}\chi_{1ol'}).\nonumber
\end{align}
By (\ref{e:estimate on main perturbation}), (\ref{e:zero estimate on main perturbation}), lemma \ref{p:inverse 2} on $\mathcal{G}$ with $m=1+\Big[\frac{1+\varepsilon}{\varepsilon}\Big]$ and (\ref{a:assumption on parameter}), we have
\begin{align}
\|\mathcal{G}(K_{11})\|_0
\leq C_m\sum\limits_{j=0}^7\Big(\sum_{i=0}^{m-1}\lambda_1^{-(i+1)}\Big\|\sum\limits_{[l]=j}\nabla(\beta_{1l}b_{1l})\Big\|_i
+\lambda_1^{-m}\Big[\sum\limits_{[l]=j}\nabla(\beta_{1l}b_{1l})\Big]_m
\Big)
\leq C_m\kappa \mu\Lambda\lambda_1^{-1}.\nonumber
\end{align}
Since $k_1\cdot (k_{1h}^{\perp},0)=0$, we have
\begin{align}
K_{12}=&\sum\limits_{j=0}^7\sum\limits_{[l]=j}\sum\limits_{1\leq|l-l'|<2}k_1\cdot\nabla(b_{1l}\beta_{1l'})\Big(e^{i\lambda_1(2^{j}+2^{[l']})(k_{1h}^{\perp},0)\cdot (x,y,z)-ig_{1,l,l'}(t)}\nonumber\\
&+e^{i\lambda_1(2^{j}-2^{[l']})(k_{1h}^{\perp},0)\cdot (x,y,z)-i\overline{g}_{1,l,l'}(t)}
+e^{-i\lambda_1(2^{[l']}-2^{j})(k_{1h}^{\perp},0)\cdot (x,y,z)+i\overline{g}_{1,l,l'}(t)}\nonumber\\
&+e^{-i\lambda_1(2^{j}+
2^{[l']})(k_{1h}^{\perp},0)\cdot (x,y,z)+ig_{1,l,l'}(t)}\Big).\nonumber
\end{align}
 By (\ref{e:estimate on main perturbation}), (\ref{e:zero estimate on main perturbation}), lemma \ref{p:inverse 2} on $\mathcal{G}$ with $m=1+\Big[\frac{1+\varepsilon}{\varepsilon}\Big]$, (\ref{a:assumption on parameter}) and $b_{1p}b_{1q}=0$ if $|p-q|\geq2$, as before, we get
\begin{align}
\|\mathcal{G}(K_{12})\|_0
\leq C_m\kappa \mu\Lambda\lambda_1^{-1}.\nonumber
\end{align}
Then we arrive at
$$\|\mathcal{G}({\rm div}K_1)\|_0\leq C_0(\varepsilon)\kappa \mu\Lambda\lambda_1^{-1}.$$
Thus, we obtain the first estimate of (\ref{e:esyimate terpeture 1}).

Directly differentiating in $(x,y)$ and $t$ on $K_{11}, K_{12}$ and following the argument of lemma \ref{e:oscillation estimate}, we get
 \begin{align}
 \|(\partial_t, \partial_x,\partial_y)(\mathcal{G}({\rm div}K_1))\|_0
\leq C_0(\varepsilon)\kappa \mu\Lambda,\nonumber
\end{align}
which is the second estimate in (\ref{e:esyimate terpeture 1}).

We compute
\begin{align}
\partial_zK_1
=\sum\limits_{j=0}^7\sum\limits_{[l]=j}\nabla\partial_z(\beta_{1l}b_{1l})\cdot k_1\Big(e^{2i\lambda_1 2^{j} (k_{1h}^{\perp},0)\cdot \big((x,y,z)-\frac{l}{\mu}t\big)}+e^{-2i\lambda_1 2^{j} (k_{1h}^{\perp},0)\cdot \big((x,y,z)-\frac{l}{\mu}t\big)}\Big),\nonumber
\end{align}
by (\ref{e:estimate on main perturbation}), (\ref{e:diffusion derivative estimate on main perturbation}),  (\ref{e:zero estimate on main perturbation}), (\ref{e:zero diffusion derivative estimate on main perturbation}), lemma \ref{p:inverse 2} on $\mathcal{G}$ with $m=1+\Big[\frac{1+\varepsilon}{\varepsilon}\Big]$ and (\ref{a:assumption on parameter}), we arrive at
 \begin{align}
 \|\partial_z\mathcal{G}(K_{11})\|_0
 \leq& C_m\sum\limits_{j=0}^7\Big(\sum_{i=0}^{m-1}\lambda_1^{-(i+1)}\big\|\sum\limits_{[l]=j}\nabla\partial_z(\beta_{1l}b_{1l})\big\|_i
+\lambda_1^{-m}\big[\sum\limits_{[l]=j}\nabla\partial_z(\beta_{1l}b_{1l})\big]_m
\Big)\nonumber\\
\leq& C_m\kappa\mu\bar{\Lambda}\ell_1^{-1}\lambda_1^{-1}.\nonumber
\end{align}
A similar argument gives
 \begin{align}
 \|\partial_z\mathcal{G}(K_{12})\|_0
\leq C_m\kappa\mu\bar{\Lambda}\ell_1^{-1}\lambda_1^{-1}.\nonumber
\end{align}
Then we obtain the third estimate of (\ref{e:esyimate terpeture 1}).

A straightforward computation gives
\begin{align}\label{e:representation w1ol}
w_1\cdot\nabla\theta_{0\ell_1}
=&\sum\limits_{j=0}^7\sum\limits_{[l]=j}\Big(g_{1l}\cdot\nabla\theta_{0\ell_1}e^{i\lambda_1 2^{j} (k_{1h}^{\perp},0)\cdot \big((x,y,z)-\frac{l}{\mu}t\big)}\nonumber\\
&+g_{-1l}\cdot\nabla\theta_{0\ell_1}e^{-i\lambda_1 2^{j} (k_{1h}^{\perp},0)\cdot \big((x,y,z)-\frac{l}{\mu}t\big)}\Big).
\end{align}
Again, by (\ref{e:estimate higher convolution}), (\ref{e:estimate on transport amp}), (\ref{e:zero estimate on transport amp}), lemma \ref{p:inverse 2} on $\mathcal{G}$ with $m=1+\Big[\frac{1+\varepsilon}{\varepsilon}\Big]$ and  (\ref{a:assumption on parameter}), we have
\begin{align}
\|\mathcal{G}(w_1\cdot\nabla\theta_{0\ell_1})\|_0
\leq &C_m\sum\limits_{j=0}^7\Big\{\sum_{i=0}^{m-1}\lambda_1^{-(1+i)}\Big(\Big\|\sum\limits_{[l]=j}g_{1l}\cdot\nabla\theta_{0\ell_1}(t,\cdot)\Big\|_i
+\Big\|\sum\limits_{[l]=j}g_{-1l}\cdot\nabla\theta_{0\ell_1}(t,\cdot)\Big\|_i\Big)\nonumber\\
&+\lambda_1^{-m}\Big(\Big[\sum\limits_{[l]=j}g_{1l}\cdot\nabla\theta_{0\ell_1}(t,\cdot)\Big]_m
+\Big[\sum\limits_{[l]=j}g_{-1l}\cdot\nabla\theta_{0\ell_1}(t,\cdot)\Big]_m\Big)\Big\}
\leq C_m\sqrt{\kappa}\Lambda\lambda_1^{-1}.\nonumber
\end{align}
Differentiating in $(x,y)$ and $t$ on (\ref{e:representation w1ol}), similarly,
\begin{align}
\|(\partial_t, \partial_x,\partial_y)\mathcal{G}(w_1\cdot\nabla\theta_{0\ell_1})\|_0
\leq C_m\sqrt{\kappa}\Lambda.\nonumber
\end{align}
Differentiating in $z$ on (\ref{e:representation w1ol}), we arrive at
\begin{align}
\partial_z(w_1\cdot\nabla\theta_{0\ell_1})
=&\sum\limits_{j=0}^7\sum\limits_{[l]=j}\Big(\partial_z(g_{1l}\cdot\nabla\theta_{0\ell_1})e^{i\lambda_1 2^{j} (k_{1h}^{\perp},0)\cdot \big((x,y,z)-\frac{l}{\mu}t\big)}\nonumber\\
&+\partial_z(g_{-1l}\cdot\nabla\theta_{0\ell_1})e^{-i\lambda_1 2^{j} (k_{1h}^{\perp},0)\cdot \big((x,y,z)-\frac{l}{\mu}t\big)}\Big),\nonumber
\end{align}
By (\ref{e:estimate higher convolution}), (\ref{e:diffusion derivative estimate on transport amp}) and (\ref{e:zero diffusion derivative estimate on transport amp}), following a similar argument as before, we obtain
\begin{align}
\|\partial_z\mathcal{G}(w_1\cdot\nabla\theta_{0\ell_1})\|_0
\leq C_m\sqrt{\kappa}\Lambda\ell_{1z}^{-1}\lambda_1^{-1},\nonumber
\end{align}
which gives the proof of (\ref{e:esyimate terpeture 2}).

Finally calculating
\begin{align}\label{r:representation formula on xx derivative}
\partial_{zz}\chi_1
=\sum\limits_{j=0}^7\sum\limits_{[l]=j}\Big(\partial_{zz}h_{1l}e^{i\lambda_1 2^{j} (k_{1h}^{\perp},0)\cdot \big((x,y,z)-\frac{l}{\mu}t\big)}+\partial_{zz}h_{-1l}e^{-i\lambda_1 2^{j}(k_{1h}^{\perp},0)\cdot \big((x,y,z)-\frac{l}{\mu}t\big)}\Big),
\end{align}
by (\ref{e:diffusion two order derivative estimate on transport amp}), (\ref{e:zero diffusion two order derivative estimate on transport amp}), lemma \ref{p:inverse 2} on $\mathcal{G}$ with $m=1+\Big[\frac{1+\varepsilon}{\varepsilon}\Big]$ and (\ref{a:assumption on parameter}) we obtain
\begin{align}
\|\mathcal{G}(\partial_{zz}\chi_1)\|_0\leq
 C_m\sqrt{\kappa}\mu\bar{\Lambda}\ell^{-1}_{1z}\lambda_1^{-1}.\nonumber
\end{align}
Differentiating (\ref{r:representation formula on xx derivative}) in space and time, a similar argument also gives
\begin{align}
\|(\partial_t,\partial_x,\partial_y)\mathcal{G}(\partial_{zz}\chi_1)\|_0\leq
 C_m\sqrt{\kappa}\mu\Lambda\ell_1^{-1},\quad
 \|\partial_z\mathcal{G}(\partial_{zz}\chi_1)\|_0\leq
 C_m\sqrt{\kappa}\mu\bar{\Lambda}\ell^{-2}_{1z}\lambda_1^{-1}.\nonumber
\end{align}
Then  the proof of (\ref{e:esyimate terpeture 3}) is complete.
\end{proof}
\end{Lemma}

\begin{Lemma}[The transportation Part]\label{e:trans}
\begin{align}
&\Big\|\mathcal{G}\Big(\partial_t\chi_{1}+\sum_{l\in Z^3}\frac{l}{\mu}\cdot\nabla\chi_{1l}\Big)\Big\|_0\leq  C_0(\varepsilon)\sqrt{\kappa}\mu\Lambda\lambda_1^{-1},\nonumber\\
&\Big\|(\partial_t,\nabla)\mathcal{G}\Big(\partial_t\chi_{1}+\sum_{l\in Z^3}\frac{l}{\mu}\cdot\nabla\chi_{1l}\Big)\Big\|_0\leq C_0(\varepsilon)\sqrt{\kappa}\mu\Lambda,\nonumber\\
&\Big\|\partial_z\mathcal{G}\Big(\partial_t\chi_{1}+\sum_{l\in Z^3}\frac{l}{\mu_{1}}\cdot\nabla\chi_{1l}\Big)\Big\|_0\leq  C_0(\varepsilon)\sqrt{\kappa}\mu\bar{\Lambda}\ell_1^{-1}\lambda_1^{-1}.
\end{align}
\begin{proof}
Since
\begin{align}\label{r:computation formula on transport}
&\partial_t\chi_{1}+\sum_{l\in Z^3}\frac{l}{\mu}\cdot\nabla\chi_{1l}\nonumber\\
=&\sum\limits_{j=0}^7\sum\limits_{[l]=j}\Big(\Big(\partial_t+
\frac{l}{\mu}\cdot\nabla\Big)h_{1l}e^{i\lambda_1 2^{j} (k_{1h}^{\perp},0)\cdot \big((x,y,z)-\frac{l}{\mu}t\big)}+\Big(\partial_t+
\frac{l}{\mu}\cdot\nabla\Big)h_{-1l}e^{-i\lambda_1 2^{j} (k_{1h}^{\perp},0)\cdot \big((x,y,z)-\frac{l}{\mu}t\big)}\Big).
\end{align}
By (\ref{e:estimate on transport amp}), (\ref{e:estimate on transport time derivative}), (\ref{e:zero estimate on transport time derivative}), lemma \ref{p:inverse 2} on $\mathcal{G}$ with $m=1+\Big[\frac{1+\varepsilon}{\varepsilon}\Big]$ and  (\ref{a:assumption on parameter}), as in the proof lemma \ref{e:transport estimate}, we have
\begin{align}
\Big\|\mathcal{G}\Big(\partial_t\chi_{1}+\sum_{l\in Z^3}\frac{l}{\mu}\cdot\nabla\chi_{1l}\Big)\Big\|_0\leq  C_m\sqrt{\kappa}\mu\Lambda\lambda_1^{-1}.\nonumber
\end{align}
Differentiating in $(x,y)$ and $t$ ,  we can deduce
\begin{align}
\Big\|(\partial_t, \partial_x,\partial_y)\mathcal{G}\Big(\partial_t\chi_{1}+\sum_{l\in Z^3}\frac{l}{\mu_{1}}\cdot\nabla\chi_{1l}\Big)\Big\|_0\leq  C_m\sqrt{\kappa}\mu\Lambda.\nonumber
\end{align}
Differentiating (\ref{r:computation formula on transport}) in $z$
\begin{align}
&\partial_z\Big(\partial_t\chi_{1}+\sum_{l\in Z^3}\frac{l}{\mu}\cdot\nabla\chi_{1l}\Big)\nonumber\\
=&\sum\limits_{j=0}^7\sum\limits_{[l]=j}\Big(\Big(\partial_t+
\frac{l}{\mu}\cdot\nabla\Big)\partial_zh_{1l}e^{i\lambda_1 2^{j}(k_{1h}^{\perp},0)\cdot \big((x,y,z)-\frac{l}{\mu}t\big)}+\Big(\partial_t+
\frac{l}{\mu}\cdot\nabla\Big)\partial_zh_{-1l}e^{-i\lambda_1 2^{j} (k_{1h}^{\perp},0)\cdot \big((x,y,z)-\frac{l}{\mu}t\big)}\Big),\nonumber
\end{align}
noticing $|l|\leq C_0 \mu$ and by (\ref{e:diffusion derivative estimate on transport amp}) , we obtain
\begin{align}
\Big\|\partial_z\mathcal{G}\Big(\partial_t\chi_{1}+\sum_{l\in Z^3}\frac{l}{\mu_{1}}\cdot\nabla\chi_{1l}\Big)\Big\|_0\leq  C_0(\varepsilon)\sqrt{\kappa}\mu\bar{\Lambda}\ell_1^{-1}\lambda_1^{-1}.\nonumber
\end{align}
We obtain the proof this lemma.
\end{proof}
\end{Lemma}

\begin{Lemma}[The error Part]\label{e:est 1}
\begin{align}
&\|w_{1c}\chi_1\|_0\leq C_0\kappa\mu\Lambda\lambda_1^{-1},\quad  \|(\partial_t, \partial_x, \partial_y)(w_{1c}\chi_1)\|_0\leq C_0\kappa\mu\Lambda,
\quad \|\partial_z(w_{1c}\chi_1)\|_0\leq C_0\kappa\mu\bar{\Lambda}\ell^{-1}_{1}\lambda_1^{-1}\nonumber,\\
&\Big\|\sum\limits_{l\in Z^3}\Big(v_{0\ell_1}-\frac{l}{\mu}\Big)\chi_{1l}+\sum_{l\in Z^3}\big(v_0-v_{0\ell_1}\big)\chi_{1l}+f_0-f_{0\ell_1}+w_1(\theta_0-\theta_{0\ell_1})\Big\|_0\nonumber\\
\leq&C_0(\Lambda\ell_1+\bar{\Lambda}\ell_{1z}+\sqrt{\kappa}\mu^{-1}),\nonumber\\
&\Big\|(\partial_t,\partial_x,\partial_y)\Big\{\sum\limits_{l\in Z^3}\big(v_{0\ell_1}-\frac{l}{\mu}\Big)\chi_{1l}+\sum_{l\in Z^3}\big(v_0-v_{0\ell_1}\big)\chi_{1l}+f_0-f_{0\ell_1}+w_1(\theta_0-\theta_{0\ell_1})\Big\}\Big\|_0\nonumber\\
&\leq C_0(\Lambda\ell_1+\bar{\Lambda}\ell_{1z}+\sqrt{\kappa}\mu^{-1})\lambda_1,\nonumber\\
&\Big\|\partial_z\Big\{\sum\limits_{l\in Z^3}\Big(v_{0\ell_1}-\frac{l}{\mu}\Big)\chi_{1l}+\sum_{l\in Z^3}\big(v_0-v_{0\ell_1}\big)\chi_{1l}+f_0-f_{0\ell_1}+w_1(\theta_0-\theta_{0\ell_1})\Big\}\Big\|_0
\leq C_0\bar{\Lambda}.
\end{align}
\begin{proof}
First, by  Lemma \ref{e: estimate correction}, it's easy to get
\begin{align}
\|w_{1c}\chi_1\|_0\leq C_0\kappa\mu\Lambda\lambda_1^{-1},\quad \|(\partial_t,\partial_x,\partial_y)w_{1c}\chi_1\|_0\leq C_0\kappa\mu\Lambda,\quad
 \|\partial_z(w_{1c}\chi_1)\|_0\leq C_0\kappa\mu\bar{\Lambda}(\mu\Lambda+\ell^{-1}_{1})\lambda_1^{-1},\nonumber
\end{align}
which is the first three estimates in the lemma.

We compute
\begin{align}\label{f:computation formula about product}
\sum_{l\in Z^3}\Big(v_{0\ell_1}-\frac{l}{\mu}\Big)\chi_{1l}
=\sum\limits_{l\in Z^3}\Big(v_{0\ell_1}-\frac{l}{\mu}\Big)\Big(h_{1l}e^{i\lambda_1 2^{[l]} (k_{1h}^{\perp},0)\cdot \big((x,y,z)-\frac{l}{\mu}t\big)}+h_{-1l}e^{-i\lambda_1 2^{[l]}(k_{1h}^{\perp},0)\cdot \big((x,y,z)-\frac{l}{\mu}t\big)}\Big).
\end{align}
Since $h_{\pm 1l}(x,t)\neq0$ implies $|\mu v_{0\ell_1}-l|\leq1,$
therefore,
\begin{align}
\Big\|\sum_{l\in Z^3}\Big(v_{0\ell_1}-\frac{l}{\mu}\Big)\chi_{1l}\Big\|_0\leq C_0\sqrt{\kappa}\mu^{-1}.\nonumber
\end{align}
Similarly,
\begin{align}
\Big\|(\partial_t, \partial_x,\partial_y)\sum_{l\in Z^3}\Big(v_{0\ell_1}-\frac{l}{\mu}\Big)\chi_{1l}\Big\|_0\leq C_0\sqrt{\kappa}\lambda_1\mu^{-1}.\nonumber
\end{align}
By inequality (\ref{e:estimate different}), we obtain directly
\begin{align}
\Big\|\sum_{l\in Z^3}\big(v_0-v_{0\ell_1}\big)\chi_{1l}+f_0-f_{0\ell_1}+w_1(\theta_0-\theta_{0\ell_1})\Big\|_0\leq& C_0(\Lambda\ell_1+\bar{\Lambda}\ell_{1z}),\nonumber\\
\Big\|(\partial_t, \partial_x,\partial_y)\Big[\sum_{l\in Z^3}\big(v_0-v_{0\ell_1}\big)\chi_{1l}+f_0-f_{0\ell_1}+w_1(\theta_0-\theta_{0\ell_1})\Big]\Big\|_0\leq& C_0(\Lambda\ell_1+\bar{\Lambda}\ell_{1z})\lambda_1.\nonumber
\end{align}
Collecting all these estimates, we obtain the fourth and fifth estimates.

Differentiating in $z$ on (\ref{f:computation formula about product})
\begin{align}
&\partial_z\Big(\sum_{l\in Z^3}\Big(v_{0\ell_1}-\frac{l}{\mu}\Big)\chi_{1l}\Big)\nonumber\\
=&\sum\limits_{l\in Z^3}\Big(v_{0\ell_1}-\frac{l}{\mu}\Big)\Big(\partial_zh_{1l}e^{i\lambda_1 2^{[l]} (k_{1h}^{\perp},0)\cdot \big((x,y,z)-\frac{l}{\mu}t\big)}+\partial_zh_{-1l}e^{-i\lambda_1 2^{[l]} (k_{1h}^{\perp},0)\cdot \big((x,y,z)-\frac{l}{\mu}t\big)}\Big)\nonumber\\
&+\sum\limits_{l\in Z^3}(\partial_zv_0)_{\ell_1}\Big(h_{1l}e^{i\lambda_1 2^{[l]} (k_{1h}^{\perp},0)\cdot \big((x,y,z)-\frac{l}{\mu}t\big)}+h_{-1l}e^{-i\lambda_1 2^{[l]} (k_{1h}^{\perp},0)\cdot \big((x,y,z)-\frac{l}{\mu}t\big)}\Big),\nonumber
\end{align}
by (\ref{p:unity}), (\ref{e:zero estimate on transport amp}), (\ref{e:zero diffusion derivative estimate on transport amp}) and (\ref{e:estimate different}), we obtain
\begin{align}
\Big\|\partial_z\Big(\sum_{l\in Z^3}\Big(v_{0\ell_1}-\frac{l}{\mu}\Big)\chi_{1l}\Big)\Big\|_0\leq& C_0\sqrt{\kappa}\bar{\Lambda}.\nonumber
\end{align}
A similar argument gives
\begin{align}
\Big\|\partial_z\Big(\sum_{l\in Z^3}\big(v_0-v_{0\ell_1}\big)\chi_{1l}\Big)\Big\|_0\leq& C_0\sqrt{\kappa}\bar{\Lambda}.\nonumber
\end{align}
From lemma \ref{e: estimate correction}, inequality (\ref{e:estimate different}) and (\ref{a:assumption on parameter}), it's easy to obtain
\begin{align}
\Big\|\partial_z\Big(f_0-f_{0\ell_1}+w_1(\theta_0-\theta_{0\ell_1})\Big)\Big\|_0\leq C_0\bar{\Lambda}.\nonumber
\end{align}
Collecting all these estimates, we obtain
the proof of the lemma.
\end{proof}
\end{Lemma}

Finally, we conclude that
\begin{align}\label{e:terperture first difference estimate}
\|\delta f_{01}\|_0\leq& C_0(\varepsilon)\Big(\Lambda\ell_1+\bar{\Lambda}\ell_{1z}+\sqrt{\kappa}\mu^{-1}+
\lambda_1^{-1}\sqrt{\kappa}(\mu\Lambda
+\mu\bar{\Lambda}\ell_{1z}^{-1})\Big),\nonumber\\
\|(\partial_t, \partial_x,\partial_y)\delta f_{01}\|_0\leq & C_0(\varepsilon)\lambda_1\Big(\Lambda\ell_1+\bar{\Lambda}\ell_{1z}+
\sqrt{\kappa}\mu^{-1}+\lambda_1^{-1}\sqrt{\kappa}(\mu\Lambda
+\mu\bar{\Lambda}\ell_{1z}^{-1})\Big),\nonumber\\
\|\partial_z\delta f_{01}\|_0\leq & C_0(\varepsilon)\Big[\bar{\Lambda}+\lambda_1^{-1}\sqrt{\kappa}
\Big(\mu\bar{\Lambda}\ell_1^{-1}+\Lambda\ell_{1z}^{-1}+\mu\bar{\Lambda}\ell_{1z}^{-2}\Big)\Big].
\end{align}

In conclusion, combining Corollary \ref{e:first difference sequence estimate}, (\ref{e:first stress error estimate}) and
(\ref{e:terperture first difference estimate}), we have $(v_{01}, p_{01}, \theta_{01}, R_{01}, f_{01})\in C_c^{\infty}((a-\kappa, b+\kappa)\times T^3)$, they solve system (\ref{d:anistropic boussinesq reynold}) and satisfy
\begin{align}
 R_{01}=-\sum\limits_{i=2}^6(e(t)-a_{i\ell_1})k_i\otimes k_i+\delta R_{01},\quad
    f_{01}=-\sum\limits_{i=2}^3c_{i\ell_1}k_i+\delta f_{01}\nonumber
\end{align}
and
\begin{align}\label{e:the first step conclusion}
& \|v_{01}-v_0\|_0\leq \frac{M\sqrt{\kappa}}{12}+C_0\frac{\sqrt{\kappa}\mu\Lambda}{\lambda_1},\quad\|(\partial_t,\partial_x,\partial_y)(v_{01}-v_0)\|_0\leq C_0\sqrt{\kappa}\lambda_1,\nonumber\\
& \|\partial_z(v_{01}-v_0)\|_0\leq C_0\sqrt{\kappa}\mu\bar{\Lambda},
\quad\|p_{01}-p_0\|_0=0,\quad \|\partial_z(\theta_{01}-\theta_{0})\|_0\leq C_0\sqrt{\kappa}\mu\bar{\Lambda}, \nonumber\\
& \|\theta_{01}-\theta_{0}\|_0\leq \frac{M\sqrt{\kappa}}{6}+C_0\frac{\sqrt{\kappa}\mu\Lambda}{\lambda_1},\quad
 \|(\partial_t, \partial_x,\partial_y)(\theta_{01}-\theta_{0})\|_0\leq C_0\sqrt{\kappa}\lambda_1,\nonumber\\
& \|\delta R_{01}\|_0\leq  C_0(\varepsilon)\Big(\Lambda\ell_1+\bar{\Lambda}\ell_{1z}+
\sqrt{\kappa}\mu^{-1}+\lambda_1^{-1}\sqrt{\kappa}(\mu\Lambda
+\mu\bar{\Lambda}\ell_{1z}^{-1})\Big),\nonumber\\
&\|(\partial_t, \partial_x,\partial_y)\delta R_{01}\|_0\leq C_0(\varepsilon)\lambda_1\Big(\Lambda\ell_1+\bar{\Lambda}\ell_{1z}+
\sqrt{\kappa}\mu^{-1}+\lambda_1^{-1}\sqrt{\kappa}(\mu\Lambda
+\mu\bar{\Lambda}\ell_{1z}^{-1})\Big),\nonumber\\
&\|\partial_z\delta R_{01}\|_0\leq  C_0(\varepsilon)\Big[\bar{\Lambda}+\lambda_1^{-1}\sqrt{\kappa}
\Big(\mu\bar{\Lambda}\ell_1^{-1}+\mu\bar{\Lambda}\ell_{1z}^{-2}\Big)\Big],\nonumber\\
&\|\delta f_{01}\|_0\leq C_0(\varepsilon)\Big(\Lambda\ell_1+\bar{\Lambda}\ell_{1z}+
\sqrt{\kappa}\mu^{-1}+\lambda_1^{-1}\sqrt{\kappa}(\mu\Lambda
+\mu\bar{\Lambda}\ell_{1z}^{-1})\Big),\nonumber\\
&\|(\partial_t, \partial_x,\partial_y)\delta f_{01}\|_0\leq  C_0(\varepsilon)\lambda_1\Big(\Lambda\ell_1+\bar{\Lambda}\ell_{1z}+
\sqrt{\kappa}\mu^{-1}+\lambda_1^{-1}\sqrt{\kappa}(\mu\Lambda
+\mu\bar{\Lambda}\ell_{1z}^{-1})\Big),\nonumber\\
&\|\partial_z\delta f_{01}\|_0\leq  C_0(\varepsilon)\Big[\bar{\Lambda}+\lambda_1^{-1}\sqrt{\kappa}
\Big(\mu\bar{\Lambda}\ell_1^{-1}+\Lambda\ell_{1z}^{-1}
+\mu\bar{\Lambda}\ell_{1z}^{-2}\Big)\Big].
\end{align}
Thus we complete the first step.

\setcounter{equation}{0}

\section{Constructions of $v_{0n}, p_{0n}, \theta_{0n}, R_{0n}, f_{0n}$, $2\leq n\leq 6$}

In this section, we assume $2\leq n \leq 6$ and we will construct $(v_{0n}, p_{0n}, \theta_{0n}, R_{0n}, f_{0n})$ by inductions.
Suppose that for $1 \leq m < n\leq 6$, we have constructed $$(v_{0m},~p_{0m},~\theta_{0m},~R_{0m},~f_{0m})\in C^\infty((a-\kappa,b+\kappa)\times T^3)$$ and they solve
  the system (\ref{d:anistropic boussinesq reynold}). Furthermore, we have
 \begin{align}\label{f:R0m f0m}
 R_{0m}=-\sum\limits_{i=m+1}^{6}(e(t)-a_{i\ell_1})k_i\otimes k_i
 +\sum\limits_{i=1}^{m}\delta R_{0i},\quad
    f_{0m}=\sum\limits_{i=m+1}^{3}c_{i\ell_1}k_i+\sum\limits_{i=1}^{m}\delta f_{0i}
\end{align}
 and
 \begin{align}\label{e:induction estimate sequence}
 &\|v_{0m}-v_{0(m-1)}\|_0\leq \frac{M\sqrt{\kappa}}{12}+C_0\sqrt{\kappa}\frac{\sqrt{\kappa}\mu\lambda_{m-1}}{\lambda_m},\quad\|(\partial_t, \partial_x,\partial_y)(v_{0m}-v_{0(m-1)})\|_0\leq C_0\sqrt{\kappa}\lambda_m,\nonumber\\
 &\|\partial_z(v_{0m}-v_{0(m-1)})\|_0\leq C_0(\sqrt{\kappa}\mu)^m\bar{\Lambda}, \quad \|\partial_z(\theta_{0m}-\theta_{0(m-1)})\|_0\leq C_0(\sqrt{\kappa}\mu)^m\bar{\Lambda}, \nonumber\\
 &\|\theta_{0m}-\theta_{0(m-1)}\|_0\leq \frac{M\sqrt{\kappa}}{6}+C_0\sqrt{\kappa}\frac{\sqrt{\kappa}\mu\lambda_{m-1}}{\lambda_m},\quad
 \|(\partial_t, \partial_x,\partial_y)(\theta_{0m}-\theta_{0(m-1)})\|_0\leq C_0\sqrt{\kappa}\lambda_m,\nonumber\\
 & \|p_{0m}-p_{0(m-1)}\|_0=0,\nonumber\\
&\|\delta R_{0m}\|_0\leq  C(\varepsilon)\Big[\kappa\lambda_{m-1}\ell_m+\sqrt{\kappa}\mu^{-1}
+\Big(\kappa\mu\lambda_{m-1}+
(\sqrt{\kappa}\mu)^{m}\bar{\Lambda}\ell_{mz}^{-1}\Big)\lambda_m^{-1}\Big],\nonumber\\
&\|(\partial_t, \partial_x,\partial_y)\delta R_{0m}\|_0\leq C(\varepsilon)\lambda_m\Big[\kappa\lambda_{m-1}\ell_m+\sqrt{\kappa}\mu^{-1}
+\Big(\kappa\mu\lambda_{m-1}+
(\sqrt{\kappa}\mu)^{m}\bar{\Lambda}\ell_{mz}^{-1}\Big)\lambda_m^{-1}\Big],\nonumber\\
&\|\partial_z\delta R_{0m}\|_0\leq  C(\varepsilon)\Big[(\sqrt{\kappa}\mu)^m\bar{\Lambda}\Big(\mu^{-1}
+\sqrt{\kappa}\lambda_{m-1}\ell_m\Big)+(\sqrt{\kappa}\mu)^{m}\bar{\Lambda}
\Big(\ell_m^{-1}+
\ell_{mz}^{-2}\Big)\lambda_m^{-1}\Big],\nonumber\\
&\|\delta f_{0m}\|_0\leq  C(\varepsilon)\Big[\kappa\lambda_{m-1}\ell_m+\sqrt{\kappa}\mu^{-1}+
\Big(\kappa\mu\lambda_{m-1}+
(\sqrt{\kappa}\mu)^m\Lambda\ell_{mz}^{-1}\Big)\lambda_m^{-1}
\Big],\nonumber\\
&\|(\partial_t, \partial_x,\partial_y)\delta f_{0m}\|_0\leq  C(\varepsilon)\lambda_m\Big[\kappa\lambda_{m-1}\ell_m+\sqrt{\kappa}\mu^{-1}+
\Big(\kappa\mu\lambda_{m-1}+
(\sqrt{\kappa}\mu)^m\Lambda\ell_{mz}^{-1}\Big)\lambda_m^{-1}
\Big],\nonumber\\
&\|\partial_z\delta f_{0m}\|_0\leq  C(\varepsilon)\Big[(\sqrt{\kappa}\mu)^m\bar{\Lambda}\Big(\mu^{-1}+\sqrt{\kappa}\lambda_{m-1}\ell_m\Big)
+\Big((\sqrt{\kappa}\mu)^m\bar{\Lambda}(\ell_{mz}^{-2}+\ell_m^{-1})
+\kappa\lambda_{m-1}\ell_{mz}^{-1}\Big)\lambda_m^{-1}\Big],
\end{align}
where $(v_{00},p_{00},~~\theta_{00}):=(v_{0},~~p_{0},~~\theta_{0})$, $\lambda_0:=\Lambda \kappa^{-1}+\bar{\Lambda}\ell_{1z}\kappa^{-1}\ell_1^{-1}$ and parameters $\ell_m, \ell_{mz}$ and $\lambda_m$  will be chosen such that
\begin{align}\label{a:assumption on parameter m sequence}
\ell_1^{-1}\geq \mu\Lambda,\quad\ell_i^{-1}\geq \sqrt{\kappa}\mu\lambda_{i-1}, i\geq2, \quad \lambda_m\geq \ell_m^{-(1+\varepsilon)},
\quad \ell_m^{-1} >\ell_{mz}^{-1}\geq\mu(\sqrt{\kappa}\mu)^{m-1}\bar{\Lambda},\quad \lambda_m,  \frac{\lambda_m}{\mu}\in{\rm N}.
\end{align}

Next, we construct n-th step by induction.

\subsection{The n-th perturbations $w_{n}$ and $\chi_n$}

   \indent

In this and subsequent part,  we always assume that
  \begin{align}\label{a:assumption on parameter n}
  \quad  \ell_n^{-1}\geq \sqrt{\kappa}\mu\lambda_{n-1},\quad \lambda_n\geq \ell_n^{-(1+\varepsilon)},\quad \ell_n^{-1}> \ell^{-1}_{nz}\geq \mu(\sqrt{\kappa}\mu)^{n-1}\bar{\Lambda},\quad
  \lambda_n, \frac{\lambda_n}{\mu}\in {\rm N} .
\end{align}
\subsubsection{Construction of n-th velocity perturbation}
   For $2\leq n \leq 6$ and any $l\in Z^3$, we denote $ b_{nl}$ by
   \begin{align}\label{e:bnl definition}
    b_{nl}:=\sqrt{\frac{e(t)-a_{n\ell_1}}{2}}\alpha_l(\mu v_{0(n-1)\ell_n})
    \end{align}
  and main $l$-perturbation $w_{nol}$ by
   \begin{align}
   w_{nol}:=b_{nl}k_n\Big(e^{i\lambda_n 2^{[l]} (k_{nh}^{\perp},0)\cdot \big((x,y,z)-\frac{l}{\mu}t\big)}+e^{-i\lambda_n 2^{[l]} (k_{nh}^{\perp},0)\cdot \big((x,y,z)-\frac{l}{\mu}t\big)}\Big).\nonumber
   \end{align}
    Here $v_{0(n-1)\ell_n}=v_{0(n-1)}\ast (\varphi_{\ell_n}(t)\varphi_{\ell_n}(x)\varphi_{\ell_n}(y)\varphi_{\ell_{nz}}(z))$.
Then we set the n-th main perturbation
   \begin{align}
    w_{no}:=\sum_{l\in Z^3}w_{nol}.\nonumber
   \end{align}
Same as in the first step, the $l$-correction is denoted by
  $w_{ncl}$ as
\begin{align}\label{d:w ncl}
w_{ncl}:=\left(
    \begin{array}{ccc}
    -t_n\partial_zb_{nl}+\partial_yb_{nl}\\
    s_n\partial_zb_{nl}-\partial_xb_{nl}\\
    t_n\partial_xb_{nl}-s_n\partial_yb_{nl}
    \end{array}
    \right)
    \Bigg(\frac{e^{i\lambda_n 2^{[l]} (k_{nh}^{\perp},0)\cdot \big((x,y,z)-\frac{l}{\mu}t\big)}}{i\lambda_n 2^{[l]}}+
    \frac{e^{-i\lambda_n 2^{[l]} (k_{nh}^{\perp},0)\cdot \big((x,y,z)-\frac{l}{\mu}t\big)}}{-i\lambda_n 2^{[l]}}\Bigg)
\end{align}
and the n-th correction is denoted by
   \begin{align}
    w_{nc}:=\sum_{l\in Z^3}w_{ncl},\nonumber
    \end{align}
where $k_n=(k_{n1}, k_{n2}, k_{n3})$ and $(s_n,t_n)\in R^2$ such that $s_nk_{n1}+t_nk_{n2}=-k_{n3}$.

A straightforward computation gives
\begin{align}
w_{nol}+w_{ncl}={\rm curl}\Big(b_{nl}\vec{a}_n\Big(\frac{e^{i\lambda_n 2^{[l]} (k_{nh}^{\perp},0)\cdot \big((x,y,z)-\frac{l}{\mu}t\big)}}{i\lambda_n 2^{[l]}}+\frac{e^{-i\lambda_n 2^{[l]} (k_{nh}^{\perp},0)\cdot \big((x,y,z)-\frac{l}{\mu}t\big)}}{-i\lambda_n2^{[l]}}\Big)\Big),\nonumber
\end{align}
where $\vec{a}_n=(s_n,t_n,1).$\\
Finally we set n-th  perturbation
   \begin{align*}
    w_n:=&w_{no}+w_{nc}.
   \end{align*}
  It's obvious that
   $\hbox{div}w_n=0.$
Moreover, they are all real vector-valued functions and
\begin{align}\label{b:bound n2}
\|w_{no}\|_0\leq\frac{M\sqrt{\kappa}}{12}.
\end{align}
We set
\begin{align} \label{d:definition on n velocity}
g_{nl}:=b_{nl}k_n+\frac{1}{i\lambda_n2^{[l]}}\left(
    \begin{array}{ccc}
    -t_n\partial_zb_{nl}+\partial_yb_{nl}\\
    s_n\partial_zb_{nl}-\partial_xb_{nl}\\
    t_n\partial_xb_{nl}-s_n\partial_yb_{nl}
    \end{array}
    \right),
g_{-nl}:=b_{nl}k_n+\frac{1}{-i\lambda_n2^{[l]}}\left(
    \begin{array}{ccc}
    -t_n\partial_zb_{nl}+\partial_yb_{nl}\\
    s_n\partial_zb_{nl}-\partial_xb_{nl}\\
    t_n\partial_xb_{nl}-s_n\partial_yb_{nl}
    \end{array}
    \right),
\end{align}
thus
\begin{align}\label{r:another representation on wnl}
w_{nl}=g_{nl}e^{i\lambda_n 2^{[l]} (k_{nh}^{\perp},0)\cdot \big((x,y,z)-\frac{l}{\mu}t\big)}+g_{-nl}e^{-i\lambda_n 2^{[l]} (k_{nh}^{\perp},0)\cdot \big((x,y,z)-\frac{l}{\mu}t\big)}.
\end{align}
It's obvious that $b_{nl}\in C_c^{\infty}((a-\kappa,b+\kappa)\times T^3)$ and
   $$w_{nol},w_{ncl},w_{nl},w_n\in C_c^{\infty}((a-\kappa,b+\kappa)\times T^3).$$
Then we complete the construction of n-th perturbation $w_n$.

 \subsubsection{Construction of n-th temperature perturbation}

To construct $\chi_n$, we first denote $\beta_{nl}$ by
   \begin{align}\label{d:difinition bnl}
    \beta_{nl}:=\left \{
    \begin {array}{ll}
    -\frac{c_{n\ell_1}}{\sqrt{2(e(t)-a_{n\ell_1})}}\alpha_{l}(\mu v_{0(n-1)\ell_n}), \qquad n=2,3\\
    0,\qquad n=4,...,6,
    \end{array}
    \right.
    \end{align}
and denote main $l$-perturbation $\chi_{nol}$ by
    \begin{align}
    \chi_{nol}:=
    \beta_{nl}\Big(e^{i\lambda_n 2^{[l]} (k_{nh}^{\perp},0)\cdot \big((x,y,z)-\frac{l}{\mu}t\big)}+e^{-i\lambda_n 2^{[l]} (k_{nh}^{\perp},0)\cdot \big((x,y,z)-\frac{l}{\mu}t\big)}\Big)\nonumber
    \end{align}
    and $l$-corection $\chi_{ncl}$ by
    \begin{align}
    \chi_{ncl}:=
    \frac{\nabla\beta_{nl}\cdot (k_{nh}^{\perp},0)}{i\lambda_n 2^{[l]}|k_{nh}|^2}\Big(e^{i\lambda_n 2^{[l]}(k_{nh}^{\perp},0)\cdot \big((x,y,z)-\frac{l}{\mu}t\big)}-e^{-i\lambda_n 2^{[l]} (k_{nh}^{\perp},0)\cdot \big((x,y,z)-\frac{l}{\mu}t\big)}\Big).\nonumber
    \end{align}

We denote $l$-perturbation $\chi_{nl}$ and n-th perturbation $\chi_n$ by
 \begin{align}
 \chi_{nl}:=\chi_{nol}+\chi_{ncl}\nonumber
 \end{align}
 and
   \begin{align}
   \chi_n:=\sum\limits_{l\in Z^3}\chi_{nl}.\nonumber
   \end{align}
   If we set
   \begin{align}\label{d:definition on n temperture}
   h_{nl}=\beta_{nl}+\frac{\nabla\beta_{nl}\cdot (k_{nh}^{\perp},0)}{i\lambda_n 2^{[l]}|k_{nh}|^2},\quad h_{-nl}=\beta_{nl}+\frac{\nabla\beta_{nl}\cdot (k_{nh}^{\perp},0)}{-i\lambda_n 2^{[l]}|k_{nh}|^2},
   \end{align}
   then
   \begin{align}
   \chi_n=\sum_{l\in Z^3}\Big(h_{nl}e^{i\lambda_n 2^{[l]} (k_{nh}^{\perp},0)\cdot \big((x,y,z)-\frac{l}{\mu}t\big)}+h_{-nl}e^{-i\lambda_n 2^{[l]} (k_{nh}^{\perp},0)\cdot \big((x,y,z)-\frac{l}{\mu}t\big)}\Big).\nonumber
   \end{align}
By (\ref{b:bound on decomposition coeffience}) and (\ref{d:difinition bnl}) on $\beta_{nl}$, we have
$\|\beta_{nl}\|_0\leq\frac{M\sqrt{\kappa}}{300}.$
Thus
\begin{align}\label{e:bound on n main preturbation}
  \|\chi_{n}\|_0\leq\frac{M\sqrt{\kappa}}{6}.
\end{align}

\subsection{The constructions of  $v_{0n}$, $p_{0n}$, $\theta_{0n}$, $f_{0n}$, $R_{0n}$ }

\indent

First, we denote $M_n$ by
   \begin{align}
   M_n:=&\sum\limits_{l\in Z^3}b^2_{nl}k_n\otimes k_n\Big( e^{2i\lambda_n 2^{[l]} (k_{nh}^{\perp},0)\cdot \big((x,y,z)-\frac{l}{\mu}t\big)}+e^{-2i\lambda_n 2^{[l]} (k_{nh}^{\perp},0)\cdot \big((x,y,z)-\frac{l}{\mu}t\big)}\Big)\nonumber\\
   &+\sum\limits_{l,l'\in Z^3 ,l\neq l'}w_{nol}\otimes w_{nol'}\nonumber
   \end{align}
 and
$N_n,K_n$ by
 \begin{align}
   N_n:=&\sum_{l\in Z^3}\Big[w_{nl}\otimes \Big(v_{0(n-1)\ell_n}-\frac{l}{\mu}\Big)
   +\Big(v_{0(n-1)\ell_n}-\frac{l}{\mu}\Big)\otimes w_{nl}\Big]\nonumber\\
   &+\sum_{l\in Z^3}\Big[w_{nl}\otimes \big(v_{0(n-1)}-v_{0(n-1)\ell_n}\big)
   +\big(v_{0(n-1)}-v_{0(n-1)\ell_n}\big)\otimes w_{nl}\Big],\nonumber\\
    K_n:=&
    \left \{
    \begin {array}{ll}
    \sum\limits_{l\in Z^3}\beta_{nl}b_{nl}k_n\Big(e^{2i\lambda_n 2^{[l]} (k_{nh}^{\perp},0)\cdot \big((x,y,z)-\frac{l}{\mu}t\big)}
   + e^{-2i\lambda_n 2^{[l]}(k_{nh}^{\perp},0)\cdot \big((x,y,z)-\frac{l}{\mu}t\big)}\Big)\nonumber\\
   +\sum\limits_{l,l'\in Z^3 ,l\neq l'}w_{nol}\chi_{nol'},\quad n=2,3\\
   0,\quad n=4,...,6.
    \end{array}
    \right.\nonumber
 \end{align}
 Notice that $N_{n}$ is a symmetric matrix. Then we set
   \begin{align}\label{d:difinition n sequence}
    &v_{0n}:=v_{0(n-1)}+w_n,\quad
    p_{0n}:=p_{0(n-1)},\quad
    \theta_{0n}:=
    \left \{
    \begin {array}{ll}
    \theta_{0(n-1)}+\chi_n, \quad n=2,3\nonumber\\
    \theta_{0(n-1)}, \quad n=4,...,6
    \end{array}
    \right. \nonumber\\
     &R_{0n}:=R_{0(n-1)}+2\sum\limits_{l\in Z^3}b^2_{nl}k_n\otimes k_n+\delta R_{0n},\nonumber\\
   & f_{0n}:
    =\left \{
    \begin {array}{ll}
    f_{0(n-1)}+2\sum\limits_{l\in Z^3}\beta_{nl}b_{nl}k_n+\delta f_{0n},\quad n=2,3\\
     f_{0(n-1)}+\delta f_{0n}, \quad n=4,...,6,
    \end{array}
    \right.
   \end{align}
   where
   \begin{align}
   \delta R_{0n}:=&\mathcal{R}(\hbox{div}M_n)+N_n+\mathcal{R}\Big\{\partial_tw_{n}
   +\hbox{div}\Big[\sum_{l\in Z^3}\Big(w_{nl}\otimes\frac{l}{\mu}
   +\frac{l}{\mu}\otimes w_{nl}\Big)\Big]\Big\}-\mathcal{R}(\partial_{zz}w_n)\nonumber\\
   &-\mathcal{R}(\chi_ne_3)
   +(w_{no}\otimes w_{nc}+w_{nc}\otimes w_{no}+w_{nc}\otimes w_{nc})\nonumber
   \end{align}
 and
   \begin{align}
   \delta f_{0n}
   :=\left \{
    \begin {array}{ll}
      \mathcal{G}(\hbox{div}K_n)
   +\mathcal{G}(w_n\cdot\nabla\theta_{0(n-1)\ell_n})
   +\mathcal{G}\Big(\partial_t\chi_n+\sum\limits_{l\in Z^3}\frac{l}{\mu}\cdot\nabla\chi_{nl}\Big)-\mathcal{G}(\partial_{zz}\chi_n)\nonumber\\
   +\sum\limits_{l\in Z^3}\Big(v_{0(n-1)\ell_n}-\frac{l}{\mu}\Big)\chi_{nl}+w_{nc}\chi_n+\sum\limits_{l\in Z^3}\big(v_{0(n-1)}-v_{0(n-1)\ell_n}\big)\chi_{nl}\nonumber\\
   +w_n(\theta_{0(n-1)}-\theta_{0(n-1)\ell_n}), \quad n=2,3\\
   \mathcal{G}(w_n\cdot\nabla\theta_{0(n-1)\ell_n})+w_n(\theta_{0(n-1)}-\theta_{0(n-1)\ell_n}),\quad n=4,...,6.
   \end{array}
    \right.
  \end{align}
   Here $\theta_{0(n-1)\ell_n}=\theta_{0(n-1)}\ast (\varphi_{\ell_n}(t)\varphi_{\ell_{n}}(x)\varphi_{\ell_n}(y)\varphi_{\ell_{nz}}(z))$.
   By Lemma \ref{l:reyn}, we know that $\delta R_{0n}$ is a symmetric matrix. It's obvious that
   \begin{align*}
    \hbox{div}v_{0n}=\hbox{div}v_{0(n-1)}+\hbox{div}w_{n}=0.
   \end{align*}
  By the definition of $R_{0n}$, $\delta R_{0n}$ together with the fact that $v_{0n}$, $p_{0n}$, $\theta_{0n}$ and  $(v_{0(n-1)}$, $p_{0(n-1)}$, $
    \theta_{0(n-1)}$, $R_{0(n-1)}$, $f_{0(n-1)})$ are solutions of the system (2.2), we have
\begin{align}
    \hbox{div}R_{0n}=&\hbox{div}R_{0(n-1)}+\partial_tw_n-\partial_{zz}w_n-
   \chi_ne_3\nonumber\\
    &+\hbox{div}(w_{no}\otimes w_{no}+w_{n}\otimes v_{0(n-1)}+v_{0(n-1)}\otimes w_{n}
    +w_{no}\otimes w_{nc}+w_{nc}\otimes w_{no}+w_{nc}\otimes w_{nc})\nonumber\\
    =&\partial_tv_{0(n-1)}+\hbox{div}(v_{0(n-1)}\otimes v_{0(n-1)})+\nabla p_{0(n-1)}-\partial_{zz}v_{0(n-1)}-\theta_{0(n-1)}e_3\nonumber\\
    &+\partial_tw_n-\partial_{zz}w_n-
    \chi_ne_3 +\hbox{div}(w_{no}\otimes w_{no}+w_{n}\otimes v_{0(n-1)}
    +v_{0(n-1)}\otimes w_{n}\nonumber\\
    &+w_{no}\otimes w_{nc}+w_{nc}\otimes w_{no}+w_{nc}\otimes w_{nc})\nonumber\\
    =&\partial_tv_{0n}+\hbox{div}(v_{0n}\otimes v_{0n})+\nabla p_{0n}-\partial_{zz}v_{0n}-\theta_{0n}e_3,\nonumber
    \end{align}
 where we used
 $$\fint_{T^3}w_{n}(t,x,y,z)dxdydz=0,\quad\hbox{div}(M_n)+\hbox{div}\Big(2\sum\limits_{l\in Z^3}b^2_{nl}k_n\otimes k_n\Big)=\hbox{div}(w_{no}\otimes w_{no}).$$
Furthermore, from the definition of $f_{0n}$, $ \delta f_{0n}$ as well as $v_{0n}$, $\theta_{0n}$ and that $(v_{0(n-1)}$, $p_{0(n-1)}$, $
    \theta_{0(n-1)}$, $R_{0(n-1)}$, $f_{0(n-1)})$ are solutions of the system (2.2), we have, for $n=2,3$
\begin{align}
 \hbox{div}f_{0n}=&\hbox{div}f_{0(n-1)}+\partial_{t}\chi_n-\partial_{zz}\chi_n
 +\hbox{div}(w_{no}\chi_n+w_{nc} \chi_n+v_{0(n-1)} \chi_n+w_n \theta_{0(n-1)})\nonumber\\
 =&\hbox{div}(v_{0(n-1)}\theta_{0(n-1)}+w_{no}\chi_n+w_{nc} \chi_n+v_{0(n-1)} \chi_n+w_n\theta_{0(n-1)})\nonumber\\
 &+\partial_t\big(\theta_{0(n-1)}+\chi_n\big)
 -\partial_{zz}\big(\theta_{0(n-1)}+\chi_n\big)\nonumber\\
 =&\partial_t\theta_{0n}+\hbox{div}(v_{0n}\theta_{0n})-\partial_{zz}\theta_{0n}\nonumber
 \end{align}
and for $n=4,5,6$
 \begin{align}
 \hbox{div}f_{0n}=\hbox{div}f_{0(n-1)}+w_{n}\cdot\nabla\theta_{0(n-1)}
 =&\partial_t\theta_{0(n-1)}+\hbox{div}(v_{0(n-1)}\theta_{0(n-1)}+w_n\theta_{0(n-1)})-\partial_{zz}\theta_{03}\nonumber\\
 =&\partial_t\theta_{0n}+\hbox{div}(v_{0n}\theta_{0n})-\partial_{zz}\theta_{0n}.\nonumber
 \end{align}
Thus, the functions $(v_{0n},~p_{0n},~\theta_{0n},~R_{0n},~f_{0n})$ solves the system (2.2).

\setcounter{equation}{0}

\section{The representations}

\indent

In this section, we will compute the following terms
$$R_{0(n-1)}+2\sum\limits_{l\in Z^3}b^2_{nl}k_n\otimes k_n,\quad {\rm and}\quad f_{0n}+2\sum\limits_{l\in Z^3}\beta_{nl}b_{nl}k_n.$$

\subsection{The representation of $R_{0(n-1)}+2\sum\limits_{l\in Z^3}b^2_{nl}k_n\otimes k_n$}

\indent

First, by the definition (\ref{e:bnl definition}) on $b_{nl}$, we have
\begin{align}
2\sum\limits_{l\in Z^3}b^2_{nl}k_n\otimes k_n
=(e(t)-a_{n\ell_1})k_n\otimes k_n.\nonumber
\end{align}
By (\ref{f:R0m f0m}), we obtain
\begin{align}
R_{0(n-1)}+2\sum\limits_{l\in Z^3}b^2_{nl}k_n\otimes k_n
=-\sum\limits_{i=n+1}^{6}(e(t)-a_{i\ell_1})k_i\otimes k_i
 +\sum\limits_{i=1}^{n-1}\delta R_{0i}.\nonumber
\end{align}
Meanwhile, we have
\begin{align}
R_{0n}=-\sum\limits_{i=n+1}^{6}(e(t)-a_{i\ell_1})k_i\otimes k_i
 +\sum\limits_{i=1}^{n}\delta R_{0i}.
\end{align}
In particular,
$R_{06}=\sum\limits_{i=1}^{6}\delta R_{0i}.$
In next section, we will prove that $\delta R_{0n}$ is small.
\subsection{The representation of $f_{0n}+2\sum\limits_{l\in Z^3}\beta_{nl}b_{nl}k_n$}

\indent

By the definition (\ref{e:bnl definition}) on $b_{nl}$ and (\ref{d:difinition bnl}) on $\beta_{nl}$, we have, for $n=2,3$
\begin{align}
2\sum\limits_{l\in Z^3}\beta_{nl}b_{nl}k_n
=\sum\limits_{l\in Z^3}\alpha_l^2(\mu v_{0(n-1)\ell_n})c_{n\ell_1}k_n
=-c_{n\ell_1}k_n.\nonumber
\end{align}
Thus, by (\ref{f:R0m f0m}), we have
\begin{align}
f_{0n}+2\sum\limits_{l\in Z^3}\beta_{nl}b_{nl}k_n
=\sum\limits_{i=1}^n \delta f_{0i}+\sum\limits_{i=n+1}^3 c_{i\ell_1}k_i, \quad n=2,3.\nonumber
\end{align}
In particular,
$f_{03}=\sum\limits_{i=1}^{3}\delta f_{0n}.$
Finally, by (\ref{d:difinition n sequence}), we have
\begin{align}
f_{0n}
=\sum\limits_{i=1}^{n}\delta f_{0i},\quad n=4,5,6.\nonumber
\end{align}

\setcounter{equation}{0}

 \section{Estimates on $\delta R_{0n}$ and $\delta f_{0n}$}

 \indent


 First, we summarize some estimates on $b_{nl}$ and $\beta_{nl}$.
\begin{Lemma}\label{e:estimate various order n}
For any $l\in Z^3$ and integer $r\geq 1$, we have the following estimates: for any $t> 0$
\begin{align}\label{e:estimate various higher order n}
\|b_{nl}(t,\cdot)\|_r+\|\beta_{nl}(t,\cdot)\|_r\leq& C_r\kappa\mu\lambda_{n-1}\ell_n^{-(r-1)},\nonumber\\
\|\partial_{z}b_{nl}(t,\cdot)\|_r+\|\partial_{z}\beta_{nl}(t,\cdot)\|_r\leq& C_r(\sqrt{\kappa}\mu)^{n}\bar{\Lambda}\ell_n^{-r},\nonumber\\
\|\partial_{zz}b_{nl}(t,\cdot)\|_r+\|\partial_{zz}\beta_{nl}(t,\cdot)\|_r\leq& C_r(\sqrt{\kappa}\mu)^{n}\bar{\Lambda}\ell_{nz}^{-1}\ell_n^{-r},\nonumber\\
\|\partial_{zzz}b_{nl}(t,\cdot)\|_r+\|\partial_{zzz}\beta_{nl}(t,\cdot)\|_r\leq& C_r(\sqrt{\kappa}\mu)^{n}\bar{\Lambda}\ell_{nz}^{-2}\ell_n^{-r},\nonumber\\
\|\partial_tb_{nl}(t,\cdot)\|_r+\|\partial_t\beta_{nl}(t,\cdot)\|_r\leq& C_r\kappa\mu\lambda_{n-1}\ell_n^{-r},\nonumber\\
\|\partial_{tt}b_{nl}(t,\cdot)\|_r+\partial_{tt}\beta_{nl}(t,\cdot)\|_r\leq& C_r\kappa\mu\lambda_{n-1}\ell_n^{-(r+1)}
\end{align}
and
\begin{align}\label{e:global estimate various higher order n}
\|g_{\pm nl}(t,\cdot)\|_r+\|h_{\pm nl}(t,\cdot)\|_r\leq& C_r\kappa\mu\lambda_{n-1}\ell_n^{-(r-1)},\nonumber\\
\|\partial_zg_{\pm nl}(t,\cdot)\|_r+\|\partial_{z}h_{\pm nl}(t,\cdot)\|_r\leq& C_r(\sqrt{\kappa}\mu)^{n}\bar{\Lambda}\ell_n^{-r},\nonumber\\
\|\partial_tg_{\pm nl}(t,\cdot)\|_r+\|\partial_th_{\pm nl}(t,\cdot)\|_r\leq& C_r\kappa\mu\lambda_{n-1}\ell_n^{-r},\nonumber\\
\|\partial_{tt}g_{\pm nl}(t,\cdot)\|_r+\|\partial_{tt}h_{\pm nl}(t,\cdot)\|_r\leq& C_r\kappa\mu\lambda_{n-1}\ell_n^{-(r+1)},\nonumber\\
\|\partial_{zz}g_{\pm nl}(t,\cdot)\|_r+\|\partial_{zz}h_{\pm nl}(t,\cdot)\|_r\leq & C_r(\sqrt{\kappa}\mu)^{n}\bar{\Lambda}\ell_{nz}^{-1}\ell_n^{-r},\nonumber\\
\|\partial_{zzz}g_{\pm nl}(t,\cdot)\|_r+\|\partial_{zzz}h_{\pm nl}(t,\cdot)\|_r\leq &C_r(\sqrt{\kappa}\mu)^{n}\bar{\Lambda}\ell_{nz}^{-2}\ell_n^{-r}.
\end{align}
Moreover,  for any $t> 0$
\begin{align}\label{e:zero estimate various higher order n}
\|b_{nl}(t,\cdot)\|_0+\|\beta_{nl}(t,\cdot)\|_0\leq& C_0\sqrt{\kappa},\nonumber\\
\|\partial_zb_{nl}(t,\cdot)\|_0+\|\partial_z\beta_{nl}(t,\cdot)\|_0\leq& C_0(\sqrt{\kappa}\mu)^n\bar{\Lambda},\nonumber\\
\|\partial_{zz}b_{nl}(t,\cdot)\|_0+\|\partial_{zz}\beta_{nl}(t,\cdot)\|_0\leq& C_0(\sqrt{\kappa}\mu)^{n}\bar{\Lambda}\ell_{nz}^{-1},\nonumber\\
\|\partial_{zzz}b_{nl}(t,\cdot)\|_0+\|\partial_{zzz}\beta_{nl}(t,\cdot)\|_0\leq& C_0(\sqrt{\kappa}\mu)^{n}\bar{\Lambda}\ell_{nz}^{-2},\nonumber\\
\|\partial_tb_{nl}(t,\cdot)\|_0+\|\partial_t\beta_{nl}(t,\cdot)\|_0\leq& C_0\kappa\mu\lambda_{n-1},\nonumber\\
\|\partial_{tt}b_{nl}(t,\cdot)\|_0+\partial_{tt}\beta_{nl}(t,\cdot)\|_0\leq& C_0\kappa\mu\lambda_{n-1}\ell_n^{-1}
\end{align}
and
\begin{align}\label{e:global zero estimate various higher order n}
\|g_{\pm nl}(t,\cdot)\|_0+\|h_{\pm nl}(t,\cdot)\|_0\leq& C_0\sqrt{\kappa},\nonumber\\
\|\partial_zg_{\pm nl}(t,\cdot)\|_0+\|\partial_{z}h_{\pm nl}(t,\cdot)\|_0\leq& C_0(\sqrt{\kappa}\mu)^n\bar{\Lambda},\nonumber\\
\|\partial_tg_{\pm nl}(t,\cdot)\|_0+\|\partial_th_{\pm nl}(t,\cdot)\|_0\leq& C_0\kappa\mu\lambda_{n-1},\nonumber\\
\|\partial_{tt}g_{\pm nl}(t,\cdot)\|_0+\|\partial_{tt}h_{\pm nl}(t,\cdot)\|_0\leq& C_0\kappa\mu\lambda_{n-1}\ell_n^{-1},\nonumber\\
\|\partial_{zz}g_{\pm nl}(t,\cdot)\|_0+\|\partial_{zz}h_{\pm nl}(t,\cdot)\|_0\leq& C_0(\sqrt{\kappa}\mu)^{n}\bar{\Lambda}\ell_{nz}^{-1},\nonumber\\
\|\partial_{zzz}g_{\pm nl}(t,\cdot)\|_0+\|\partial_{zzz}h_{\pm nl}(t,\cdot)\|_0\leq& C_0(\sqrt{\kappa}\mu)^{n}\bar{\Lambda}\ell_{nz}^{-2}.
\end{align}
\begin{proof}
The proof is similar to that of Lemma 6.1.
First, by (\ref{i:composition inequality 2}), for any $r\geq 2$, we obtain
\begin{align}
\Big[\sqrt{e(t)-a_{n\ell_1}(t,\cdot)}\Big]_r\leq& C_r\sqrt{\kappa}(\mu^r\Lambda^r+\mu\Lambda\ell_1^{-(r-1)})\leq C_r\sqrt{\kappa}\mu\Lambda\ell_n^{-(r-1)}\nonumber
\end{align}
and for $r=0,1$
\begin{align}
\big[\sqrt{e(t)-a_{n\ell_1}(t,\cdot)}\big]_r\leq& C_0\sqrt{\kappa}\mu^r\Lambda^r.\nonumber
\end{align}
By (\ref{e:induction estimate sequence}), we know that $\|v_{0(n-1)}\|_{1}\leq C_0\sqrt{\kappa}\lambda_{n-1}$, thus, by (\ref{i:composition inequality 2}), for any $r\geq 2$
\begin{align}\label{e:estimate r derivative on cut off function}
\big[\alpha_l(\mu v_{0(n-1)\ell_n})(t,\cdot)\big]_r\leq & C_r\sum\limits_{i=1}^r\|(\nabla^{i}\alpha)_l\|_0[\mu v_{0(n-1)\ell_n}(t,\cdot))]_1^{(i-1)\frac{r}{r-1}}[\mu v_{0(n-1)\ell_n}(t,\cdot))]_r^{\frac{r-i}{r-1}}\nonumber\\
\leq & C_r((\sqrt{\kappa}\mu\lambda_{n-1})^r+\sqrt{\kappa}\mu\lambda_{n-1}\ell_n^{-(r-1)})\leq  C_r\sqrt{\kappa}\mu\lambda_{n-1}\ell_n^{-(r-1)}
\end{align}
and for $r=0,1$
\begin{align}
\big[\alpha_l(\mu v_{0(n-1)\ell_n})(t,\cdot)\big]_r\leq C_0(\sqrt{\kappa}\mu\lambda_{n-1})^r.\nonumber
\end{align}
As in the proof of lemma \ref{e:estimate various}, we can obtain
for any integer $r\geq 0$
\begin{align}
\Big\|\partial_z\Big(\sqrt{e(t)-a_{n\ell_1}(t,\cdot)}\Big)\Big\|_r\leq& C_r\sqrt{\kappa}\mu\bar{\Lambda}\ell_1^{-r},\nonumber\\
\Big\|\partial_{zz}\Big(\sqrt{e(t)-a_{n\ell_1}(t,\cdot)}\Big)\Big\|_r\leq & C_r\sqrt{\kappa}\mu\bar{\Lambda}\ell^{-1}_{1z}\ell_1^{-r},\nonumber\\
\Big\|\partial_{zzz}\Big(\sqrt{e(t)-a_{n\ell_1}(t,\cdot)}\Big)\Big\|_r\leq&  C_r\sqrt{\kappa}\mu\bar{\Lambda}\ell^{-2}_{1z}\ell_1^{-r}.\nonumber
\end{align}
By (\ref{e:induction estimate sequence}), we know that $\|\partial_zv_{0(n-1)}\|_0\leq C_0(\sqrt{\kappa}\mu)^{n-1}\bar{\Lambda}$, thus by inequality  (\ref{i:inequality 1}), assumption (\ref{a:assumption on parameter n}) and a similar argument as in (\ref{e:estimate on x derivative}), we have, for any integer $r\geq1$
\begin{align}
\big[\partial_z(\alpha_l(\mu v_{0(n-1)\ell_n})(t,\cdot))\big]_r\leq& C_r\mu(\sqrt{\kappa}\mu)^{n-1}\bar{\Lambda}\ell_n^{-r},\nonumber\\
\big[\partial_{zz}(\alpha_l(\mu v_{0(n-1)\ell_n})(t,\cdot))\big]_r\leq& C_r\mu(\sqrt{\kappa}\mu)^{n-1}\bar{\Lambda}\ell_{nz}^{-1}\ell_n^{-r},\nonumber\\
\big[\partial_{zzz}(\alpha_l(\mu v_{0(n-1)\ell_n})(t,\cdot))\big]_r
 \leq& C_r\mu(\sqrt{\kappa}\mu)^{n-1}\bar{\Lambda}\ell_{nz}^{-2}\ell_n^{-r}.\nonumber
\end{align}
A straightforward computation gives
\begin{align}
\|\partial_z(\alpha_l(\mu v_{0(n-1)\ell_n})\|_0\leq& C_0\mu(\sqrt{\kappa}\mu)^{n-1}\bar{\Lambda},\nonumber\\
\|\partial_{zz}(\alpha_l(\mu v_{0(n-1)\ell_n})\|_0\leq& C_0\mu(\sqrt{\kappa}\mu)^{n-1}\bar{\Lambda}\ell_{nz}^{-1},\nonumber\\
\|\partial_{zzz}(\alpha_l(\mu v_{0(n-1)\ell_n})\|_0\leq& C_0\mu(\sqrt{\kappa}\mu)^{n-1}\bar{\Lambda}\ell_{nz}^{-2}.\nonumber
\end{align}
By inequality (\ref{i:inequality 1}) and assumption (\ref{a:assumption on parameter n}), we can obtain the first four estimates on $b_{nl}$ in
(\ref{e:estimate various higher order n}) and (\ref{e:zero estimate various higher order n}). Similarly, we can obtain the first four estimates on $\beta_{nl}$ in
(\ref{e:estimate various higher order n}) and (\ref{e:zero estimate various higher order n}).

We let
$$\Gamma_n=\sqrt{\frac{e(t)-a_{n\ell_1}}{2}}$$
and observe that
$$b_{nl}=\Gamma_n\alpha_l(\mu v_{0(n-1)\ell_n}).$$
Differentiating in time on $\Gamma_n$,
we have, for any $r\geq 0$
\begin{align}\label{e:estimate r derivative and time}
\|\partial_t\Gamma(t,\cdot)\|_r\leq C_r\kappa\mu\lambda_{n-1}\ell_n^{-r},\quad
\|\partial_{tt}\Gamma(t,\cdot)\|_r\leq C_r\kappa\mu\lambda_{n-1}\ell_n^{-(r+1)}.
\end{align}
Thus, by (\ref{e:estimate r derivative on cut off function}), (\ref{e:estimate r derivative and time}) and $\|\partial_tv_{0(n-1)}\|_0\leq \sqrt{\kappa}\lambda_{n-1}$, we obtain, for any $r\geq 0$
\begin{align}
\|\partial_tb_{nl}(t,\cdot)\|_r\leq& C_r\kappa\mu\lambda_{n-1}\ell_n^{-r},\quad
\|\partial_{tt}b_{nl}(t,\cdot)\|_r\leq C_r\kappa\mu\lambda_{n-1}\ell_n^{-r-1}.\nonumber
\end{align}
Similarly, we obtain, for any $r\geq 0$
\begin{align}
\|\partial_t\beta_{nl}(t,\cdot)\|_r\leq& C_r\kappa\mu\lambda_{n-1}\ell_n^{-r},\quad
\|\partial_{tt}\beta_{nl}(t,\cdot)\|_r\leq C_r\kappa\mu\lambda_{n-1}\ell_n^{-r-1}.\nonumber
\end{align}
Then we obtain the later two estimates in
(\ref{e:estimate various higher order n}) and (\ref{e:zero estimate various higher order n}). From the definition (\ref{d:definition on n velocity}) on $g_{nl}$ and the definition (\ref{d:definition on n temperture}) on $h_{nl}$, we deduce (\ref{e:global estimate various higher order n}) and (\ref{e:global zero estimate various higher order n}) directly from  (\ref{e:estimate various higher order n}) and (\ref{e:zero estimate various higher order n}). Thus, the proof of this lemma is complete.
\end{proof}
\end{Lemma}

 From the definition of $w_{no}, w_{nc}, \chi_{no}, \chi_{nc}$ and the above lemma, we have the following estimates.
 \begin{Lemma}[Estimates on the n-th main perturbation and correction]\label{e:n main perturbation and corection estimate}
 \begin{align}\label{e:n error estimate}
 &\|w_{no}\|_0\leq C_0\sqrt{\kappa},\quad  \|(\partial_t, \partial_x, \partial_y)w_{no}\|_0\leq C_0\sqrt{\kappa}\lambda_n,\quad \|\partial_zw_{no}\|_0\leq C_0(\sqrt{\kappa}\mu)^{n}\bar{\Lambda},\nonumber\\
 &\|w_{nc}\|_0\leq C_0\kappa\mu\lambda_{n-1}\lambda_n^{-1}, \quad \|(\partial_t, \partial_x, \partial_y)w_{nc}\|_0\leq C_0\kappa\mu\lambda_{n-1}, \quad \|\partial_zw_{nc}\|_0\leq C_0(\sqrt{\kappa}\mu)^n\bar{\Lambda}\ell_n^{-1}\lambda_n^{-1},
 \end{align}
 and for $n=2,3$
 \begin{align}\label{e:n error tempeture}
  \|\chi_{no}\|_0\leq& C_0\sqrt{\kappa},\quad  \|(\partial_t,\partial_x, \partial_y)\chi_{no}\|_0\leq C_0\sqrt{\kappa}\lambda_n,\quad  \|\partial_z\chi_{no}\|_0\leq C_0(\sqrt{\kappa}\mu)^n\bar{\Lambda},\nonumber\\
   \|\chi_{nc}\|_0\leq& C_0\kappa\mu\lambda_{n-1}\lambda_n^{-1},\quad  \|(\partial_t, \partial_x, \partial_y)\chi_{nc}\|_0\leq C_0\kappa\mu\lambda_{n-1},\quad \|\partial_z\chi_{nc}\|_0\leq  C_0(\sqrt{\kappa}\mu)^n\bar{\Lambda}\ell_n^{-1}\lambda_n^{-1}.
\end{align}
\begin{proof}
By (\ref{b:bound n2}), (\ref{e:bound on n main preturbation}), lemma \ref{e:estimate various order n} and assumption (\ref{a:assumption on parameter n}), the proof is similar to that of Lemma \ref{e: estimate correction}, we omit it here.
 \end{proof}
 \end{Lemma}
From the above lemma and (\ref{b:bound n2}), (\ref{e:bound on n main preturbation}), (\ref{d:difinition n sequence}), we also conclude that
\begin{Corollary}\label{e:n sequence difference estiamte}
For $2\leq n\leq6$
\begin{align}
&\|v_{0n}-v_{0(n-1)}\|_0\leq \frac{M\sqrt{\kappa}}{12}+C_0\frac{\kappa\mu\lambda_{n-1}}{\lambda_n},\quad\|(\partial_t, \partial_x, \partial_y)(v_{0n}-v_{0(n-1)})\|_0\leq C_0\sqrt{\kappa}\lambda_n,\nonumber\\
&\|\partial_z(v_{0n}-v_{0(n-1)})\|_0\leq C_0(\sqrt{\kappa}\mu)^{n}\bar{\Lambda},\quad \|p_{0n}-p_{0(n-1)}\|_0=0\nonumber
\end{align}
and for $n=2,3$
\begin{align}
& \|\theta_{0n}-\theta_{0{n-1}}\|_0\leq \frac{M\sqrt{\kappa}}{6}+C_0\frac{\kappa\mu\lambda_{n-1}}{\lambda_n},\quad
 \|(\partial_t, \partial_x, \partial_y)(\theta_{0n}-\theta_{0(n-1)})\|_0\leq C_0\sqrt{\kappa}\lambda_n,\nonumber\\
& \|\partial_z(\theta_{0n}-\theta_{0(n-1)})\|_0\leq C_0(\sqrt{\kappa}\mu)^n\bar{\Lambda}.\nonumber
\end{align}
\end{Corollary}

 \subsection{Estimates on $\delta R_{0n}$}

 \indent

 As before, we again split $\delta R_{0n}$ into three parts:  \\
  (1) the oscillation part $$\mathcal{R}(\hbox{div}M_n)-\mathcal{R}(\chi_ne_3)-\mathcal{R}(\partial_{zz}w_n),$$
  (2) the transportation part
  \begin{align}
  &\mathcal{R}\Big\{\partial_tw_{n}
   +\hbox{div}\Big[\sum_{l\in Z^3}\Big(w_{nl}\otimes\frac{l}{\mu}+\frac{l}{\mu}\otimes w_{nl}\Big)\Big]\Big\}=\mathcal{R}\Big(\partial_tw_{n}+\sum_{l\in Z^3}\frac{l}{\mu}\cdot\nabla w_{nl}\Big),\nonumber
   \end{align}
  (3) the error part
  \begin{align}
  &N_n+(w_{no}\otimes w_{nc}+w_{nc}\otimes w_{no}+w_{nc}\otimes w_{nc}).\nonumber
   \end{align}
As before, we will estimate each term separately.

\begin{Lemma}[The oscillation part]
\begin{align}
\|\mathcal{R}({\rm div}M_n)\|_0\leq& C_0(\varepsilon)\kappa (\sqrt{\kappa}\mu\lambda_{n-1})\lambda_n^{-1},\quad \|(\partial_t, \partial_x,\partial_y)\mathcal{R}({\rm div}M_n)\|_0\leq C_0(\varepsilon)\kappa (\sqrt{\kappa}\mu\lambda_{n-1}),\nonumber\\
\|\partial_z\mathcal{R}({\rm div}M_n)\|_0\leq& C_0(\varepsilon)\sqrt{\kappa} (\sqrt{\kappa}\mu)^n\bar{\Lambda}\ell_n^{-1}\lambda_n^{-1},\nonumber\\
\big\|\mathcal{R}(\chi_ne_3)\big\|_0\leq& C_0(\varepsilon)\sqrt{\kappa}\lambda_n^{-1},\quad
\big\|(\partial_t,\partial_x,\partial_y)\mathcal{R}(\chi_ne_3)\big\|_0\leq C_0(\varepsilon)\sqrt{\kappa},\nonumber\\
\big\|\partial_z\mathcal{R}(\chi_ne_3)\big\|_0\leq& C_0(\varepsilon)(\sqrt{\kappa}\mu)^n\bar{\Lambda}\lambda_n^{-1},\nonumber\\
\big\|\mathcal{R}(\partial_{zz}w_n)\big\|_0\leq& C_0(\varepsilon)(\sqrt{\kappa}\mu)^{n}\bar{\Lambda}\ell_{nz}^{-1}\lambda_n^{-1},\quad
\big\|(\partial_t, \partial_x,\partial_y)\mathcal{R}(\partial_{zz}w_n)\big\|_0\leq C_0(\varepsilon)(\sqrt{\kappa}\mu)^{n}\bar{\Lambda}\ell_{nz}^{-1},\nonumber\\
\big\|\partial_z\mathcal{R}(\partial_{zz}w_n)\big\|_0\leq& C_0(\varepsilon)(\sqrt{\kappa}\mu)^{n}\bar{\Lambda}\ell_{nz}^{-2}\lambda_n^{-1}.\nonumber
\end{align}
\end{Lemma}

\begin{Lemma}[The transportation part]
\begin{align}
&\Big\|\mathcal{R}\Big(\partial_tw_{n}+\sum_{l\in Z^3}\frac{l}{\mu}\cdot\nabla w_{nl}\Big)\Big\|_0\leq C_0(\varepsilon)\kappa\mu\lambda_{n-1}\lambda_n^{-1},\nonumber\\
&\Big\|(\partial_t, \partial_x,\partial_y)\mathcal{R}\Big(\partial_tw_{n}+\sum_{l\in Z^3}\frac{l}{\mu}\cdot\nabla w_{nl}\Big)\Big\|_0\leq  C_0(\varepsilon)\kappa\mu\lambda_{n-1},\nonumber\\
&\Big\|\partial_z\mathcal{R}\Big(\partial_tw_{n}+\sum_{l\in Z^3}\frac{l}{\mu}\cdot\nabla w_{nl}\Big)\Big\|_0\leq  C_0(\varepsilon)(\sqrt{\kappa}\mu)^n\bar{\Lambda}\ell_n^{-1}\lambda_n^{-1}.\nonumber
\end{align}
\end{Lemma}

\begin{Lemma}[Estimates on error part I]\label{e:n step error difference}
\begin{align}
\|N_n\|_0\leq& C_0\sqrt{\kappa}\mu^{-1},\quad \|(\partial_t,\partial_x,\partial_y)N_n\|_0\leq  C_0\lambda_n\sqrt{\kappa}\mu^{-1},\quad
\|\partial_zN_n\|_0\leq  C_0(\sqrt{\kappa}\mu)^n\bar{\Lambda}\mu^{-1}.\nonumber
\end{align}
\end{Lemma}

\begin{Lemma}[Estimates on error part II]
\begin{align}
\|w_{no}\otimes w_{nc}+w_{nc}\otimes w_{no}+w_{nc}\otimes w_{nc}\|_0\leq& C_0\kappa\sqrt{\kappa}\mu\lambda_{n-1}\lambda_n^{-1},\nonumber\\
\|(\partial_t, \partial_x,\partial_y)(w_{no}\otimes w_{nc}+w_{nc}\otimes w_{no}+w_{nc}\otimes w_{nc})\|_0\leq& C_0\kappa\sqrt{\kappa}\mu\lambda_{n-1},\nonumber\\
\|\partial_z(w_{no}\otimes w_{nc}+w_{nc}\otimes w_{no}+w_{nc}\otimes w_{nc})\|_0\leq& C_0\sqrt{\kappa}(\sqrt{\kappa}\mu)^n\bar{\Lambda}\ell_n^{-1}\lambda_n^{-1}.\nonumber
\end{align}
\end{Lemma}

By Lemma \ref{e:estimate various order n}, the proof of the above four lemmas except Lemma \ref{e:n step error difference} are similar to that of Lemma \ref{e:oscillation estimate}, Lemma \ref{e:transport estimate}, Lemma \ref{e: error 2} respectively,  we omit it here and only give a proof of Lemma \ref{e:n step error difference}.
\begin{proof}[Proof of Lemma \ref{e:n step error difference}]
Recall that $N_n=N_{n1}+N_{n2}$, where
\begin{align}
N_{n1}:=&\sum_{l\in Z^3}\Big[w_{nl}\otimes \Big(v_{0(n-1)\ell_n}-\frac{l}{\mu}\Big)
   +\Big(v_{0(n-1)\ell_n}-\frac{l}{\mu}\Big)\otimes w_{nl}\Big],\nonumber\\
N_{n2}:=&\sum_{l\in Z^3}\Big[w_{nl}\otimes \big(v_{0(n-1)}-v_{0(n-1)\ell_n}\big)
   +\big(v_{0(n-1)}-v_{0(n-1)\ell_n}\big)\otimes w_{nl}\Big].\nonumber
\end{align}
Since $w_{nl}\neq 0$ implies $|\mu v_{0(n-1)\ell_n}-l|<1$, then
$$\|N_{n1}\|_0\leq C_0\sqrt{\kappa}\mu^{-1}.$$
By (\ref{e:induction estimate sequence}), we obtain $\|(\partial_t, \partial_x,\partial_y)v_{0(n-1)}\|_0\leq C_0\sqrt{\kappa}\lambda_{n-1}$ and $\|\partial_zv_{0(n-1)}\|_0\leq C_0(\sqrt{\kappa}\mu)^{n-1}\bar{\Lambda}$. By (\ref{a:assumption on parameter n}), we have
$$\|N_{n2}\|_0\leq C_0(\kappa\lambda_{n-1}\ell_n+\sqrt{\kappa}(\sqrt{\kappa}\mu)^{n-1}\bar{\Lambda}\ell_{nz})\leq C_0\sqrt{\kappa}\mu^{-1}.$$
Combining the two parts, we arrive at the first estimate of Lemma \ref{e:n step error difference}. The second estimate can be obtained by the same argument.

By (\ref{p:unity}), Lemma \ref{e:n main perturbation and corection estimate}, $|\mu v_{0(n-1)\ell_n}-l|<1$, $\|\partial_z v_{0(n-1)}\|_0\leq C_0(\sqrt{\kappa}\mu)^{n-1}\bar{\Lambda}$ and (\ref{a:assumption on parameter n}), we conclude
\begin{align}
\|\partial_zN_{n1}\|_0\leq C_0(\sqrt{\kappa}\mu)^n\bar{\Lambda}\mu^{-1}+C_0\sqrt{\kappa}(\sqrt{\kappa}\mu)^{n-1}\bar{\Lambda}\leq C_0(\sqrt{\kappa}\mu)^n\bar{\Lambda}\mu^{-1}.\nonumber
\end{align}
A similar argument on $N_{n2}$ also gives
\begin{align}
\|\partial_zN_{n2}\|_0\leq C_0\sqrt{\kappa}(\sqrt{\kappa}\mu)^{n-1}\bar{\Lambda}+ C_0 (\sqrt{\kappa}\mu)^n\bar{\Lambda}\Big(\sqrt{\kappa}\lambda_{n-1}\ell_n+(\sqrt{\kappa}\mu)^{n-1}\bar{\Lambda}\ell_{nz}\Big) \leq C_0(\sqrt{\kappa}\mu)^n\bar{\Lambda}\mu^{-1}.\nonumber
\end{align}
Then we obtain the third estimate.
Thus, the proof of this lemma is complete.
\end{proof}
Finally, we conclude
\begin{align}\label{e:n iterative stress error estimate}
\|\delta R_{0n}\|_0\leq & C_0(\varepsilon)\Big[\sqrt{\kappa}\mu^{-1}
+\lambda_n^{-1}\Big(\kappa\mu\lambda_{n-1}+
(\sqrt{\kappa}\mu)^{n}\bar{\Lambda}\ell_{nz}^{-1}\Big)\Big],\nonumber\\
\|(\partial_t, \partial_x,\partial_y)\delta R_{0n}\|_0\leq &C_0(\varepsilon)\lambda_n\Big[\sqrt{\kappa}\mu^{-1}
+\lambda_n^{-1}\Big(\kappa\mu\lambda_{n-1}+
(\sqrt{\kappa}\mu)^{n}\bar{\Lambda}\ell_{nz}^{-1}\Big)\Big],\nonumber\\
\|\partial_z\delta R_{0n}\|_0\leq&  C_0(\varepsilon)\Big[(\sqrt{\kappa}\mu)^n\bar{\Lambda}\mu^{-1}
+(\sqrt{\kappa}\mu)^{n}\bar{\Lambda}
\Big(\ell_n^{-1}+
\ell_{nz}^{-2}\Big)\lambda_n^{-1}\Big].
 \end{align}

\subsection{Estimates on $\delta f_{0n}$}

 \indent


As before, for $n=2,3$, we split $\delta f_{0n}$ into three parts:\\
(1) the Oscillation Part:
\begin{align}
\mathcal{G}(\hbox{div}K_n)+\mathcal{G}(w_n\cdot\nabla\theta_{0{n-1}\ell_n})-
\mathcal{G}(\partial_{zz}\chi_n).\nonumber
\end{align}
(2) the transportation part:
\begin{align}
\mathcal{G}\Big(\partial_t\chi_n+\sum_{l\in Z^3}\frac{l}{\mu}\cdot\nabla\chi_{nl}\Big).\nonumber
\end{align}
(3) the error part:
\begin{align}
w_{nc}\chi_n+\sum_{l\in Z^3}\Big(v_{0{n-1}\ell_n}-\frac{l}{\mu}\Big)\chi_{nl}+\sum_{l\in Z^3}\big(v_{0(n-1)}-v_{0{n-1}\ell_n}\big)\chi_{nl}+w_n(\theta_{0{n-1}}-
\theta_{0{n-1}\ell_n}).\nonumber
\end{align}
We estimate them term by term.

\begin{Lemma}[The Oscillation Part]
\begin{align}
\|\mathcal{G}({\rm div}K_n)\|_0\leq& C_0(\varepsilon)\kappa (\sqrt{\kappa}\mu\lambda_{n-1})\lambda_n^{-1},\quad \|(\partial_t, \partial_x,\partial_y)\mathcal{G}({\rm div}K_n)\|_0\leq C_0(\varepsilon)\kappa (\sqrt{\kappa}\mu\lambda_{n-1}),\nonumber\\
\|\partial_z\mathcal{G}({\rm div}K_n)\|_0
\leq &C_0(\varepsilon)\sqrt{\kappa} (\sqrt{\kappa}\mu)^n\bar{\Lambda}\ell_n^{-1}\lambda_n^{-1},\nonumber\\
\|\mathcal{G}(w_n\cdot\nabla\theta_{0{n-1}\ell_n})\|_0\leq& C_0(\varepsilon)\kappa\lambda_{n-1}\lambda_n^{-1},\quad \|(\partial_t,  \partial_x,\partial_y)\mathcal{G}(w_n\cdot\nabla\theta_{0{n-1}\ell_n})\|_0\leq C_0(\varepsilon)\kappa\lambda_{n-1},\nonumber\\
 \|\partial_z\mathcal{G}(w_n\cdot\nabla\theta_{0{n-1}\ell_n})\|_0\leq& C_0(\varepsilon)\kappa\lambda_{n-1}\ell_{nz}^{-1}\lambda_n^{-1},\nonumber\\
\|\mathcal{G}(\partial_{zz}\chi_n)\|_0\leq& C_0(\varepsilon)(\sqrt{\kappa}\mu)^n\bar{\Lambda}\ell_{nz}^{-1}\lambda_n^{-1},\quad
\|(\partial_t, \partial_x,\partial_y)\mathcal{G}(\partial_{zz}\chi_n)\|_0\leq C_0(\varepsilon)(\sqrt{\kappa}\mu)^n\bar{\Lambda}\ell_{nz}^{-1},\nonumber\\
\|\partial_z\mathcal{G}(\partial_{zz}\chi_n)\|_0\leq& C_0(\varepsilon)(\sqrt{\kappa}\mu)^n\bar{\Lambda}\ell_{nz}^{-2}\lambda_n^{-1}.\nonumber
\end{align}
\end{Lemma}

\begin{Lemma}[The transportation Part]
For $n=2,3$, we have
\begin{align}
&\Big\|\mathcal{G}\Big(\partial_t\chi_n+\sum_{l\in Z^3}\frac{l}{\mu}\cdot\nabla\chi_{nl}\Big)\Big\|_0\leq C_0(\varepsilon)\kappa\mu\lambda_{n-1}\lambda_n^{-1},\nonumber\\
&\Big\|(\partial_t, \partial_x,\partial_y)\mathcal{G}\Big(\partial_t\chi_n+\sum_{l\in Z^3}\frac{l}{\mu}\cdot\nabla\chi_{nl}\Big)\Big\|_0\leq C_0(\varepsilon)\kappa\mu\lambda_{n-1},\nonumber\\
&\Big\|\partial_z\mathcal{G}\Big(\partial_t\chi_n+\sum_{l\in Z^3}\frac{l}{\mu}\cdot\nabla\chi_{nl}\Big)\Big\|_0\leq C_0(\varepsilon)(\sqrt{\kappa}\mu)^n\bar{\Lambda}\ell_n^{-1}\lambda_n^{-1}.\nonumber
\end{align}
\end{Lemma}
The proof of the above two lemmas are similar to that of Lemma \ref{e:osci}, Lemma \ref{e:trans}  respectively,  we omit it here.
\begin{Lemma}[The error Part]
For $n=2,3$, we have
\begin{align}
\|w_{nc}\chi_n\|_0\leq C_0\sqrt{\kappa}\kappa\mu\lambda_{n-1}\lambda_n^{-1},&\quad \|(\partial_t, \partial_x,\partial_y)(w_{nc}\chi_n)\|_0\leq C_0\kappa(\sqrt{\kappa}\mu\lambda_{n-1}),\nonumber\\
\|\partial_z(w_{nc}\chi_n)\|_0\leq& C_0 \sqrt{\kappa}(\sqrt{\kappa}\mu)^n\bar{\Lambda}\ell_n^{-1}\lambda_n^{-1},\nonumber\\
\Big\|\sum\limits_{l\in Z^3}\Big(v_{0{n-1}\ell_n}-\frac{l}{\mu}\Big)\chi_{nl}+&\sum_{l\in Z^3}\big(v_{0(n-1)}-v_{0(n-1)\ell_n}\big)\chi_{nl}\nonumber\\
&+w_n(\theta_{0(n-1)}-\theta_{0(n-1)\ell_n})\Big\|_0
\leq C_0\sqrt{\kappa}\mu^{-1},\nonumber\\
\Big\|(\partial_t, \partial_x,\partial_y)\Big[\sum\limits_{l\in Z^3}\Big(v_{0(n-1)\ell_n}-\frac{l}{\mu}\Big)\chi_{nl}&+\sum_{l\in Z^3}\big(v_{0(n-1)}-v_{0(n-1)\ell_n}\big)\chi_{nl}\nonumber\\
&+w_n(\theta_{0(n-1)}-\theta_{0(n-1)\ell_n})\Big]\Big\|_0
\leq C_0\lambda_n\sqrt{\kappa}\mu^{-1},\nonumber\\
\Big\|\partial_z\Big[\sum\limits_{l\in Z^3}\Big(v_{0(n-1)\ell_n}-\frac{l}{\mu}\Big)\chi_{nl}&+\sum_{l\in Z^3}\Big(v_{0(n-1)}-v_{0(n-1)\ell_n}\Big)\chi_{nl}\nonumber\\
&+w_n(\theta_{0(n-1)}-\theta_{0(n-1)\ell_n})\Big]\Big\|_0
\leq C_0(\sqrt{\kappa}\mu)^n\bar{\Lambda}\mu^{-1}.
\end{align}
\end{Lemma}
\begin{proof}
First, by (\ref{e:n error estimate}) and (\ref{a:assumption on parameter n}), the first three estimates are obtained direct.

Since $w_{nl}\neq 0$ implies $|\mu n_{0(n-1)\ell_n}-l|<1$ and by (\ref{e:induction estimate sequence}), we obtain $$\|(\partial_t, \partial_x,\partial_y)v_{0(n-1)}\|_0\leq C_0\sqrt{\kappa}\lambda_{n-1},$$ $$\|(\partial_t, \partial_x,\partial_y)\theta_{0(n-1)}\|_0\leq C_0\sqrt{\kappa}\lambda_{n-1}, $$ $$\|\partial_z v_{0(n-1)}\|_0\leq C_0(\sqrt{\kappa}\mu)^{n-1}\bar{\Lambda}$$ $$\|\partial_z \theta_{0(n-1)}\|_0\leq C_0(\sqrt{\kappa}\mu)^{n-1}\bar{\Lambda}.$$ By (\ref{a:assumption on parameter n}), we arrive at
\begin{align}
\Big\|\sum\limits_{l\in Z^3}\Big(v_{0(n-1)\ell_n}-\frac{l}{\mu}\Big)\chi_{nl}&+\sum_{l\in Z^3}\big(v_{0(n-1)}-v_{0(n-1)\ell_n}\big)\chi_{nl}\nonumber\\
&+w_n(\theta_{0(n-1)}-\theta_{0(n-1)\ell_n})\Big\|_0
\leq C_0\sqrt{\kappa}\mu^{-1}.
\end{align}
Similarly
\begin{align}
\Big\|(\partial_t, \partial_x,\partial_y)\Big[\sum\limits_{l\in Z^2}\Big(v_{0(n-1)\ell_n}-\frac{l}{\mu}\Big)\chi_{nl}&+\sum_{l\in Z^3}\big(v_{0(n-1)}-v_{0(n-1)\ell_n}\big)\chi_{2l}\nonumber\\
&+w_n(\theta_{0(n-1)}-\theta_{0(n-1)\ell_n})\Big]\Big\|_0
\leq C_0\lambda_n\sqrt{\kappa}\mu^{-1}.\nonumber
\end{align}
Thus, we get the fourth and fifth estimates in this lemma.

By Lemma \ref{e:n main perturbation and corection estimate} and  assumption (\ref{a:assumption on parameter n}), we can obtain
\begin{align}
\Big\|\partial_z\Big[\sum\limits_{l\in Z^3}\Big(v_{0(n-1)\ell_n}-\frac{l}{\mu}\Big)\chi_{nl}\Big]\Big\|_0=&\Big\|\sum\limits_{l\in Z^3}(\partial_zv_{0(n-1)})_{\ell_n}\chi_{nl}\nonumber\\
&+\sum\limits_{l\in Z^3}\Big(v_{0(n-1)\ell_n}-\frac{l}{\mu}\Big)\partial_z\chi_{nl}\Big\|_0
\leq C_0(\sqrt{\kappa}\mu)^n\bar{\Lambda}\mu^{-1},\nonumber\\
\Big\|\partial_z\sum_{l\in Z^3}\big(v_{0(n-1)}-v_{0(n-1)\ell_n}\big)\chi_{nl}\Big\|_0=&\Big\|\sum_{l\in Z^3}\big(\partial_zv_{0(n-1)}-(\partial_zv_{0(n-1)})_{\ell_n}\big)\chi_{nl}\nonumber\\
&+\sum_{l\in Z^3}\big(v_{0(n-1)}-v_{0(n-1)\ell_n}\big)\partial_z\chi_{nl}\Big\|_0
\leq C_0(\sqrt{\kappa}\mu)^n\bar{\Lambda}\mu^{-1},\nonumber\\
\big\|\partial_z(w_n(\theta_{0(n-1)}-\theta_{0(n-1)\ell_n}))\big\|_0=&\|\partial_zw_n(\theta_{0(n-1)}-\theta_{0(n-1)\ell_n})\nonumber\\
&+w_n(\partial_z\theta_{0(n-1)}-(\partial_z\theta_{0(n-1)})_{\ell_n})\|_0
\leq C_0(\sqrt{\kappa}\mu)^n\bar{\Lambda}\mu^{-1}.\nonumber
\end{align}
Thus, collecting all those estimates, we arrive at
\begin{align}
\Big\|\partial_z\Big[\sum\limits_{l\in Z^3}\Big(v_{0(n-1)\ell_n}-\frac{l}{\mu}\Big)\chi_{nl}&+\sum_{l\in Z^3}\big(v_{0(n-1)}-v_{0(n-1)\ell_n}\big)\chi_{nl}\nonumber\\
&+w_n(\theta_{0(n-1)}-\theta_{0(n-1)\ell_n})\Big]\Big\|_0
\leq C_0(\sqrt{\kappa}\mu)^n\bar{\Lambda}\mu^{-1},\nonumber
\end{align}
which is the last estimate in this lemma.
\end{proof}
Finally, for $n=2,3$, we conclude
\begin{align}\label{e:step two temperature stress estimate}
\|\delta f_{0n}\|_0\leq & C_0(\varepsilon)\Big[\sqrt{\kappa}\mu^{-1}+\Big(\kappa\mu\lambda_{n-1}+
(\sqrt{\kappa}\mu)^n\bar{\Lambda}\ell_{nz}^{-1}\Big)\lambda_n^{-1}
\Big],\nonumber\\
\|(\partial_t, \partial_x,\partial_y)\delta f_{0n}\|_0\leq & C_0(\varepsilon)\lambda_n\Big[\sqrt{\kappa}\mu^{-1}+\Big(\kappa\mu\lambda_{n-1}+
(\sqrt{\kappa}\mu)^n\bar{\Lambda}\ell_{nz}^{-1}\Big)\lambda_n^{-1}\Big],\nonumber\\
\|\partial_z\delta f_{0n}\|_0\leq & C_0(\varepsilon)\Big[(\sqrt{\kappa}\mu)^n\bar{\Lambda}\mu^{-1}
+\Big((\sqrt{\kappa}\mu)^n\bar{\Lambda}(\ell_{nz}^{-2}+\ell_n^{-1})
+\kappa\lambda_{n-1}\ell_{nz}^{-1}\Big)\lambda_n^{-1}\Big].
\end{align}
Next, we deal with $\delta f_{0n}$ for $n=4,5,6$.
\begin{lemma}
For $n=4,5,6$, we have the following estimates,
\begin{align}
\|\delta f_{0n}\|_0\leq&  C_0(\varepsilon)\big(\kappa\lambda_3\lambda_n^{-1}
+\kappa\lambda_3\ell_n+\sqrt{\kappa}(\sqrt{\kappa}\mu)^3\bar{\Lambda}\ell_{nz}\big)\leq C_0(\varepsilon)\sqrt{\kappa}\mu^{-1},\nonumber\\
\|\partial_z\delta f_{0n}\|_0\leq & C_0(\varepsilon)\Big[(\sqrt{\kappa}\mu)^n\bar{\Lambda}\mu^{-1}
+\kappa\lambda_3\ell_{nz}^{-1}\lambda_n^{-1}\Big],\nonumber\\
\|(\partial_t, \partial_x,\partial_y)\delta f_{0n}\|_0\leq& C_0(\varepsilon)\lambda_n\big(\kappa\lambda_3\lambda_n^{-1}
+\kappa\lambda_3\ell_n+\sqrt{\kappa}(\sqrt{\kappa}\mu)^3\bar{\Lambda}\ell_{nz}\big)\leq C_0(\varepsilon)\lambda_n\sqrt{\kappa}\mu^{-1}.\nonumber\\
\end{align}
\end{lemma}
\begin{proof}
Recall that
$$\delta f_{0n}=\mathcal{G}(w_n\cdot\nabla\theta_{0(n-1)\ell_n})+w_n(\theta_{0(n-1)}-\theta_{0(n-1)\ell_n}).$$
By (\ref{r:another representation on wnl}), we have
\begin{align}
w_n\cdot\nabla\theta_{0(n-1)\ell_n}=&\sum_{l\in Z^3}\Big(g_{nl}\cdot\nabla\theta_{0(n-1)\ell_n} e^{i\lambda_n 2^{[l]} (k_{nh}^{\perp},0)\cdot \big((x,y,z)-\frac{l}{\mu}t\big)}\nonumber\\
&+g_{-nl}\cdot\nabla\theta_{0(n-1)\ell_n} e^{-i\lambda_n 2^{[l]} (k_{nh}^{\perp},0)\cdot \big((x,y,z)-\frac{l}{\mu}t\big)}\Big).\nonumber
\end{align}
By Lemma \ref{e:estimate various order n} and Corollary \ref{e:n sequence difference estiamte},  as in the proof of (\ref{e:esyimate terpeture 2}), we can obtain
\begin{align}
&\|\mathcal{G}(w_n\cdot\nabla\theta_{0(n-1)\ell_n})\|_0\leq C_0\kappa\lambda_3\lambda_n^{-1},\quad\|(\partial_t, \partial_x,\partial_y)\mathcal{G}(w_n\cdot\nabla\theta_{0(n-1)\ell_n})\|_0\leq C_0\kappa\lambda_3,\nonumber\\
 &\|\partial_z\mathcal{G}(w_n\cdot\nabla\theta_{0(n-1)\ell_n})\|_0\leq C_0\kappa\lambda_3\ell_{nz}^{-1}\lambda_n^{-1}.\nonumber
\end{align}
Moreover, by Lemma \ref{e:estimate various order n}, Corollary \ref{e:n sequence difference estiamte} and assumption (\ref{a:assumption on parameter n}), we have, for $n=4,5,6$
\begin{align}
\|w_n(\theta_{0(n-1)}-\theta_{0(n-1)\ell_n})\|_0\leq& C_0(\kappa\lambda_3\ell_n+\sqrt{\kappa}(\sqrt{\kappa}\mu)^3\bar{\Lambda}\ell_{nz}),\nonumber\\ \|(\partial_t,\partial_x,\partial_y)w_n(\theta_{0(n-1)}-\theta_{0(n-1)\ell_n})\|_0\leq & C_0\lambda_n(\kappa\lambda_3\ell_n+\sqrt{\kappa}(\sqrt{\kappa}\mu)^3\bar{\Lambda}\ell_{nz}),  \nonumber\\ \|\partial_z(w_n(\theta_{0(n-1)}-\theta_{0(n-1)\ell_n}))\|_0\leq&
C_0(\sqrt{\kappa}\mu)^n\bar{\Lambda}\mu^{-1}.\nonumber
\end{align}
Collecting all these estimates, we complete the proof of this lemma.
\end{proof}

Therefore, for $2\leq n\leq 6$, we have
\begin{align}\label{e:stress error n various estimate}
\|\delta f_{0n}\|_0\leq & C_0(\varepsilon)\Big[\sqrt{\kappa}\mu^{-1}+\Big(\kappa\mu\lambda_{n-1}+
(\sqrt{\kappa}\mu)^n\bar{\Lambda}\ell_{nz}^{-1}\Big)\lambda_n^{-1}\Big],\nonumber\\
\|(\partial_t, \partial_x,\partial_y)\delta f_{0n}\|_0\leq & C_0(\varepsilon)\lambda_n\Big[\sqrt{\kappa}\mu^{-1}+\Big(\kappa\mu\lambda_{n-1}+
(\sqrt{\kappa}\mu)^n\bar{\Lambda}\ell_{nz}^{-1}\Big)\lambda_n^{-1}\Big],\nonumber\\
\|\partial_z\delta f_{0n}\|_0\leq & C_0(\varepsilon)\Big[(\sqrt{\kappa}\mu)^n\bar{\Lambda}\mu^{-1}
+\Big((\sqrt{\kappa}\mu)^n\bar{\Lambda}(\ell_{nz}^{-2}+\ell_n^{-1})
+\kappa\lambda_{n-1}\ell_{nz}^{-1}\Big)\lambda_n^{-1}\Big].
\end{align}
By (\ref{e:n iterative stress error estimate}), (\ref{e:stress error n various estimate}), (\ref{e:induction estimate sequence}), (\ref{e:the first step conclusion}) and Corollary \ref{e:n sequence difference estiamte},  we can construct $$(v_{0n},~p_{0n},~\theta_{0n},\\~R_{0n},~f_{0n}),\quad n=1,...,6.$$
They solve system (\ref{d:anistropic boussinesq reynold}) and satisfy
\begin{align}\label{f:nform}
 R_{0n}=-\sum\limits_{i=n+1}^{6}(e(t)-a_{i\ell_n})k_i\otimes k_i
 +\sum\limits_{i=1}^{n}\delta R_{0i},\quad
  f_{0n}=\sum\limits_{i=n+1}^{3}c_{i\ell_1}k_i+\sum\limits_{i=1}^{n}\delta f_{0i}
\end{align}
and
 \begin{align}
 &\|v_{0n}-v_{0(n-1)}\|_0\leq \frac{M\sqrt{\kappa}}{12}+C_0\frac{\kappa\mu\lambda_{n-1}}{\lambda_n},\quad\|(\partial_t, \partial_x,\partial_y)(v_{0n}-v_{0(n-1)})\|_0\leq C_0\sqrt{\kappa}\lambda_n,\nonumber\\
 &\|\theta_{0n}-\theta_{0(n-1)}\|_0\leq \left\{
 \begin{array}{ll}
 \frac{M\sqrt{\kappa}}{6}+C_0\frac{\kappa\mu\lambda_{n-1}}{\lambda_n},\quad n=1,2,3,\\
 0,\qquad n=4,5,6
 \end{array}
 \right.\quad\|\partial_z(v_{0n}-v_{0(n-1)})\|_0\leq C_0(\sqrt{\kappa}\mu)^n\bar{\Lambda},\nonumber\\
 &\|(\partial_t, \partial_x,\partial_y)(\theta_{0n}-\theta_{0(n-1)})\|_0\leq \left\{
 \begin{array}{ll}
 C_0\sqrt{\kappa}\lambda_n,\quad n=1,2,3,\\
 0,\qquad n=4,5,6
 \end{array}
 \right.\nonumber\\
 &\|\partial_z(\theta_{0n}-\theta_{0(n-1)})\|_0\leq \left\{
 \begin{array}{ll}
 C_0(\sqrt{\kappa}\mu)^n\bar{\Lambda},\quad n=1,2,3,\\
 0,\qquad n=4,5,6
 \end{array}
 \right.\quad
   \|p_{0n}-p_{0(n-1)}\|_0=0,\nonumber\\
&\|\delta R_{0n}\|_0\leq  C_0(\varepsilon)\Big[\sqrt{\kappa}\mu^{-1}+\kappa\lambda_{n-1}\ell_n+\lambda_n^{-1}\Big(\kappa\mu\lambda_{n-1}+
(\sqrt{\kappa}\mu)^{n}\bar{\Lambda}\ell_{nz}^{-1}\Big)\Big],\nonumber\\
&\|(\partial_t, \partial_x,\partial_y)\delta R_{0n}\|_0\leq C_0(\varepsilon)\lambda_n\Big[\sqrt{\kappa}\mu^{-1}+\kappa\lambda_{n-1}\ell_n+\lambda_n^{-1}\Big(\kappa\mu\lambda_{n-1}+
(\sqrt{\kappa}\mu)^{n}\bar{\Lambda}\ell_{nz}^{-1}\Big)\Big],\nonumber\\
&\|\partial_z\delta R_{0n}\|_0\leq  C_0(\varepsilon)\Big[(\sqrt{\kappa}\mu)^n\bar{\Lambda}\big(\mu^{-1}+\sqrt{\kappa}\lambda_{n-1}\ell_n\big)+(\sqrt{\kappa}\mu)^{n}\bar{\Lambda}
\Big(\ell_n^{-1}+\ell_{nz}^{-2}\Big)\lambda_n^{-1}\Big],\nonumber\\
&\|\delta f_{0n}\|_0\leq C_0(\varepsilon)\Big[\sqrt{\kappa}\mu^{-1}+\kappa\lambda_{n-1}\ell_n+\Big(\kappa\mu\lambda_{n-1}+(\sqrt{\kappa}\mu)^n\Lambda\ell_{nz}^{-1}\Big)\lambda_n^{-1}
\Big],\nonumber\\
&\|(\partial_t, \partial_x,\partial_y)\delta f_{0n}\|_0\leq  C_0(\varepsilon)\lambda_n\Big[\sqrt{\kappa}\mu^{-1}+\kappa\lambda_{n-1}\ell_n+\Big(\kappa\mu\lambda_{n-1}+
(\sqrt{\kappa}\mu)^n\Lambda\ell_{nz}^{-1}\Big)\lambda_n^{-1}
\Big],\nonumber\\
&\|\partial_z\delta f_{0n}\|_0\leq  C_0(\varepsilon)\Big[(\sqrt{\kappa}\mu)^n\bar{\Lambda}\big(\mu^{-1}+\sqrt{\kappa}\lambda_{n-1}\ell_n\big)+\Big((\sqrt{\kappa}\mu)^n\bar{\Lambda}(\ell_{nz}^{-2}+\ell_n^{-1})
+\kappa\lambda_{n-1}\ell_{nz}^{-1}\Big)\lambda_n^{-1}\Big].\nonumber
\end{align}
In particular, we conclude that
\begin{align}\label{d:final R f}
 R_{06}=\sum\limits_{i=1}^{6}\delta R_{0i},\quad
 f_{06}=\sum\limits_{i=1}^{6}\delta f_{0i}.\nonumber
     \end{align}

\setcounter{equation}{0}

\section{Proof of proposition 2.1}
In this section, we collect all the estimates from the preceding sections. From these estimates , we can prove Proposition \ref{p: iterative 1} by choosing the appropriate parameters
$\ell_n,  \ell_{nz}, \mu, \lambda_n$ for $1\leq n\leq 6$.
\begin{proof}
In section 9, we have constructed functions $(v_{06},~p_{06},~\theta_{06},~R_{06},~f_{06})$, they solve system (\ref{d:anistropic boussinesq reynold}) and satisfy
\begin{align}
\|R_{06}\|_0\leq& C_0(\varepsilon) \sum\limits_{n=1}^{6}\Big[\kappa\lambda_{n-1}\ell_n+\sqrt{\kappa}\mu^{-1}
+\lambda_n^{-1}\Big(\kappa\mu\lambda_{n-1}+
(\sqrt{\kappa}\mu)^{n}\bar{\Lambda}\ell_{nz}^{-1}\Big)\Big],\nonumber\\
\|f_{06}\|_0\leq& C_0(\varepsilon) \sum\limits_{n=1}^{6}\Big[\kappa\lambda_{n-1}\ell_n+\sqrt{\kappa}\mu^{-1}+
\lambda_n^{-1}\Big(\kappa\mu\lambda_{n-1}+
(\sqrt{\kappa}\mu)^n\Lambda\ell_{nz}^{-1}\Big)\Big],\nonumber\\
\|v_{06}-&v_0\|_0\leq\frac{M\sqrt{\kappa}}{2}+C_0\sqrt{\kappa}\sum\limits_{n=2}^{6}\frac{\sqrt{\kappa}\mu\lambda_{n-1}}{\lambda_n}+
C_0\frac{\sqrt{\kappa}\mu\Lambda}{\lambda_1},\nonumber\\
\|\theta_{06}-&\theta_0\|_0\leq\frac{M\sqrt{\kappa}}{2}+C_0\sqrt{\kappa}\sum\limits_{n=2}^{3}\frac{\sqrt{\kappa}\mu\lambda_{n-1}}{\lambda_n}
+C_0\frac{\sqrt{\kappa}\mu\Lambda}{\lambda_1},\quad
\quad  \|p_{06}-p_0\|_0=0
\end{align}
and
\begin{align}
\|R_{06}\|_{C^1}\leq &C_0(\varepsilon) \sum\limits_{n=1}^{6}\lambda_n\Big[\kappa\lambda_{n-1}\ell_n+\sqrt{\kappa}\mu^{-1}
+\lambda_n^{-1}\Big(\kappa\mu\lambda_{n-1}+
(\sqrt{\kappa}\mu)^{n}\bar{\Lambda}\ell_{nz}^{-1}\Big)\Big],\nonumber\\
\|R_{06}\|_{C^1_z}\leq& C_0(\varepsilon)\sum_{n=1}^6\Big[(\sqrt{\kappa}\mu)^n\bar{\Lambda}\Big(\mu^{-1}
+\sqrt{\kappa}\lambda_{n-1}\ell_n\Big)+\lambda_n^{-1}(\sqrt{\kappa}\mu)^{n}\bar{\Lambda}
\Big(\ell_n^{-1}+
\ell_{nz}^{-2}\Big)\Big],\nonumber\\
\|f_{06}\|_{C^1}\leq &C_0(\varepsilon) \sum\limits_{n=1}^{6}\lambda_n\Big[\kappa\lambda_{n-1}\ell_n+\sqrt{\kappa}\mu^{-1}
+\lambda_n^{-1}\Big(\kappa\mu\lambda_{n-1}+
(\sqrt{\kappa}\mu)^{n}\bar{\Lambda}\ell_{nz}^{-1}\Big)\Big],\nonumber\\
\|f_{06}\|_{C^1_z}\leq& C_0(\varepsilon)\sum_{n=1}^6\Big[(\sqrt{\kappa}\mu)^n\bar{\Lambda}\Big(\mu^{-1}+\sqrt{\kappa}\lambda_{n-1}\ell_n
\Big)+\lambda_n^{-1}\Big((\sqrt{\kappa}\mu)^n\bar{\Lambda}(\ell_{nz}^{-2}+\ell_n^{-1})
+\kappa\lambda_{n-1}\ell_{nz}^{-1}\Big)\Big],\nonumber\\
\|v_{06}&-v_0\|_{C^1}\leq C_0\sqrt{\kappa}\sum\limits_{n=1}^{6}\lambda_n,\quad \quad
\|\theta_{06}-\theta_0\|_{C^1}\leq C_0\sqrt{\kappa}\sum\limits_{n=1}^{3}\lambda_n,\nonumber\\
\|\partial_z(v_{06}&-v_0)\|_0\leq C_0\sum\limits_{n=1}^{6}(\sqrt{\kappa}\mu)^{n}\bar{\Lambda},\quad \quad
\|\partial_z(\theta_{06}-\theta_0)\|_0\leq C_0\sum\limits_{n=1}^{3}(\sqrt{\kappa}\mu)^{n}\bar{\Lambda}.
\end{align}
where $\lambda_0:=\Lambda \kappa^{-1}+\bar{\Lambda}\ell_{1z}\kappa^{-1}\ell_1^{-1}$.
We divide the remainder proof into four steps:\\
{\bf Step 1}.
We now specify the choice of the parameters. First we choose
\begin{align}\label{e:first choose of parameters}
\mu=\frac{L_v\sqrt{\kappa}}{\bar{\kappa}},\quad \ell_1=\frac{\bar{\kappa}}{L_v\Lambda},\quad \ell_{1z}=\frac{\bar{\kappa}}{L_v\bar{\Lambda}},
\quad\lambda_1=D_v\frac{\sqrt{\kappa}\mu\Lambda^{1+\varepsilon}+
\mu^2\bar{\Lambda}^2}{\bar{\kappa}}
\end{align}
where $L_v$ is a sufficiently large constant which depends only on $\|v\|_0$.

Next, we impose
\begin{align}\label{p:parameter n 23}
\ell_n=\frac{1}{L_v}\frac{\bar{\kappa}}{\kappa\lambda_{n-1}},\quad\ell_{nz}=
\frac{1}{L_v}\frac{\bar{\kappa}}{\sqrt{\kappa}(\sqrt{\kappa}\mu)^{n-1}\bar{\Lambda}},\quad
\lambda_n=D_v
\frac{\kappa\mu\lambda^{1+\varepsilon}_{n-1}+(\sqrt{\kappa}\mu)^{n}\bar{\Lambda}\ell^{-1}_{nz}}{\bar{\kappa}},\quad n=2,...,6.
\end{align}
where $D_v\geq L_v^{2}$ is a sufficiently large constant which depends only on $\|v\|_0$.

{\bf Step 2. Compatibility condition.} We check that all the conditions in (\ref{a:assumption on parameter}), (\ref{a:assumption on parameter m sequence}) are satisfied by our choice of the parameters.

We first check the first triple $(\ell_1, \mu, \lambda_1)$.  By (\ref{e:first choose of parameters}), it's easy to see $$\ell_1^{-1}=L_v\frac{\Lambda}{\bar{\kappa}}\geq\frac{L_v\sqrt{\kappa}}{\bar{\kappa}}\Lambda\geq\mu\Lambda,\quad \ell_{1z}^{-1}=L_v\frac{\bar{\Lambda}}{\bar{\kappa}}\geq\frac{L_v\sqrt{\kappa}}{\bar{\kappa}}\bar{\Lambda}\geq\mu\bar{\Lambda}.$$
Since $\bar{\kappa}\leq\kappa^{\frac{3}{2}}$, by (\ref{e:first choose of parameters})
 $$\mu\geq \frac{1}{\kappa}.$$
From $D_v\geq L_v^{1+\varepsilon}$, we arrive at
$$\lambda_1\geq \Big(L_v\frac{\Lambda}{\bar{\kappa}}\Big)^{1+\varepsilon}\geq \ell_1^{-(1+\varepsilon)}.$$
Thus, (\ref{a:assumption on parameter}) is satisfied.

Next, for $n=2,...,6$,  obviously
$$\ell_n^{-1}=L_v\frac{\kappa\lambda_{n-1}}{\bar{\kappa}}=\sqrt{\kappa}\frac{L_v\sqrt{\kappa}}{\bar{\kappa}}\lambda_{n-1}=\sqrt{\kappa}\mu\lambda_{n-1},\quad
\ell_{nz}^{-1}=\frac{L_v\sqrt{\kappa}(\sqrt{\kappa}\mu)^{n-1}\bar{\Lambda}}{\bar{\kappa}}=\mu(\sqrt{\kappa}\mu)^{n-1}\bar{\Lambda}. $$
By (\ref{p:parameter n 23}), it's easy to obtain
$$\lambda_n\geq\ell_n^{-(1+\varepsilon)}.$$
Thus, (\ref{a:assumption on parameter m sequence}) is satisfied.

From (\ref{p:parameter n 23}) and $\varepsilon$ small, a straightforward computation yields
\begin{align}\label{v:value of frequency parameter}
\lambda_3\leq& C_0D_v^{3+4\varepsilon}\Big(\kappa^{\frac{2+\varepsilon}{2}}
 \Big(\frac{\sqrt{\kappa}\mu}{\bar{\kappa}}\Big)^{3+4\varepsilon}\Lambda^{(1+\varepsilon)^3}+\frac{\kappa^{2+\varepsilon}\mu^{4+6\varepsilon}\bar{\Lambda}^{2(1+\varepsilon)^2}}
{\bar{\kappa}^{3+4\varepsilon}}\Big),\nonumber\\
\lambda_6\leq &C_0D_v^7\Big(\kappa^{\frac{5+11\varepsilon}{2}}
 \Big(\frac{\sqrt{\kappa}\mu}{\bar{\kappa}}\Big)^{6+16\varepsilon}\Lambda^{(1+\varepsilon)^6}
 +\frac{\kappa^{5+11\varepsilon}\mu^{7+21\varepsilon}\bar{\Lambda}^{2(1+\varepsilon)^5}}
{\bar{\kappa}^{6+16\varepsilon}}\Big).\nonumber\\
\end{align}
{\bf Step 3. $C^0$ estimates.} Fixed $\varepsilon$ small. Thus, (10.1) implies
\begin{align}
\|R_{06}\|_0\leq C_v \Big(\bar{\kappa}L_v^{-1}+\bar{\kappa}D_v^{-1}\Big),\quad&
\|f_{06}\|_0\leq C_v \Big(\bar{\kappa}L_v^{-1}+\bar{\kappa}D_v^{-1}\Big),\nonumber\\
\|v_{06}-v_0\|_0\leq \frac{M\sqrt{\kappa}}{2}+\frac{\sqrt{\kappa}}{D_v},\quad
\|\theta_{06}-&\theta_0\|_0\leq \frac{M\sqrt{\kappa}}{2}+\frac{\sqrt{\kappa}}{D_v},\quad
\|p_{06}-p_0\|_0=0.\nonumber
\end{align}
Choosing first $L_v$ and then, $D_v$ sufficiently large, we can achieve the desired inequalities
(\ref{e:next step strees estimate})-(\ref{e:next step pressure estimate}).\\
{\bf Step 4. $C^1$ estimates.}
By the specified choices of parameters, we have
\begin{align}
\sqrt{\kappa}\lambda_{n-1}\ell_n\leq\mu^{-1},\quad
\lambda_n^{-1}(\sqrt{\kappa}\mu)^{n}\bar{\Lambda}
\Big(\ell_n^{-1}+\ell_{nz}^{-2}\Big)
\leq \kappa^{\frac{3}{2}}\mu^2\bar{\Lambda},\quad n=1,...,6.\nonumber
\end{align}
Hence
\begin{align}
&\|R_{06}\|_{C^1}\leq \lambda_6\bar{\kappa},\quad \|f_{06}\|_{C^1}\leq \lambda_6\bar{\kappa},\quad
\|v_{06}-v_0\|_{C^1}\leq C_0\sqrt{\kappa}\lambda_6,\quad\|\theta_{06}-\theta_0\|_{C^1}\leq C_0\sqrt{\kappa}\lambda_3,\nonumber\\
&\|R_{06}\|_{C_z^1}\leq C_0\kappa^3\mu^5\bar{\Lambda},\quad \|f_{06}\|_{C_z^1}\leq  C_0\kappa^{3}\mu^5\bar{\Lambda},\quad
\|v_{06}-v_0\|_{C_z^1}\leq  C_0(\sqrt{\kappa}\mu)^6\bar{\Lambda},\nonumber\\
&\|\theta_{06}-\theta_0\|_{C_z^1}\leq  C_0(\sqrt{\kappa}\mu)^3\bar{\Lambda}.\nonumber
\end{align}
Since $\varepsilon$ is small, we conclude
\begin{align}
\Lambda_1:=&\max\{1,\|R_{06}\|_{C^1}, \|f_{06}\|_{C^1}, \|v_{06}\|_{C^1}, \|\theta_{06}\|_{C^1}\}\leq \Lambda+C_0\sqrt{\kappa}\lambda_6\nonumber\\
\leq& C_0D_v^7\Big(\kappa^{\frac{6+11\varepsilon}{2}}
 \Big(\frac{\sqrt{\kappa}\mu}{\bar{\kappa}}\Big)^{6+16\varepsilon}\Lambda^{(1+\varepsilon)^6}+\frac{\kappa^{5.5+11\varepsilon}\mu^{7+21\varepsilon}\bar{\Lambda}^{2(1+\varepsilon)^5}}
{\bar{\kappa}^{6+16\varepsilon}}\Big),\nonumber\\
\bar{\Lambda}_1:=&\max\{1,\|R_{06}\|_{C_z^1}, \|f_{06}\|_{C_z^1}, \|v_{06}\|_{C_z^1}, \|\theta_{06}\|_{C_z^1}\}
\leq \bar{\Lambda}+C_0(\sqrt{\kappa}\mu)^6\bar{\Lambda}
\leq C_0L_v^6\Big(\frac{\kappa}{\bar{\kappa}}\Big)^6\bar{\Lambda}.\nonumber
\end{align}
Setting $A= C_0D_v^{7}$, we conclude estimate (\ref{b:bound on first derivative}).

Moreover, we have
\begin{align}
\|\partial_z\theta_{06}\|_0\leq& \bar{\Lambda}+C_0(\sqrt{\kappa}\mu)^3\bar{\Lambda}\leq A\Big(\frac{\kappa}{\bar{\kappa}}\Big)^3\bar{\Lambda},\nonumber\\
\|(\partial_t,\partial_x,\partial_y)\theta_{06}\|_0\leq& \Lambda+C_0\sqrt{\kappa}\lambda_3
\leq A\Big(\kappa^{\frac{3+\varepsilon}{2}}
 \Big(\frac{\kappa}{\bar{\kappa}^2}\Big)^{3+4\varepsilon}\Lambda^{(1+\varepsilon)^3}
 +\frac{\kappa^{\frac{5}{2}+\varepsilon}\bar{\Lambda}^{2(1+\varepsilon)^2}}
{\bar{\kappa}^{3+4\varepsilon}}\Big(\frac{\sqrt{\kappa}}{\bar{\kappa}}\Big)^{4+6\varepsilon}\Big),\nonumber
\end{align}
thus, we arrive at (\ref{e:growth estimate for temperture}).
Finally, we set
$$\tilde{v}:=v_{06},~\tilde{p}:=p_{06},~\tilde{\theta}:=\theta_{06},~\tilde{R}:=R_{06},~\tilde{f}:=f_{06}.$$
then $(\tilde{v},~\tilde{p},~\tilde{\theta},~\tilde{R},~\tilde{f})$ are what in our Proposition (\ref{p: iterative 1}).
\end{proof}

{\bf Acknowledgments.}
The authors are grateful to the referees for the invaluable comments and suggestions, which have helped us improve the paper significantly.
The authors are deeply grateful to Zhouping Xin and Tianwen Luo for very valuable discussions. The research is partially supported by the Chinese NSF under grant 11471320 and 11631008.

\end{document}